\documentclass[11pt, a4paper]{article}

\usepackage[utf8]{inputenc}
\usepackage[english]{babel}
\usepackage[T1]{fontenc}
\usepackage[colorlinks = true,
            linkcolor = blue,
            urlcolor  = blue,
            citecolor = blue,
            anchorcolor = blue]{hyperref}
\usepackage[a4paper,width=150mm,top=25mm,bottom=25mm]{geometry}
\usepackage{fancyhdr}

\usepackage{amsmath,amsthm,amssymb, enumitem, mathtools, mathrsfs}
\usepackage{siunitx}
\setitemize{noitemsep}

\usepackage{tikz}
\usepackage{tikz-cd}
\usepackage{caption}
\usepackage[all]{xy}
\usepackage{graphicx}
\usepackage{chngcntr}
\usepackage{booktabs}
\usepackage[labelformat=simple]{subcaption}

\usepackage{pdfpages}

\usepackage{csquotes}
\usepackage[
backend=biber,
maxbibnames=4,
sorting=anyt
]{biblatex}
\graphicspath{ {images/} }

\theoremstyle{plain}
\newtheorem{theorem}{Theorem}[section]
\newtheorem*{theorem*}{Theorem}
\newtheorem*{satz*}{Satz}
\newtheorem{proposition}[theorem]{Proposition}
\newtheorem{lemma}[theorem]{Lemma}
\newtheorem{corollary}[theorem]{Corollary}
\theoremstyle{definition}
\newtheorem{definition}[theorem]{Definition}
\newtheorem{assumption}[theorem]{Assumption}

\theoremstyle{remark}
\newtheorem{remark}[theorem]{Remark}

\numberwithin{equation}{section}

\DeclareMathOperator{\id}{id}

\newcommand{\ev}{\mathbb{E}}
\newcommand{\pr}{\mathbb{P}}

\newcommand{\R}{\mathbb{R}}
\renewcommand{\P}{\mathcal{P}}
\newcommand{\F}{\mathcal{F}}
\renewcommand{\L}{\mathcal{L}}
\newcommand{\M}{\mathcal{M}}
\renewcommand{\d}{\mathrm{d}}

\newcommand{\define}{\mathpunct{:}}
\newcommand{\bb}[1]{\mathbb{#1}}
\renewcommand{\bf}[1]{\mathbf{#1}}
\renewcommand{\cal}[1]{\mathcal{#1}}

\begin{document}

\noindent
\begin{center}
    \Large
    \textbf{Control of McKean--Vlasov SDEs with Contagion Through Killing at a State-Dependent Intensity}

    \vspace{1em}

    \normalsize
    Philipp Jettkant\footnote[1]{Department of Mathematics, Imperial College London, UK, \href{mailto:p.jettkant@imperial.ac.uk}{p.jettkant@imperial.ac.uk}.}  \& Ben Hambly\footnote[2]{Mathematical Institute, University of Oxford, UK, \href{mailto:ben.hambly@maths.ox.ac.uk}{ben.hambly@maths.ox.ac.uk}.}

\end{center}

\vspace{1em}

\begin{abstract}
We consider a novel McKean--Vlasov control problem with contagion through killing of particles and common noise. Each particle is killed at an exponential rate according to an intensity process that increases whenever the particle is located in a specific region. The removal of a particle pushes others towards the removal region, which can trigger cascades that see particles exiting the system in rapid succession. We study the control of such a system by a central agent who intends to preserve particles at minimal cost. Our theoretical contribution is twofold. Firstly, we rigorously justify the McKean--Vlasov control problem as the limit of a corresponding sequences of controlled finite particle systems. Our proof is based on a controlled martingale problem and tightness arguments. Secondly, we connect our framework with models in which particles are killed once they hit the boundary of the removal region. We show that these models appear in the limit as the exponential rate tends to infinity. As a corollary, we obtain new existence results for McKean--Vlasov SDEs with singular interaction through hitting times which extend those in the established literature. We conclude the paper with numerical investigations of our model applied to government control of systemic risk in financial systems.
\end{abstract}

\section{Introduction}

McKean--Vlasov or mean-field control (MFC) problems arise naturally as the infinite population limit of finite systems of controlled particles interacting through their empirical distribution. The dynamics of a representative particle in the infinite population limit are described by a McKean--Vlasov stochastic differential equation (SDE). Along with its close companion mean-field games, introduced independently by Lasry and Lions \cite{lasry_mfg_2007} and Caines, Huang, and Malham\'e \cite{caines_mfg_2006}, McKean--Vlasov control has received widespread attention in recent years \cite{anderson_maximum_principle_2011, lauriere_dp_mfc_2014, carmona_fbsde_smp_2015, pham_dynamic_programming_mkv_2017, djete_mvoc_dpp_2022, cardaliaguet_rate_mfc_2023}.

In this paper, we study the optimal control of a McKean--Vlasov SDE that features contagion through killing of particles in the presence of a common noise $W^0$. The state of the representative particle is described by a diffusion process $X = (X_t)_{t \geq 0}$ on the real line with McKean--Vlasov dynamics. The particle has a cumulative intensity process $\Lambda = (\Lambda_t)_{t \geq 0}$ that increases according to a state-dependent intensity $\lambda(X_t)$. Once the accumulated intensity $\Lambda_t$ exceeds a critical threshold, modelled by a standard exponential random variable $\theta$, the particle is removed from the system. Hence, the underlying McKean--Vlasov SDE does not depend on the conditional law $\L(X_t \vert W^0)$ of all particles but rather on the conditional subprobability distribution $\nu_t = \pr(X_t \in \cdot,\, \theta > \Lambda_t \vert W^0)$ of the remaining particles. Furthermore, the removal of a particle feeds back into the system by pushing the remaining particles instantaneously towards regions of high killing intensity. We refer to this mechanism as contagion. The force of the contagion is proportional to the fraction of killed particles $L_t = \pr(\theta > \Lambda_t \vert W^0)$, called the loss.

Our goal is to connect this McKean--Vlasov control problem with the control of a corresponding finite particle system. In the uncontrolled setting without killing or contagion, it is well-known that as the particle system's population size $N$ tends to infinity, the particles become independent and their empirical measure converges to the law of the solution of a McKean--Vlasov SDE -- the so-called mean-field limit. This phenomenon is known as propagation of chaos \cite{kac_kinetic_theory_1956, mckean_prop_of_chaos_1969, sznitman_prop_chaos_1991}. For controlled systems, however, the optimisation step makes establishing the convergence to the McKean--Vlasov control problem more challenging. A general approach to resolve this issue uses a relaxed formulation for the control problem and an associated martingale problem. This technique was applied in \cite{lacker_limit_theory_mve} to study a broad class of McKean--Vlasov control problems and extended in \cite{mkv_control_limit_2020} to systems with common noise.

The existing literature, however, does not cover our model because of the presence of the killing and contagion mechanisms. On the level of the particle system, killing leads to a sudden loss of mass in the empirical measure of the remaining particles, while the contagion mechanism introduces jumps into the particles' trajectories. These innovations bring about technical challenges that we need to address in this work. The first is that the control $\gamma^{N, i} = (\gamma^{N, i}_t)_{t \geq 0}$ acting on particle $i \in \{1, \dots, N\}$ generally depends on the associated critical threshold $\theta_i \sim \text{Exp}(1)$, because, naturally, the controller takes the killing of particles into account. To ensure that the controller cannot anticipate the killing events, this dependence occurs only through the information revealed by the indicators
\begin{equation*}
    I^{N, i}_t = \bf{1}_{\theta_i > \Lambda^{N, i}_t},
\end{equation*}
where $\Lambda^{N, i}_t$ is the cumulative intensity of the $i$th particle. To benefit from the regularisation of the mean-field limit afforded by the exponential random variable $\theta$, we must ensure that this information structure survives the passage to the infinite particle limit. We achieve this by studying the compensators of the processes $I^{N, i}_t$ and their convergence as $N \to \infty$. While this guarantees that the loss $L_t$ becomes smooth in the mean-field limit, the state-intensity pair $(X_t, \Lambda_t)$ will generally still depend on $\theta$. As a consequence, $\theta$ does not fully smooth out the discontinuous dependence of the subprobability distribution $\nu_t = \pr(X_t \in \cdot,\, \theta > \Lambda_t \vert W^0)$ on $\Lambda_t$, so the McKean--Vlasov SDE describing the dynamics of the representative particle proves difficult to analyse. To deal with this issue, we establish the equivalence between optimisation over controls with and without dependence on $\theta$ in the MFC problem, under the assumption that the control problem has a linear-convex structure. We show that from any given control, possibly depending on $\theta$, we can construct a new control, which is independent of $\theta$ and achieves the same or a lower cost. The latter control can be chosen to be in closed-loop form $g(t, X_t, \nu_t)$. Interestingly, the feedback function $g$ does not depend on the intensity $\Lambda_t$ or the indicator $I_t = \bf{1}_{\theta > \Lambda_t}$. This perhaps surprising result is a consequence of the memoryless property of the exponential distribution of $\theta$. The independence from $\theta$, makes it possible to rewrite $\nu_t$ in the form $\ev[e^{-\Lambda_t} \delta_{X_t} \vert W^0]$, leading to a well-posed McKean--Vlasov SDE for the mean-field limit. A related equivalence result was recently obtained in the context of the control of conditional processes through different means by Carmona, Lauri\`ere \& Lions \cite{carmona_nonstandard_2024}.

\subsection{Related Literature}

The killing mechanism in this article appeared in earlier work on mean-field models of loss from default in large portfolios \cite{dai_pra_large_portfolio_losses_2009, cvitanic_lln_self_2012, giesecke_default_typical_2013, giesecke_large_portfolio_2015}. While in the first two references the intensity depends on an underlying state process (as is the case in our model), the latter two articles follow a reduced-form approach and directly model the instantaneous default intensity $\lambda_t$ of constituents in a portfolio.
In both cases, however, killing occurs when the cumulative intensity $\Lambda_t$ exceeds an exponential random variable and feeds back into the remaining system. \cite{giesecke_default_typical_2013} proves propagation of chaos for the system of intensity processes by proving that their empirical measure converges to the solution of a nonlinear partial differential equation which describes the evolution of the mean-field limit. \cite{giesecke_large_portfolio_2015} extends the result to the common noise setting. Subsequent works study the central limit theorem \cite{spiliopoulos_fluctuation_2014}, large deviations \cite{spiliopoulos_ld_2015}, and parameter estimation \cite{giesecke_interence_large_systems_2020} for the reduced-form mean-field model. The arguments in \cite{giesecke_default_typical_2013, giesecke_large_portfolio_2015} no longer apply once controls enter the picture, so we cannot draw on their ideas in this paper. 

The removal according to a state-dependent intensity $\lambda(X_t)$ can at least heuristically be viewed as a regularisation of models with absorption. In these models, a member of the population is killed (or absorbed) once its state $X_t$ hits the negative half-line. Formally, this corresponds to the choice $\lambda(x) = \infty \bf{1}_{(-\infty, 0)}(x)$. McKean--Vlasov SDEs with absorption have been studied by various authors and we can distinguish between two strands within the literature. The first considers ``smooth'' interaction through hitting times, i.e.\@ the feedback through the loss $L_t$, which measures the number of particles killed up to time $t$, may appear in the drift and possibly the diffusion coefficient of the state equation \cite{hambly_mckean_vlasov_absorbing_2017, hambly_spde_model_2019, campi_mfg_absorption_2018, campi_mfg_hitting_2021, burzoni_mean_field_absorption_2023, hambly_smooth_feedback_2023}. In the ``singular'' variant the loss $L_t$ acts instantaneously, as it does in our model, and appears as a finite variation term in the dynamics of the representative particle \cite{nadtochiy_singular_hitting_2019, nadtochiy_hitting_network_2020, hambly_mckean_vlasov_blow_up_2019, cuchiero_minimality_2023, blow_ups_common_noise_ledger_2021, delarue_supercooled_uniqueness_2022, bayraktar_mve_hitting_2024}. In a system with absorption, the population's response as a function of the control is merely continuous if the interaction through the hitting times is smooth and can even exhibit discontinuities in the singular setting. Thus, classical methods of optimal control such as the stochastic maximum principle and the value function approach are not applicable. This shortcoming is our main motivation for studying the regularised model and explains why there are few publications \cite{cuchiero_optimal_bailout_2023, bayraktar_mf_sys_risk_2023} that study the control of absorbing McKean-Vlasov systems. Nonetheless, we are able to connect both frameworks by showing that the regularised model with intensity-based killing converges in a suitable sense to the absorption model with singular interaction as the intensity function $\lambda$ approaches $\infty\mathbf{1}_{(-\infty, 0)}$ (cf.\@ Theorem \ref{thm:convergence_struc}). As a corollary, we obtain a novel existence result for McKean--Vlasov SDEs with singular interaction through hitting times in the presence of common noise which extends those in the established literature \cite{hambly_mckean_vlasov_absorbing_2017, hambly_spde_model_2019, hambly_mckean_vlasov_blow_up_2019, blow_ups_common_noise_ledger_2021}.

The paper \cite{cuchiero_optimal_bailout_2023}, mentioned above, analyses the optimal control of absorbing dynamics with singular interaction, interpreted as a model for government bailouts (similar to the financial model we describe in Subsection \ref{sec:toy_model}). Leveraging the specific structure of their chosen cost functional, they deduce the convergence of the corresponding controlled particle system to the McKean--Vlasov control problem through tightness arguments. Their proof technique does not extend to models with nonlinear dependence on the control, so is not applicable to our work. As a consequence their framework is less flexible than ours and, in particular, does not include common noise. For their numerical computations, they rely on an alternative regularisation procedure, which traps particles at zero with high probability instead of absorbing them. Their numerical results are broadly in line with ours and show that the optimal control is of bang-bang type, meaning that the central agent only intervenes if a particle's state nears the killing region.

Yet another regularisation technique, first introduced in \cite{hambly_spde_model_2019}, was recently analysed in \cite{hambly_smooth_feedback_2023} from the convergence angle. The regularisation proceeds by mollification of the instantaneous action of the loss $L_t$, which leads to a model with smooth but path-dependent interaction through hitting times. \cite{hambly_smooth_feedback_2023} shows that one recovers the singular model in the limit as the mollification tends to zero. That paper does not touch on controlled versions of the problem.

Finally, let us mention the ``Up the River'' problem formulated by Aldous \cite{aldous_up_the_river_2002}, where a unit drift is distributed among a finite number of Brownian particles on the positive half-line which are absorbed at zero. The goal is to keep as many particles as possible alive for all times. \cite{tang_up_the_river_2018} shows that the asymptotically optimal strategy as $N \to \infty$ is to simply push the laggard, i.e.\@ to allocate the entire drift to the particle closest to the absorbing boundary. \cite{bayraktar_mf_sys_risk_2023} extends this framework by introducing singular interaction through hitting times amongst the particles.

\subsection{A Model for Government Interventions in Financial Systems} \label{sec:toy_model}

One application of our framework is in the modelling of systemic risk and government interventions in financial markets. In this context, the members of the finite particle system, labelled $i = 1$,~\ldots, $N$, represent financial institutions with mutual obligations or common exposures, for instance commercial banks in an interbank lending market or hedge funds investing in government bonds. In the case of commercial banks, the state $X^i_t$ measures an institution's level of equity at time $t$. Whenever the equity breaches a given threshold, the institution's default intensity becomes positive. Default occurs once the accumulated intensity $\Lambda^i_t$ exceeds the exponential random variable $\theta_i$. A defaulted entity does not pay back its obligations to its creditors in full, which diminishes the equity of the crediting institutions. This may trigger further defaults that wipe out more equity, and so on. In this way, a large number of financial institutions can go bankrupt in a short period of time with devastating consequences for the economy. To prevent the development of such default cascades a government or another central authority, such as a central bank, may intervene in the financial system, either by supporting market activity, e.g.\@ through quantitative easing, or by bailing out institutions outright. We can formulate such an intervention as an optimal control problem, in which the central agent weighs the negative externalities of market interference, such as moral hazard or higher inflation, with the cost of a potential systemic crisis. 

A simple model for government interventions in financial systems could read as follows: the processes $X^i = (X^i_t)_{t \geq 0}$, called \textit{distance-to-breach}, have the dynamics
\begin{equation*}
    X^i_t = (\xi_i - c) + \int_0^t \gamma^i_s \, \d s + \sigma W^i_t + \sigma_0 W^0_t - \alpha L^N_t, \quad \Lambda^i_t = \int_0^t \lambda(X^i_s) \, \d s
\end{equation*}
where $\xi_1$,~\ldots, $\xi_N$ are the banks' initial capital positions, $W^0$, $W^1$,~\ldots, $W^N$ are independent Brownian motions, $\gamma^1$,~\ldots, $\gamma^N$ are capital injections from the government, and $L^N_t = \frac{1}{N}\sum_{j = 1}^N \bf{1}_{\theta_i \leq \Lambda^i_t}$ denotes the fraction of banks that have defaulted up to time $t \geq 0$. We assume $\gamma^i_t \geq 0$, so the government can only provide capital to the banks but not withdraw it again. The diffusion $\sigma W^i_t + \sigma_0 W^0$ represents fluctuations in the value of assets held by the bank, which can be correlated among institutions through the common noise $W^0$, $\alpha$ measures the size of the banks' mutual obligations, and $c > 0$ is a capital buffer. If we choose the intensity function $\lambda$ to be $\lambda(x) = \lambda_0 \max\{-x, 0\}$ for a positive constant $\lambda_0$, then a bank is at risk of default once the equity breaches the buffer $c$. We capture the government's objective through the cost
\begin{equation} \label{eq:cost_toy_model}
    J^N(\gamma^1, \dots, \gamma^N) = \frac{1}{N} \sum_{i = 1}^N \ev\biggl[\int_0^{T \land \tau_i} w \gamma^i_t \, \d t + L^N_T\biggr],
\end{equation}
where $T > 0$ is some finite time horizon and $w$ weighs the negative externalities of a bailout with the cost of a systemic crisis as measured by the fraction of defaulted entities at time $T$. We provide numerical illustrations for this model in Section \ref{sec:num_sim}.

Let us briefly note that there are other applications of our framework, for instance it could also be used to model opinion dynamics in advertising or political campaigns. Here the particles are the target audience of the campaign and the state $X^i_t$ represents an audience member's level of satisfaction with a product or a political candidate. If people's satisfaction wanes, for instance because of a damaging news report, they abandon the product or candidate, which negatively impacts other people's views, who are in turn more likely to change their opinion. 
In response, the product originator or political candidate may attempt to shore up confidence by increasing their spending on advertisements to maintain people's satisfaction.

\subsection{Main Contributions and Structure of the Paper}

This paper makes three main contributions. Firstly, in Section \ref{sec:main_results_convergence} we introduce a new and flexible mean-field control model which features killing of particles and a contagion mechanism. Ours is the first work that considers these features in such a general setting and, in particular, in the presence of common noise. 

Secondly, in Section \ref{sec:ps} we prove the convergence of the controlled particle system for the regularised model to the corresponding McKean--Vlasov control problem. We achieve this by introducing a relaxed formulation of the McKean--Vlasov control problem (cf.\@ Subsection \ref{sec:main_results_ps}) and expressing the underlying McKean--Vlasov SDE for the state process as a controlled martingale problem. In Subsections \ref{sec:finite_ps} to \ref{sec:mgale_problem} we show that any controlled sequence of particle systems converges subsequentially to a solution of this martingale problem. As such, the limiting solution yields a control for the relaxed formulation. To establish the convergence we carefully analyse the compensators of the indicator processes $I^{N, 1}$,~\ldots, $I^{N, N}$ and prove their asymptotic decorrelation. 

For relaxed controls without dependence on $\theta$, for which the mean-field limit is well-posed, we can use the theory from Djete, Possama\"i, and Tan \cite{mkv_control_limit_2020} to show that the relaxed formulation is equivalent to the strong formulation of the control problem, see Subsections \ref{sec:existence_uniqueness} and \ref{sec:prove_ps}. This leaves a potential gap between the relaxed formulation for controls with and without dependence on $\theta$. In a linear-convex setting, we can bridge the gap by constructing a control without dependence on $\theta$ from any arbitrary relaxed control while maintaining or improving the associated cost. The construction, carried out in Subsection \ref{eq:proof_all_equiv}, exploits the mimicking result for stochastic Fokker--Planck equations by Lacker, Shkolnikov \& Zhang \cite{lacker_mimicking_2020} together with measurable selection arguments. It yields a control in closed-loop form that, surprisingly, only depends on the state $X_t$ and the subprobability $\nu_t$, but not the intensity $\Lambda_t$ or the indicator $I_t = \bf{1}_{\theta > \Lambda_t}$.

Our third contribution is to connect the regularised framework with the absorbing model in Section \ref{sec:sin_limit}. As for the regularised model, we adopt a relaxed formulation (cf.\@ Subsection \ref{sec:main_results_singular}) and introduce a suitable martingale problem. Then, in Subsections \ref{sec:reg_to_sing} to \ref{sec:proves_singular} we show that any sequence of regularised models with intensity functions $\lambda^n(x)$, that tend to $\infty \bf{1}_{(-\infty, 0)}(x)$ in a suitable sense, subsequentially converges to a solution of the martingale problem. This, in particular, implies a novel and very general existence theorem for McKean--Vlasov SDEs with singular interaction through hitting times in the presence of common noise, that extends existing results \cite{hambly_mckean_vlasov_absorbing_2017, hambly_spde_model_2019, hambly_mckean_vlasov_blow_up_2019, blow_ups_common_noise_ledger_2021}. However, in contrast to the regularised model, we cannot in general conclude that the relaxed control for the absorbing model, induced by the solution to the martingale problem, is optimal. This is due to the lack of uniqueness for McKean--Vlasov SDEs with singular interaction through hitting times \cite{nadtochiy_singular_hitting_2019, hambly_mckean_vlasov_blow_up_2019, blow_ups_common_noise_ledger_2021}. Only if we revert to a simpler setting, in which the underlying McKean--Vlasov SDE has constant coefficients and the cost functions satisfy certain monotonicity properties, can we deduce that the limiting control is indeed optimal, see Section \ref{sec:sing_simple}.

In Section \ref{sec:num_sim} we provide numerical simulations for the McKean--Vlasov control problem arising from our model for government interventions in financial systems.

\section{Main Results} \label{sec:main_results_convergence}

In this section, we state and explain the main results of this paper. We begin by listing some frequently used notation. We write $x \land y$ and $x \lor y$ for the minimum and maximum of two real numbers $x$ and $y$. We use $\F_1 \lor \F_2$ (or $\bb{F}^1 \lor \bb{F}^2$) to designate the $\sigma$-algebra (or filtration) generated by the two $\sigma$-algebras $\F_1$ and $\F_2$ (or filtrations $\bb{F}^1$ or $\bb{F}^2$). For a random variable $\xi$ and a stochastic process $X = (X_t)_{0 \leq t \leq T}$, we denote by $\bb{F}^{\xi, X} = (\F^{\xi, X}_t)_{0 \leq t \leq T}$ the filtration defined by $\F^{\xi, X}_t = \sigma(\xi, X_s \define 0 \leq s \leq t)$. The space of c\`adl\`ag functions from $[0, T]$ to a Polish space $E$ is denoted by $D_E[0, T]$ and for an element $f \in D_{\R^d}[0, T]$, we set $\lvert f\rvert^{\ast}_t = \sup_{0 \leq s \leq t} \lvert f_s\rvert$ for $t \in [0, T]$. If $(E, d)$ is a metric space and $p \geq 1$, we let $\P^p(E)$ be the space of probability measures $\mu$ on $E$ such that $\int_E d^p(e, x) \, \d \mu(x) < \infty$ for some $e \in E$. The space $\P^p(E)$ is endowed with the $p$-Wasserstein metric. Note that this definition is independent of the choice of $e \in E$.

\subsection{Convergence of the Nearly Optimally Controlled Particle System} \label{sec:main_results_ps}

We introduce a general particle system for our contagion model. As before the state space of the particle will be one-dimensional (excluding the cumulative intensity). This is sufficient to highlight the novel features of the model and is in line with the majority of the literature on McKean--Vlasov SDEs with interaction through hitting times, which inspired our work. Most results that concern the convergence of the nearly optimally controlled particle system to the regularised model generalise straightforwardly to higher dimensions. 

In what follows, we assume that the processes $X^i$ and $\Lambda^i$, $i = 1$~\ldots, $N$, follow the dynamics
\begin{align} \label{eq:ps}
\begin{split}
    \d X^i_t &= b(t, X^i_t, \nu^N_t, \gamma^i_t) \, \d t + \sigma(t, X^i_t, \nu^N_t) \, \d W^i_t + \sigma_0(t, X^i_t, \nu^N_t) \, \d W^0_t \\
    &\ \ \ - \alpha(t, X^i_{t-}, \nu^N_{t-}) \, \d L^N_t, \\
    \d \Lambda^i_t &= \lambda(t, X^i_t, \nu^N_t) \, \d t
\end{split}
\end{align}
with initial conditions $X^i_0 = \xi_i \sim \nu_0$ and $\Lambda^i_0 = 0$. The driving noises $W^0$, $W^1$,~\ldots, $W^N$ are Brownian motions and we introduce the \textit{empirical subprobability distribution}
\begin{equation}
    \nu^N_t = \frac{1}{N} \sum_{i = 1}^N I^i_t \delta_{X^i_t}
\end{equation}
as well as the \textit{loss} $L^N_t = 1 - \nu^N_t(\R)$, where $I^i_t = \bf{1}_{\theta_i > \Lambda^i_t}$ and $\theta_1$,~\ldots, $\theta_N$ each follow a standard exponential distribution. The indicator process $I^i$ models the killing of particles, which are triggered once $\Lambda^i_t \geq \theta_i$.

We assume that the collection $((\xi_i, \theta_i, W^i)_i, W^0)$ is independent and introduce the filtration $\bb{F}^N =(\F^N_t)_{0 \leq t \leq T}$ defined by
\begin{equation*}
    \F^N_t = \sigma\bigl(\xi_i,\, \theta_i,\, W^i_s,\, W^0_s \define 0 \leq s \leq t,\, i = 0, \dots, N\bigr).
\end{equation*}
The processes $\gamma^1$,~\ldots, $\gamma^N$ are \textit{controls}, which take values in a non-empty, convex, and compact subset $G$ of $\R^d$. Characterising the set of admissible controls is slightly subtle. First, let us note that if $\gamma^1$,~\ldots, $\gamma^N$ are $\bb{F}^N$-progressively measurable, then there exists a unique strong solution to SDE \eqref{eq:ps} (see the discussion at the outset of Section \ref{sec:finite_ps}). Now, we call $\bb{F}^N$-progressively measurable $G$-valued processes $\gamma^1$,~\ldots, $\gamma^N$ \textit{admissible controls (for the particle system)} if $\gamma^1$,~\ldots, $\gamma^N$ are progressively measurable with respect to the filtration $\bb{F}^{N, \bf{I}} = (\F^{N, \bf{I}}_t)_{0 \leq t \leq T}$ given by
\begin{equation} \label{eq:filtration_ps}
    \F^{N, \bf{I}}_t = \sigma\bigl(\xi_i,\, I^i_s,\, W^i_s,\, W^0_s \define 0 \leq s \leq t,\, i = 1, \dots, N\bigr).
\end{equation}

\begin{remark} \label{rem:controls}
Admissible controls depend on $\theta_1$,~\ldots, $\theta_N$ only through the information revealed by the indicator processes $I^1$,~\ldots, $I^N$ up to the current time. The evolution of the indicator processes depends on the controls themselves, so that with respect to $I^1$,~\ldots, $I^N$ admissible controls are in feedback form. This feature makes the definition slightly awkward, but is necessary to ensure that the controller cannot anticipate the killing of particles. The nonanticipativity guarantees that the indicator process $I^i$ has a compensator, given by $\int_0^{\cdot} I^i_t \lambda(t, X^i_t, \nu^N_t) \, \d t$, meaning that $M^i = I^i + \int_0^{\cdot} I^i_t \lambda(t, X^i_t, \nu^N_t) \, \d t$ is a martingale. For distinct particles, these martingales decorrelate asymptotically (cf.\@ Lemma \ref{lem:compensated_martingale_l_2}), so that
\begin{equation*}
    L^N_t = \frac{1}{N} \sum_{i = 1}^N (1 - I^i_t) = \frac{1}{N} \sum_{i = 1}^N (1 - M^i_t) + \frac{1}{N} \sum_{i = 1}^N \int_0^t I^i_s \lambda(s, X^i_s, \nu^N_s) \, \d s
\end{equation*}
and $\frac{1}{N} \sum_{i = 1}^N \int_0^t I^i_s \lambda(s, X^i_s, \nu^N_s) \, \d s$ have the same weak limit points. Consequently, the loss becomes continuous, even differentiable, in the limit.
\end{remark}

The central planner attempts to minimise the \textit{cost functional}
\begin{equation} \label{eq:cost_functional_ps}
    J^N(\gamma^1, \dots, \gamma^N) = \frac{1}{N} \sum_{i = 1}^N\ev\biggl[\int_0^{\tau_i \land T} f(t, X^i_t, \nu^N_t, \gamma^i_t) \, \d t + \psi(\nu^N_T)\biggr]
\end{equation}
for some final time horizon $T > 0$ over admissible controls $\gamma^1$,~\ldots, $\gamma^N$. Here $\tau_i = \inf\{t > 0 \define \Lambda^i_t \geq \theta_i\}$ is the killing time of particle $i$. We denote the infimum of $J^N$ over admissible controls by $V^N$ and call it the \textit{value}. Our goal is to study the behaviour of the controlled particle system as $N$ tends to infinity.

First, let us discuss the coefficients of SDE \eqref{eq:ps} as well as the cost functions in Equation \eqref{eq:cost_functional_ps}. For $p \geq 1$, we denote by $\M^p_{\leq 1}(\R)$ the space of $p$-integrable subprobability measures, i.e.\@ the space of measures $v$ on $\R$ for which $v(\R) \leq 1$ and $M_p^p(v) = \int_{\R} \lvert x \rvert^p \, \d v(x) < \infty$. We equip $\M^p_{\leq 1}(\R)$ with the metric
\begin{equation} \label{eq:subprob_metric}
    d_p(v_1, v_2) = W_p(m_1, m_2) + \lvert v_1(\R) - v_2(\R)\rvert,
\end{equation}
where $m_i = v_i + (1 - v_i(\R))\delta_0$, $W_p$ denotes the $p$-Wasserstein distance for probability measures on $\R$, and $v_1$, $v_2 \in \M^p_{\leq 1}(\R)$. This turns $\M^p_{\leq 1}(\R)$ into a Polish space (cf.\@ Proposition \ref{prop:space_of_subprobs}). We show in Proposition \ref{prop:krd} that $d_1$ coincides with the Kantorovich--Rubinstein distance on the space of subprobability measures with finite first moment. Hence, the metrics $d_p$ can be viewed as a strengthening of the weak convergence of measures. Clearly, the \textit{loss function} $\cal{M}^p_{\leq 1}(\R) \to [0, 1]$, $v \mapsto 1 - v(\R)$ is continuous, in fact $1$-Lipschitz continuous, with respect to the metric $d_p$.

Let us now state the assumptions on the coefficients and cost functions for the particle system.

\begin{assumption} \label{ass:red_form_cont_prob}
Let $b \define [0, T] \times \R \times \M^1_{\leq 1}(\R) \times G \to \R$, $\sigma$, $\sigma_0$, and $\alpha \define [0, T] \times \R \times \M^1_{\leq 1}(\R) \to \R$, $\lambda \define [0, T] \times \R \times \M^1_{\leq 1}(\R) \to [0, \infty)$, $f \define [0, T] \times \R \times \M^2_{\leq 1}(\R) \times G \to \R$, and $\psi \define \M^2_{\leq 1}(\R) \to \R$ be measurable. Let $G$ be a non-empty, convex, and compact subset of $\R^d$ and $\nu_0 \in \P^2(\R)$. We assume there exists a constant $C > 0$, such that
\begin{enumerate}[noitemsep, label = (\roman*)]
    \item \label{it:growth} the coefficients $\sigma$, $\sigma_0$, and $\alpha$ are bounded by $C$ and for all $t$, $x$, $v$, $g$ we have
    \begin{align*}
        \lvert b(t, x, v, g)\rvert + \lvert \lambda(t, x, v) \rvert \leq C(1 + \lvert x \rvert + M_1(v));
    \end{align*}
    \item \label{it:continuity} the coefficient $b$ is continuous in $g$ and for $h \in \{\sigma, \sigma_0, \alpha, \lambda\}$ and all $t$, $x$, $x'$, $v$, $v'$, $g$ we have
    \begin{align*}
        \lvert b(t, x, v, g) - b(t, x', v', g)\rvert + \lvert h(t, x, v) - h(t, x', v')\rvert \leq C(\lvert x - x'\rvert + d_1(v, v'));
    \end{align*}
    \item \label{it:growth_cost} for all $t$, $x$, $v$, $g$ we have
    \begin{equation*}
        \lvert f(t, x, v, g)\rvert \leq C(1 + \lvert x \rvert^2 + M_2^2(v)), \qquad \lvert \psi(v) \rvert \leq C(1 + M_2^2(v));
    \end{equation*}
    \item \label{it:continuity_cost} the running cost $f$ is continuous in $x$, $v$, $g$ and the terminal cost $\psi$ is continuous.
\end{enumerate}
\end{assumption}

We will often use subscripts to indicate constants for a specific coefficient or cost function. E.g.\@ we may write $\lvert b(t, x, v, g)\rvert \leq C_b(1 + \lvert x \rvert + M_1(v))$.

\begin{remark}
Some remarks about Assumption \ref{ass:red_form_cont_prob} are in order. First note that the system's coefficients are defined for elements of $\cal{M}^1_{\leq 1}(\R)$ and are Lipschitz continuous with respect to the metric $d_1$. Since $d_1 \leq d_2$ on $\cal{M}^2_{\leq 1}(\R)$, continuity (as well as Lipschitz continuity) with respect to $d_1$ is stronger than that with respect to $d_2$. We need this more restrictive assumption to obtain existence and uniqueness of the mean-field limit \eqref{eq:mfl} of the particle system (see Proposition \ref{prop:existence_uniqueness_stability}). In contrast, the cost functionals are continuous functions on $\cal{M}^2_{\leq 1}(\R)$. Otherwise the growth and continuity conditions, Items \ref{it:growth}, \ref{it:continuity}, \ref{it:growth_cost}, and \ref{it:continuity_cost}, are standard.

Note that we did not make any assumptions regarding the interplay of $\lambda$ and $\alpha$, so at this level of generality, the particle system does not necessarily feature contagion. However, one natural situation, in which contagion arises, is where $\lambda(t, x, v)$ is positive for $x < 0$ and vanishes for $x \geq 0$ and $\alpha(t, x, v) > 0$. Then, particles are killed on the negative half-line and the killing of a particle pushes the remaining system closer to $(-\infty, 0)$. We already saw a special case of this, namely the financial model from Subsection \ref{sec:toy_model}, and will revisit this setup in the context of the singular model, see Subsection \ref{sec:main_results_singular} below.
\end{remark}

\subsubsection{The Strong Mean-Field Control Problem}

Our objective is to prove that the optimal cost or value $V^N$ of the $N$-particle system converges to the value of a MFC problem as the number of particles becomes infinite. If we suppose that the particles become asymptotically independent given the common noise $W^0$, we would expect that $\nu^N_t$ converges weakly to the random measure $\pr(X_t \in \cdot,\, \theta > \Lambda_t \vert W^0)$. This suggests the \textit{mean-field limit}
\begin{align} \label{eq:mfl}
\begin{split}
    \d X_t &= b(t, X_t, \nu_t, \gamma_t) \, \d t + \sigma(t, X_t, \nu_t) \, \d W_t + \sigma_0(t, X_t, \nu_t) \, \d W^0_t - \alpha(t, X_t, \nu_t) \, \d L_t, \\
    \d \Lambda_t &= \lambda(t, X_t, \nu_t) \, \d t
\end{split}
\end{align}
with initial conditions $X_0 = \xi \sim \nu_0$ and $\Lambda_0 = 0$, Brownian motions $W$ and $W^0$, \textit{conditional subprobability distribution} $\nu_t = \pr(X_t \in \cdot,\, \theta > \Lambda_t \vert W^0)$, and \textit{loss} $L_t = 1 - \nu_t(\R)$. Letting $I_t = \bf{1}_{\theta > \Lambda_t}$, one may suspect that the compensator of $I$ is given by $\int_0^{\cdot} I_t \lambda(t, X_t, \nu_t) \, \d t$, so that $M = I + \int_0^{\cdot} I_t \lambda(t, X_t, \nu_t) \, \d t$ is a martingale. Since $\theta$ is independent of $W^0$, this suggests that $\ev[M_t \vert W^0] = \ev[M_0 \vert W^0] = 1$, so that
\begin{equation*}
    L_t = 1 - \ev[I_t \vert W^0] = 1 - \ev[M_t \vert W^0] + \int_0^t \ev[I_s \lambda(s, X_s, \nu_s) \vert W^0] \, \d s = \int_0^t \langle \nu_s, \lambda(s, \cdot, \nu_s)\rangle \, \d s,
\end{equation*}
where we used in the last equality that $\nu_t = \pr(X_t \in \cdot,\, \theta > \Lambda_t \vert W^0) = \ev[I_t \delta_{X_t} \vert W^0]$. Hence, the loss should become differentiable in the mean-field limit. Note further that if $\theta$ were independent of $(\gamma, W^0)$ and, therefore, of $(X, \Lambda, W^0)$, then a simple computation would show that $\nu_t = \ev[e^{-\Lambda_t} \delta_{X_t}\vert W^0]$, in which case the above McKean--Vlasov SDE is essentially standard and can be rewritten in the form
\begin{align} \label{eq:mfl_wo_theta}
\begin{split}
    \d X_t &= b(t, X_t, \nu_t, \gamma_t) \, \d t + \sigma(t, X_t, \nu_t) \, \d W_t + \sigma_0(t, X_t, \nu_t) \, \d W^0_t \\
    &\ \ \ - \alpha(t, X_t, \nu_t)\langle \nu_t, \lambda(t, \cdot, \nu_t)\rangle \, \d t, \\
    \d \Lambda_t &= \lambda(t, X_t, \nu_t) \, \d t
\end{split}
\end{align}
with $\nu_t = \ev[e^{-\Lambda_t} \delta_{X_t}\vert W^0]$.

However, a simple example shows that this independence will in general not hold. Indeed, choosing $\gamma^i_t = I^i_t$ in the particle system, leads to the control $\gamma_t = I_t$ for the mean-field limit. But in this case, $\nu_t = \ev[I_t \delta_{X_t}\vert W^0]$ and $\ev[e^{-\Lambda_t} \delta_{X_t}\vert W^0]$ will typically not coincide. The dependence between $\theta$ and $\gamma$ inhibits the smoothing effect of the exponential random variable $\theta$, leading to irregularities in the mean-field limit. Fortunately, if $b$ is linear, $f$ is convex in the control, and $\sigma$ is nondegenerate, we can show that minimising the cost over ``nonsmooth'' controls with nonanticipative dependence on $\theta$ (such as $\gamma_t = I_t$) and ``smooth'' controls without any dependence on $\theta$ yield the same optimal value (see Theorem \ref{thm:all_equiv}). Because of their desirable regularity properties, we adopt the latter type of controls for the \textit{strong formulation} of the MFC problem. That is, an \textit{admissible strong control} is an $\bb{F}^{\xi, W, W^0}$-progressively measurable processes with values in $G$. For an admissible strong control, we can guarantee strong existence and uniqueness for McKean--Vlasov SDE \eqref{eq:mfl} (cf.\@ Proposition \ref{prop:existence_uniqueness_stability}). Here we call $(X, \nu)$ a \textit{strong solution} to McKean--Vlasov SDE \eqref{eq:mfl} if $X$ is $\bb{F}^{X_0, W, W^0}$-adapted and $\nu$ is $\bb{F}^{W^0}$-adapted. The \textit{cost functional} for the strong MFC problem is given by
\begin{equation} \label{eq:strong_cost}
    J(\gamma) = \ev\biggl[\int_0^{\tau \land T} f(t, X_t, \nu_t, \gamma_t) \, \d t + \psi(\nu_T)\biggr] = \ev\biggl[\int_0^T e^{-\Lambda_t} f(t, X_t, \nu_t, \gamma_t) \, \d t + \psi(\nu_T)\biggr],
\end{equation}
where $\tau = \inf\{t > 0 \define \Lambda_t \geq \theta\}$. Finally, we denote the \textit{value}, i.e.\@ the infimum of the cost functional over admissible strong controls, by $V$.

To establish the convergence of the particle system \eqref{eq:ps} to the McKean--Vlasov SDE \eqref{eq:mfl}, we will proceed through compactness arguments and subsequential convergence. However, to obtain the tightness of the sequence of controls $(\gamma^{N, 1}$, \dots, $\gamma^{N, N})$, we will have to embed them into a larger space with a weaker topology, leading to a relaxed formulation of the control problem. Moreover, subsequential weak limits of the subprobabilities $(\nu^N)_N$ typically fail to be adapted to the filtration $\bb{F}^{W^0}$ generated by the common noise $W^0$. So, in addition, we have to consider weak solutions to McKean--Vlasov SDE \eqref{eq:mfl}, where the mean-field component $\nu$ may be adapted to a filtration which extends $\bb{F}^{W^0}$ in an appropriate sense. We address these two issues in the following.

\subsubsection{The Relaxed Mean-Field Control Problem}

To introduce the relaxed setup, we need some preparation. We set $\bf{M}^p = \cal{M}^p_{\leq 1}(\R)$ for $p \geq 1$ and equip $D[0, T]$, $D_{[0, 1]}[0, T]$, and $D_{\bf{M}^2}[0, T]$ with the $J1$-metric (see \cite[Section 3.3]{whitt_stoch_limits_2002} for a definition). Next, we introduce the space $\bb{M}_T(G)$ of measures on $[0, T] \times G$ whose time marginal is the Lebesgue measure, topologised by the weak convergence of measures (induced by the L\'evy--Prokhorov metric). We can identify an admissible strong control $\gamma$ with an $\bb{M}_T(G)$-valued random variable $\Gamma$, called a relaxed control, via $\d \Gamma(t, g) = \d \delta_{\gamma_t} \d t$. We define the Polish spaces $\cal{S} = D[0, T] \times C([0, T]) \times \bb{M}_T(G) \times D_{[0,1]}[0, T]$ and $\Omega_0 = \P^2(\cal{S}) \times C([0, T])$. The space $\cal{S}$ can accommodate the variables $X$, $W$, $\Gamma$, and $I$ and the meaning of $\Omega_0$ will become clear in a moment. Finally, we set $\Omega_{\ast} = \cal{S} \times \Omega_0$ and let $\F^{\ast} = \cal{B}(\Omega_{\ast})$ be its Borel $\sigma$-algebra. We denote the canonical random element on $\Omega_{\ast}$ by $\Theta^{\ast} = (X^{\ast}, W^{\ast}, \Gamma^{\ast}, I^{\ast}, \mu^{\ast}, B^{\ast})$, set $\Theta^0 = (\mu^{\ast}, B^{\ast})$, and let $\F^0$ be the $\sigma$-algebra generated by $\Theta^0$. The process $X^{\ast}$ will play the role of the weak solution to the McKean--Vlasov SDE \eqref{eq:mfl}, $W^{\ast}$ and $B^{\ast}$ are the idiosyncratic and common noise, $I^{\ast}$ indicates whether the particle is alive at a given time, and $\Gamma^{\ast}$ is a relaxed control. We explain $\mu^{\ast}$ next: to any probability measure $\pr_0 \in \P(\Omega_0)$, we can associate a distribution $\pr_{\ast}$ on $\Omega_{\ast}$ by setting
\begin{equation} \label{eq:prob_associate}
    \pr_{\ast}(A \times A_0) = \int_{A_0} m(A) \, \d \pr_0(m, b)
\end{equation}
for Borel subsets $A \subset \cal{S}$ and $A_0 \subset \Omega_0$. We denote the expectation with respect to $\pr_{\ast}$ by $\ev_{\ast}$. This construction implies that
\begin{equation} \label{eq:mu}
    \mu^{\ast} = \L_{\ast}(X^{\ast}, W^{\ast}, \Gamma^{\ast}, I^{\ast} \vert \F^0),
\end{equation}
where $\L_{\ast}$ is the law under $\pr_{\ast}$. Next, we define the random variable $\nu^{\ast} \define \Omega_{\ast} \to D_{\bf{M}^2}[0, T]$ given by $\nu^{\ast}_t = \int_{\cal{S}} p_t \delta_{x_t} \, \d \mu^{\ast}(x, w, \mathfrak{g}, p)$. Then $\nu^{\ast}$ is a function $\P^2(\cal{S}) \to D_{\bf{M}^2}[0, T]$ of $\mu^{\ast}$ and by \eqref{eq:mu} it holds that $\nu^{\ast}_t = \ev_{\ast}[I^{\ast}_t \delta_{X^{\ast}_t} \vert \F^0]$, so $\nu^{\ast}_t$ is the conditional subprobability of $X^{\ast}_t$ given the common information $\F^0$.

We consider two filtrations $\bb{F}^0 = (\F^0_t)_{0 \leq t \leq T}$ and $\bb{F}^{\ast} = (\F^{\ast}_t)_{0 \leq t \leq T}$ on $\Omega_{\ast}$ defined by
\begin{align} \label{eq:filtration_canon}
\begin{split}
    \F^0_t &= \sigma\bigl(\langle \mu^{\ast}_s, \varphi\rangle,\, B^{\ast}_s \mathpunct{:} \varphi \in C_b(\cal{S}),\, 0 \leq s \leq t\bigr), \\
    \F^{\ast}_t &= \sigma\bigl(X^{\ast}_s,\, W^{\ast}_s,\, \Gamma^{\ast}([0, s] \times A)\mathpunct{:} A \in \cal{B}(G),\, 0 \leq s \leq t\bigr) \lor \F^0_t.
\end{split}
\end{align}
Here $\mu^{\ast}_t$ is the pushforward of $\mu^{\ast}$ under the map $\cal{S} \ni (x, w, \mathfrak{g}, p) \mapsto (x_{\cdot \land t}, w, \mathfrak{g}_t, p_{\cdot \land t})$, where for $\mathfrak{g} \in \bb{M}_T(G)$, $\mathfrak{g}_t$ is defined through
\begin{equation} \label{eq:measure_stop}
    \mathfrak{g}_t(A_1 \times A_2) = \mathfrak{g}((A_1 \cap [0, t]) \times A_2) + \lvert A_1 \cap (t, T] \rvert \delta_{g_0}(A_2)
\end{equation}
for measurable $A_1 \subset [0, T]$ and $A_2 \subset G$ and some fixed $g_0 \in G$. In particular,
\begin{equation} \label{eq:mu_t}
    \mu^{\ast}_t = \L_{\ast}(X^{\ast}_{\cdot \land t}, W^{\ast}, \Gamma^{\ast}_t, I^{\ast}_{\cdot \land t} \vert \F^0_T) = \L_{\ast}(X^{\ast}_{\cdot \land t}, W^{\ast}, \Gamma^{\ast}_t, I^{\ast}_{\cdot \land t} \vert \F^0_t)
\end{equation}
under $\pr_{\ast}$. Note that the map $D[0, T] \ni x \mapsto x_{\cdot \land t}$ is measurable with respect to the Borel $\sigma$-algebra induced by the $J1$-topology, since the latter is generated by the projection maps $D[0, T] \ni x \mapsto x_t$ for $t \in [0, T]$. Hence, the pushforward is well-defined. Lastly, we introduce the processes $M^{\ast}$ and $\Lambda^{\ast}$ on $\Omega_{\ast}$ given by $M^{\ast}_t = I^{\ast}_t + \int_0^t I^{\ast}_s \lambda(s, X^{\ast}_s, \nu^{\ast}_s) \, \d s$ and $\Lambda^{\ast}_t = \int_0^t \lambda(s, X^{\ast}_s, \nu^{\ast}_s) \, \d s$.

\begin{remark}
Equality \eqref{eq:mu_t} is a strengthening of the \textit{immersion property}. The filtration $\bb{F}^0$ is said to be \textit{immersed in} $\bb{F}^{\ast}$ \textit{under} $\pr_{\ast}$ if $\pr_{\ast}(A \vert \F^0_t) = \pr_{\ast}(A \vert \F^0_T)$ for all $A \in \F^{\ast}_t$. Equation \eqref{eq:mu_t} implies that the previous equality even holds for all sets $A \in \F^{\ast}_t \lor \sigma(W^{\ast})$. This property is necessitated by the common noise to show that the smooth relaxed control formulation from Definition \ref{def:rel_cntrl} below yields the same value as the strong MFC problem introduced above, see \cite[Remark 2.4]{mkv_control_limit_2020} for more details.
\end{remark}

\begin{definition} \label{def:rel_cntrl}
We call $\pr_0 \in \P(\Omega_0)$ an \textit{admissible relaxed control rule} with initial condition $\nu_0 \in \P^2(\R)$ if
\begin{enumerate}[noitemsep, label = (\roman*)]
    \item \label{it:law_integrability} $\L_{\ast}(X^{\ast}_0) = \nu_0$ and $\ev_{\ast}\int_0^T \lvert X^{\ast}_t \rvert^2 \, \d t$ is finite;
    \item \label{it:independence} $W^{\ast}$ and $B^{\ast}$ are $\bb{F}^{\ast}$-Brownian motions and the pair $(X^{\ast}_0, W^{\ast})$ is independent of $\F^0_T$ under $\pr_{\ast}$;
    \item \label{it:martingale} $M^{\ast}$ is an $(\bb{F}^{I^{\ast}} \lor \bb{F}^{\ast})$-martingale under $\pr_{\ast}$ and for all bounded $\bb{F}^{\ast}$-predictable processes $H$, it holds $\pr_{\ast}$-a.s.\@ that $\ev_{\ast}[\int_0^T H_t \, \d M^{\ast}_t \vert \F^0_T] = 0$;
    \item \label{it:sde} for all $t \in [0, T]$,
    \begin{align} \label{eq:relaxed_weak_mfl}
    \begin{split}
        X^{\ast}_t = X^{\ast}_0 &+ \int_{[0, t] \times G} b(s, X^{\ast}_s, \nu^{\ast}_s, g) \, \d \Gamma^{\ast}(s, g) + \int_0^t \sigma(s, X^{\ast}_s, \nu^{\ast}_s) \, \d W^{\ast}_s \\
        &+ \int_0^t \sigma_0(s, X^{\ast}_s, \nu^{\ast}_s) \, \d B^{\ast}_s - \int_0^t \alpha(s, X^{\ast}_s, \nu^{\ast}_s) \langle \nu^{\ast}, \lambda(s, \cdot, \nu^{\ast}_s)\rangle \, \d s
    \end{split}
    \end{align}
    holds $\pr_{\ast}$-almost surely.
\end{enumerate}
Here $\pr_{\ast}$ is the probability measure on $\Omega_{\ast}$ associated to $\pr_0$ via \eqref{eq:prob_associate}. 

An admissible relaxed control $\pr_0$ is called \textit{smooth} if $\ev[I^{\ast}_t \bf{1}_A \vert \F^0_T] = \ev[e^{-\Lambda^{\ast}_t} \bf{1}_A \vert \F^0_T]$ holds $\pr_{\ast}$-a.s.\@ for all $A \in \F^{\ast}_T$.
\end{definition}

\begin{remark}
Definition \ref{def:rel_cntrl} borrows elements from Definition 2.5 in \cite{mkv_control_limit_2020}. However, in contrast to \cite{mkv_control_limit_2020}, Equation \eqref{eq:mu_t} is built into our probabilistic setup and not a postulate of Definition \ref{def:rel_cntrl}. Note that \cite[Definition 2.13]{mkv_control_limit_2020} introduces a more restrictive notion of relaxed control rule for the case of a controlled diffusion coefficient $\sigma$, which is necessary to ensure that any relaxed control rule can be approximated in a suitable way by a sequence of strong controls. Since in this article we assume that $\sigma$ is uncontrolled, we can ignore this issue.

The first part of Item \ref{it:martingale} of Definition \ref{def:rel_cntrl} ensures that the control does not anticipate the killing of particles. The second property expresses compatibility between $M^{\ast}$ and the common information $\F^0_T$. 

Note that the definition is stated without any reference to an exponentially distributed random variable, in contrast to the particle system. All information is subsumed in the indicator process $I^{\ast}$. 
However, it can be shown that by extending $\Omega_{\ast}$, if necessary, one can find a standard exponential random variable $\theta^{\ast}$ such that $I^{\ast}_t = \bf{1}_{\theta^{\ast} > \Lambda^{\ast}_t}$. Then the control rule is smooth if $\theta^{\ast}$ can be chosen such that it is independent of $\F^{\ast}_T$ conditional on $\F^0_T$.
\end{remark}

The \textit{cost functional} $\cal{J}$ for relaxed control rules $\pr_0 \in \P(\Omega_0)$ is defined as
\begin{equation} \label{eq:relaxed_cost}
    \cal{J}(\pr_0) = \ev_{\ast}\biggl[\int_{[0, T] \times G} I^{\ast}_t f(t, X^{\ast}_t, \nu^{\ast}_t, g) \, \d \Gamma^{\ast}(t, g) + \psi(\nu^{\ast}_T)\biggr].
\end{equation}
We denote the infimum of $\cal{J}(\pr_0)$ over admissible relaxed control rules and smooth relaxed control rules by $V_{\textup{rl}}$ and $V_{\textup{srl}}$, respectively. Obviously, we have $V_{\textup{rl}} \leq V_{\textup{srl}}$. Similarly to the strong formulation, under a \textit{smooth} relaxed control rule $\pr_0$, we can write $\nu^{\ast}_t = \ev_{\ast}[I^{\ast}_t \delta_{X^{\ast}_t} \vert \F^0_T] = \ev_{\ast}[e^{-\Lambda^{\ast}_t} \delta_{X^{\ast}_t} \vert \F^0_T]$, which considerably simplifies the analysis of McKean--Vlasov SDE \eqref{eq:relaxed_weak_mfl}. Moreover, for a smooth relaxed control rule $\pr_0$, the cost functional becomes 
\begin{equation*}
    \cal{J}(\pr_0) = \ev_{\ast}\biggl[\int_{[0, T] \times G} e^{-\Lambda^{\ast}_t} f(t, X^{\ast}_t, \nu^{\ast}_t, g) \, \d \Gamma^{\ast}(t, g) + \psi(\nu^{\ast}_T)\biggr].
\end{equation*}
Since every admissible strong control $\gamma$ together with the corresponding unique strong solution $X$ of McKean--Vlasov SDE \eqref{eq:mfl} induces an admissible smooth relaxed control rule in $\P(\Omega_0)$ (cf.\@ \cite[Lemma 4.4]{djete_mvoc_dpp_2022}), we obtain the trivial inequality $V_{\textup{srl}} \leq V$. Here as above, for any $G$-valued process $\gamma$ we can define a $\bb{M}_T(G)$-valued random variable $\Gamma$ by $\d \Gamma(t, g) = \d\delta_{\gamma_t}(g) \d t$. We call $\Gamma$ the \textit{relaxed control associated to} $\gamma$.

\begin{theorem} \label{thm:convergence_ps}
Let Assumption \ref{ass:red_form_cont_prob} be satisfied. Let $(\gamma^{N, 1}, \dots, \gamma^{N, N})_N$ be a sequence of controls for the particle system and denote the associated relaxed controls by $\Gamma^{N, 1}$,~\ldots, $\Gamma^{N, N}$. Define $\mu^N = \frac{1}{N}\sum_{i = 1}^N \delta_{(X^{N, i}, W^i, \Gamma^{N, i}, I^{N, i})}$. Then, the sequence $(\mu^N, W^0)_N$ is tight on $\Omega_0$ and every subsequential limit $\pr_0 \in \P(\Omega_0)$ yields an admissible relaxed control rule. Consequently, $\liminf_{N \to \infty} V^N \geq V_{\textup{rl}}$.
\end{theorem}

In the absence of controls, the previous subsequential convergence result together with uniqueness of the mean-field limit \eqref{eq:mfl} (Proposition \ref{prop:existence_uniqueness_stability}) implies propagation of chaos for the particle system.

\begin{corollary} \label{cor:prop_of_chaos}
Let Assumption \ref{ass:red_form_cont_prob} be satisfied and further suppose that the drift coefficient does not depend on the control, i.e.\@ $b(t, x, v, g) = b_0(t, x, v)$ for some function $b_0 \define [0, T] \times \R \times \M^1_{\leq 1}(\R) \to \R$. Then $(X^{N, 1}, \nu^N)_N$ converges weakly on $D[0, T] \times D_{\bf{M}^2}[0, T]$ to the unique strong solution of the McKean--Vlasov SDE \eqref{eq:mfl}.
\end{corollary}

As we indicated above, a relaxed control rule that is obtained as the subsequential limit of the particle system may not be smooth in the sense of Definition \ref{def:rel_cntrl}. However, at the current level of generality, with only Assumption \ref{ass:red_form_cont_prob} in place, we are not able to establish the equivalence between the strong MFC problem and the relaxed MFC problem with (possibly nonsmooth) control rules. If we restrict ourselves to smooth relaxed control rules, the situation is different.

\begin{theorem} \label{thm:smooth_equiv}
Let Assumption \ref{ass:red_form_cont_prob} be satisfied. Then $V_{\textup{srl}} = V \geq \limsup_{N \to \infty} V^N$.
\end{theorem}

Since the smooth formulation leads to an essentially standard MFC problem, we can draw on existing results by Djete, Possama\"i, and Tan \cite{mkv_control_limit_2020} to prove the above theorem.

Combining Theorems \ref{thm:convergence_ps} and \ref{thm:smooth_equiv} allows us to deduce that
\begin{equation*}
    V_{\textup{srl}} = V \geq \limsup_{N \to \infty} V^N \geq \liminf_{N \to \infty} V^N \geq V_{\textup{rl}}.
\end{equation*}
Thus, if we can prove that $V_{\textup{srl}} = V_{\textup{rl}}$, it follows that all expressions in the above coincide. However, to establish this missing link, we have to restrict ourselves to a linear-convex setup. \begin{assumption} \label{ass:aff_conv}
Let $b_0 \define [0, T] \times \R \times \cal{M}^1_{\leq 1}(\R) \to \R$ satisfy the same properties as $\lambda$ in Assumption \ref{ass:red_form_cont_prob} and let $b_1 \define [0, T] \times \R \times \cal{M}^1_{\leq 1}(\R) \to \R^d$ be measurable and bounded.
We assume that
\begin{enumerate}[noitemsep, label = (\roman*)]
    \item \label{it:affine} for all $t$, $x$, $v$, $g$, we have $b(t, x, v, g) = b_0(t, x, v) + b_1(t, x, v) \cdot g$;
    \item \label{it:convex} for all $t$, $x$, $v$, the map $G \ni g \mapsto f(t, x, v, g)$ is convex;
    \item \label{it:unif_nondegen} there exists $c > 0$ such that for all $t$, $x$, $v$, we have $\sigma^2(t, x, v) \geq c$.
\end{enumerate}
\end{assumption}

Under Assumption \ref{ass:aff_conv}, we can use the mimicking result for stochastic Fokker--Planck equations by Lacker, Shkolnikov \& Zhang \cite{lacker_mimicking_2020} together with measurable selection arguments to construct a smooth relaxed control rule from any given (possibly nonsmooth) relaxed control rule, which achieves the same or a lower cost. In fact, we can do better.

\begin{definition} \label{def:closed_loop}
A smooth relaxed control rule $\pr_0 \in \P(\Omega_0)$ is called a \textit{closed-loop control rule} if there exists a measurable map $g_{\ast} \define [0, T] \times \R \times \M^1_{\leq 1}(\R) \to G$ such that $\pr_{\ast}$-a.s.\@ $\d \Gamma^{\ast}(t, g) = \delta_{g_{\ast}(t, X^{\ast}_t, \nu^{\ast}_t)}(g)\d t$. We refer to $g_{\ast}$ as the \textit{feedback function}.
\end{definition}

Note that the feedback function $g_{\ast}$ of a closed-loop control does not depend on the cumulative intensity $\Lambda^{\ast}_t$ or the indicator $I^{\ast}$. Moreover, its dependence on the conditional joint law $\L_{\ast}(X^{\ast}_t, \Lambda^{\ast}_t \vert \F^0_T)$ is only through the conditional subprobability distribution $\nu^{\ast}_t$. 

We denote the infimum of $\cal{J}$ over closed-loop controls by $V_{\textup{cl}}$. Since a closed-loop control rule is by definition a smooth relaxed control rule, we immediately see that $V_{\textup{cl}} \geq V_{\textup{srl}}$.

\begin{theorem} \label{thm:all_equiv}
Let Assumptions \ref{ass:red_form_cont_prob} and \ref{ass:aff_conv} be satisfied. Then for any relaxed control rule $\pr_0 \in \P(\Omega_0)$, there exists a closed-loop control rule $\pr_0' \in \P(\Omega_0)$ such that $\cal{J}(\pr_0') \leq \cal{J}(\pr_0)$. In particular,
\begin{equation*}
    \lim_{N \to \infty} V^N = V_{\textup{rl}} = V_{\textup{cl}} = V_{\textup{srl}} = V.
\end{equation*}
\end{theorem}

The above result not only establishes the equivalence between the strong and the relaxed formulation of the MFC problem, but also between the optimisation over open- and the closed-loop controls. It differs from existing equivalence results for MFC with common noise, such as \cite[Theorem 8.3]{lacker_mimicking_2020}, in that the feedback function $g_{\ast}$ only depends on the first component of the full state $(X^{\ast}_t, \Lambda^{\ast}_t)$ (or, alternatively, $(X^{\ast}_t, I^{\ast}_t)$) and only on $\nu^{\ast}_t$ instead of the joint law $\L_{\ast}(X^{\ast}_t, \Lambda^{\ast}_t \vert \F^0_T)$. This dimension reduction is crucial for the computational feasibility of the numerical simulations discussed in Section \ref{sec:num_sim}.

\subsection{Convergence to the Singular Limit} \label{sec:main_results_singular}

The McKean--Vlasov SDE \eqref{eq:mfl} can be viewed as a regularisation of a model with absorption and singular interaction through the hitting time. In this model particles are not killed according to an exponential clock but are removed from the system once their state hits zero. It is well-known that in this case, the system may exhibit jumps, which originate from discontinuities in the loss: a macroscopic portion of the system is removed at once (cf.\@ \cite[Theorem 1.1]{hambly_mckean_vlasov_blow_up_2019}). Under suitable assumptions (see Assumption \ref{ass:struc_limit}) we can obtain the singular framework as a limit of our regularised model as the intensity function $\lambda$ converges to $(t, x, v) \mapsto \infty \bf{1}_{(-\infty, 0)}$ in a suitable sense. 

To deal with the singular framework, we have to adapt the probabilistic setup introduced above Theorem \ref{thm:convergence_ps}. For convenience, we shall reappropriate some of the symbols used there. We set $\cal{S} = D[-1, T + 1] \times C([0, T + 1]) \times \bb{M}_T(G)$ and $\Omega_0 = \P^2(\cal{S}) \times C([0, T + 1])$. We comment on the enlargement of the time domain in Remark \ref{rem:j1_to_m1}. In contrast to the regularised model, we equip the space of c\`adl\`ag functions with the $M1$-topology (we refer the reader to \cite[Section 12.3]{whitt_stoch_limits_2002} for a definition and basic properties of the $M1$-metric and -topology). Then, we set $\Omega_{\ast} = \cal{S} \times \Omega_0$ and let the canonical random element on $\Omega_{\ast}$ be $\Theta^{\ast} = (X^{\ast}, W^{\ast}, \Gamma^{\ast}, \mu^{\ast}, B^{\ast})$. We define the filtrations $\bb{F}^0 = (\F^0_t)_{0 \leq t \leq T + 1}$ and $\bb{F} = (\F_t)_{0 \leq t \leq T + 1}$ analogously to \eqref{eq:filtration_canon}, where we extend $\mathfrak{g} \in \bb{M}_T(G)$ beyond $T$ by $\mathfrak{g}(A_1 \times A_2) = \mathfrak{g}((A_1 \cap [0, T]) \times A_2) + \lvert A_2 \cap (T, T + 1]\rvert \delta_{g_0}(A_1)$ for any measurable $A_1 \subset [0, T + 1]$, $A_2 \subset G$, and some arbitrary element $g_0 \in G$. As in Equation \eqref{eq:mu_t} we have
\begin{equation} \label{eq:mu_t_sin}
    \mu^{\ast}_t = \L_{\ast}(X^{\ast}_{\cdot \land t}, W^{\ast}, \Gamma^{\ast}_t \vert \F^0_T) = \L_{\ast}(X^{\ast}_{\cdot \land t}, W^{\ast}, \Gamma^{\ast}_t \vert \F^0_t).
\end{equation}

In addition to these random variables, we define the random time $\tau^{\ast} = \inf\{0 < t < T + 1 \define X^{\ast}_t \leq 0\}$ with the convention $\inf \emptyset = \infty$ and the flow of subprobability distributions $\nu^{\ast} \define \Omega_{\ast} \to D_{\bf{M}^2}[-1, T + 1]$ given by $\nu^{\ast}_t = \int_{\cal{S}} \delta_{x_{0-}} \, \d \mu^{\ast}(x, w, \mathfrak{g})$ if $t \in [-1, 0)$ and
\begin{equation} \label{eq:subprobability_sin}
    \nu^{\ast}_t = \int_{\cal{S}} \bf{1}_{\{\inf_{0 \leq s \leq t} x_s > 0\}} \delta_{x_t} \, \d \mu^{\ast}(x, w, \mathfrak{g})
\end{equation}
otherwise. By Equation \eqref{eq:mu_t_sin} we have $\nu^{\ast}_t = \pr_{\ast}(X^{\ast}_t \in \cdot,\, \tau^{\ast} > t \vert \F^0_T)$. Finally, we set $L^{\ast}_t = 1 - \nu^{\ast}_t(\R)$. Note that by definition $L^{\ast} = (L^{\ast}_t)_{-1 \leq t \leq T + 1}$ is c\`adl\`ag and nondecreasing with $L^{\ast}_t = 0$ on $[-1, 0)$.

\begin{remark} \label{rem:j1_to_m1}
The enlargement of the time domain as well as the shift from the $J1$- to the weaker $M1$-topology on the c\`adl\`ag space is necessary to derive the tightness and subsequential convergence of the sequence of losses as we let the intensity tend to $\infty \bf{1}_{(-\infty, 0)}$ (cf.\@ Proposition \ref{prop:sequence_tightness}). Indeed, in general, the loss of the limiting system will have jumps. However, no sequence of continuous functions can converge in $J1$ to a function with jumps. Even in the $M1$-topology a sequence of continuous functions can only converge to a function that is continuous at the interval endpoints. We make sure of this property by extending the time domain and imposing continuity at the new endpoints. 
\end{remark}

We extend the coefficients to $[0, T + 1]$ in the following way: we define $b$ and $\sigma_0$ equal to zero after $T$ and set $\sigma$ to one after $T$. We also assume that $\alpha$ is simply a continuous function $[0, T] \to \R$, which we extend to $[0, T + 1]$ in a continuous way and such that it vanishes on $[T + 1/2, T]$. The exact shape of $\alpha$ on $[T, T + 1]$ does not matter, since the cost function, which we introduce below, only depends on our setup for times in $[0, T]$.

\begin{definition} \label{def:rel_cntr_struc}
We call $\pr_0 \in \P(\Omega_0)$ an \textit{admissible relaxed control rule} (for the singular model) with initial condition $\nu_{0-} \in \P^2(\R)$ if
\begin{enumerate}[noitemsep, label = (\roman*)]
    \item $\L_{\ast}(X^{\ast}_{0-}) = \nu_{0-}$ and $\ev_{\ast}\int_0^T \lvert X^{\ast}_t \rvert^2 \, \d t$ is finite;
    \item \label{it:independence_sin} $W^{\ast}$ and $B^{\ast}$ are $\bb{F}^{\ast}$-Brownian motions and the pair $(X^{\ast}_0, W^{\ast})$ is independent of $\F^0_T$ under $\pr_{\ast}$;
    \item for all $t \in [-1, 0)$ we have $X^{\ast}_t = X^{\ast}_{0-}$ and for all $t \in [0, T + 1]$,
    \begin{align} \label{eq:relaxed_weak_mfl_struc}
    \begin{split}
        X^{\ast}_t = X^{\ast}_{0-} &+ \int_{[0, t] \times G} b(s, X^{\ast}_s, \nu^{\ast}_s, g) \, \d \Gamma^{\ast}(s, g) + \int_0^t \sigma(s, X^{\ast}_s, \nu^{\ast}_s) \, \d W^{\ast}_s \\
        &+ \int_0^t \sigma_0(s, X^{\ast}_s, \nu^{\ast}_s) \, \d B^{\ast}_s - \int_0^t \alpha(s) \, \d L^{\ast}_s \quad \pr_{\ast}\textup{-almost surely}.
    \end{split}
    \end{align}
\end{enumerate}
Here $\pr_{\ast}$ is the probability measure on $\Omega_{\ast}$ associated to $\pr_0$.
\end{definition}
The definition of an admissible relaxed control rule implies that after time $T + 1/2$ the process $X^{\ast}$ simply diffuses according to the idiosyncratic noise $W^{\ast}$, that is $X^{\ast}_t = X^{\ast}_{(T + 1/2)-} + W^{\ast}_t - W^{\ast}_{T + 1/2}$ for $t \in [T + 1/2, T + 1]$. This together with the nondegeneracy of the diffusion coefficients ensures that the random time $\tau$ is $\pr_{\ast}$-a.s.\@ continuous for any admissible control rule $\pr_0 \in \cal{P}(\Omega_0)$ (see Proposition \ref{prop:subprobabilities_converge}).

Unfortunately, in the singular case, we cannot use the same cost functional as for the regularised model. Indeed, functions of the form $D_{\bf{M}^2}[-1, T+ 1] \ni v \mapsto \psi(v_T)$ for continuous $\psi \define \cal{M}^2_{\leq 1}(\R) \to \R$ are generally not $M1$-continuous (or even upper semicontinuous) at elements of $D_{\bf{M}^2}[-1, T+ 1]$ that jump at $T$. Since the process ${\nu}^{\ast}$ might have a discontinuity at $T$ with positive probability under an admissible relaxed control rule, the terminal cost in \eqref{eq:relaxed_cost} is not suitable. However, we can replace the terminal cost by an expression of the form $D_{\bf{M}^2}[-1, T+ 1] \ni v \mapsto \frac{1}{\epsilon} \int_{T - \epsilon}^T \psi(v_t) \, \d t$ with $\epsilon > 0$ small. Indeed, this function is itself continuous and it approximates $D_{\bf{M}^2}[-1, T+ 1] \ni v \mapsto \psi(v_T)$ well for flows $v$ that are continuous at $T$. Generalising this idea slightly, we introduce the \textit{cost function}
\begin{equation}
    \cal{J}_{\textup{sg}}(\pr_0) = \ev_{\ast}\biggl[\int_{[0, \tau^{\ast} \land T] \times G} f(t, X^{\ast}_t, \nu^{\ast}_t, g) \, \d \Gamma^{\ast}(t, g) + \int_0^T \psi(t, \nu^{\ast}_t) \, \d t\biggr]
\end{equation}
for any admissible relaxed control rule $\pr_0 \in \P(\Omega_0)$, where $\psi \define [0, T] \times \cal{M}^2_{\leq 1}(\R) \to \R$. We define the \textit{value} $V_{\textup{sg}}$ as the infimum of $\cal{J}_{\textup{sg}}(\pr_0)$ over all admissible control rules $\pr_0 \in \P(\Omega_0)$.

Next, we consider the sequence of regularised approximations. Let $\lambda^n \define [0, T] \times \R \times \cal{M}^1_{\leq 1}(\R) \to \R$ be a family of functions satisfying the same assumptions as $\lambda$ in Assumption \ref{ass:red_form_cont_prob}. As for the other coefficients we extend $\lambda^n$ to a function $[0, T + 1] \times \R \times \cal{M}^1_{\leq 1}(\R) \to \R$ by setting $\lambda(t, x, \nu) = \lambda(T, x, \nu)$ for $t \in (T, T + 1]$. Then for a sequence of admissible strong controls $(\gamma^n)_n$ we can uniquely solve the McKean--Vlasov SDE
\begin{align} \label{eq:mfl_sequence}
\begin{split}
    \d X^n_t &= b(t, X^n_t, \nu^n_t, \gamma^n_t) \, \d t + \sigma(t, X^n_t, \nu^n_t) \, \d W_t+ \sigma_0(t, X^n_t, \nu^n_t) \, \d W^0_t - \alpha(t) \, \d L^n_t, \\
    \d \Lambda^n_t &= \lambda^n(t, X^n_t, \nu^n_t) \, \d t
\end{split}
\end{align}
with initial conditions $X^n_0 = \xi$ and $\Lambda^n_0 = 0$, conditional subprobability distribution $\nu^n_t = \pr(X^n_t \in \cdot,\, \theta > \Lambda^n_t \vert W^0)$, and loss $L^n_t = 1 - \nu^n_t(\R)$. Our goal is to show that the sequence $(\L(X^n, W, \Gamma^n \vert W^0), W^0)_n$ converges to an admissible relaxed control rule (for the singular model) as $\lambda^n$ tends to $\infty \bf{1}_{(-\infty, 0)}$. Here $\Gamma^n$ is the relaxed control associated to $\gamma^n$, i.e.\@ $\d \Gamma^n(t, g) = \d \delta_{\gamma^n_t} \d t$. We impose the following assumptions.

\begin{assumption} \label{ass:struc_limit}
Let the coefficients $b$, $\sigma$, and $\sigma_0$ and the cost $f$ satisfy Assumption \ref{ass:red_form_cont_prob}. Let $\alpha \define [0, T] \to \R$, $\psi \define [0, T] \times \cal{M}^2_{\leq 1}(\R) \to \R$, and $\lambda^n \define [0, T] \times \R \times \cal{M}^1_{\leq 1}(\R) \to [0, \infty)$, $n \geq 1$, be measurable. We assume that there exists a $C > 0$, such that
\begin{enumerate}[noitemsep, label = (\roman*)]
    \item \label{it:alpha} the function $\alpha$ is continuous and nonnegative;
    \item \label{it:ellipticity} for all $t$, $x$, $v$ we have $\sigma^2(t, x, v) + \sigma_0^2(t, x, v) \geq \frac{1}{C}$;
    \item \label{it:intensity_convergence} 
    \begin{enumerate}
        \item the intensity functions $\lambda^n$ satisfy the same assumptions as $\lambda$ in Assumption \ref{ass:red_form_cont_prob};
        \item for all $t$, $v$ we have $\lambda^n(t, x, v) = 0$ if $x \geq 0$;
        \item for all $t$ and any sequence $(x_n)_n$ in $\R$ with $\limsup_{n \to \infty} x_n < 0$, we have
    \begin{equation*}
        \inf_{v \in \cal{M}^1_{\leq 1}(\R)} \lambda^n(t, x_n, v) \to \infty;
    \end{equation*}
    \end{enumerate}
    \item \label{it:terminal_cost_cont} the map $\psi$ is continuous in $v$ and for all $t$, $v$ we have $\lvert \psi(t, v)\rvert \leq C(1 + M_2^2(v))$.
\end{enumerate}
\end{assumption}

\begin{remark} \label{rem:ass_singular}
The reason that $\alpha$ is assumed to be a continuous function of time is that otherwise we cannot guarantee the convergence $\int_0^{\cdot} \alpha(t, X^n_t, \nu^n_t) \, \d L^n_t \Rightarrow \int_0^{\cdot} \alpha(t, X^{\ast}_{t-}, \nu^{\ast}_{t-}) \, \d L^{\ast}_t$ due to the possible presence of common jumps in the trajectories of $\nu^{\ast}$ and $L^{\ast}$.

The nondegeneracy of the diffusion coefficients, Assumption \ref{ass:struc_limit} \ref{it:ellipticity} ensures the subsequential weak convergence of the killing times $\tau_n = \inf\{0 < t < T + 1 \define \Lambda^n_t \geq \theta\}$ to $\tau^{\ast}$ (cf.\@ Proposition \ref{prop:subprobabilities_converge}). It replaces the smoothing provided by the exponential random variable $\theta$ in the regularised model.
\end{remark}

\begin{theorem} \label{thm:convergence_struc}
Let Assumption \ref{ass:struc_limit} be satisfied and fix a sequence $(\gamma^n)_n$ of admissible controls. Next, let $X^n$ be the solution to McKean--Vlasov SDE \eqref{eq:mfl_sequence} with control $\gamma^n$ and intensity function $\lambda^n$, and denote the relaxed control associated to $\gamma^n$ by $\Gamma^n$. Then the sequence $(\L(X^n, W, \Gamma^n \vert W^0), W^0)_n$ is tight on $\Omega_0$, every subsequential limit $\pr_0 \in \P(\Omega_0)$ yields an admissible relaxed control rule (for the singular model), and
\begin{equation} \label{eq:cost_reg_seq}
    J_n(\gamma^n) = \ev\biggl[\int_0^T e^{-\Lambda^n_t} f(t, X^n_t, \nu^n_t, \gamma^n_t) \, \d t + \int_0^T \psi(t, \nu^n_t) \, \d t \biggr]
\end{equation}
converges to $\cal{J}_{\textup{sg}}$ along the corresponding subsequence.
\end{theorem}

In the statement of the theorem, we extend $X^n$ and $\nu^n$ to $[-1, T + 1]$ in the same way as $X^{\ast}$ and $\nu^{\ast}$ in Definition \ref{def:rel_cntr_struc}. Implicit in Theorem \ref{thm:convergence_struc} is a (weak) existence theorem for McKean--Vlasov SDE \eqref{eq:relaxed_weak_mfl_struc} which generalises the existing literature \cite{hambly_mckean_vlasov_absorbing_2017, hambly_spde_model_2019, hambly_mckean_vlasov_blow_up_2019, blow_ups_common_noise_ledger_2021}.

\begin{remark} \label{rem:conv_of_nu}
Theorem \ref{thm:convergence_struc} does not say anything about the subsequential weak convergence of the flow of subprobabilities $(\nu^n)_n$ on $D_{\bf{M}^2}[-1, T + 1]$ to $\nu^{\ast}$, where we recall that $\bf{M}^2 = \cal{M}^2_{\leq 1}(\R)$. In fact, at this time, we are not sure whether this convergence holds. The issue lies in the possibility that the feedback $\int_0^t \alpha(s) \, \d L^n_s$ may push a nonnegligible number of particles sufficiently far below zero before they are killed. Such a situation could arise if the limiting system jumps with positive probability at some time $t \in [0, T]$. It is then conceivable that we can find a sequence of ($\bb{F}^{W^0}$-stopping) times $t_n \in [0, T]$ with $t_n \to t$ such that the sequence $(\nu^n_{t_n}((-\infty, 0)))_n$ is bounded away from zero, which precludes weak convergence to $\nu^{\ast}$.

In Proposition \ref{prop:subprobabilities_converge} below we instead prove that $(\nu^n)_n$ converges weakly along a subsequence on the space $L^2([-1, T + 1]; \bf{M}^2)$ of square-integrable functions $[-1, T + 1] \to \bf{M}^2$. 
\end{remark}

Unlike in Theorem \ref{thm:convergence_ps}, even if we assume that the sequence $(\gamma^n)_n$ of admissible controls in Theorem \ref{thm:convergence_struc} satisfies $J_n(\gamma^n) \leq \inf_{\gamma} J_n(\gamma) + \epsilon_n$ with $\epsilon_n \to 0$, we cannot conclude that $J_n(\gamma^n)$ converges to the infimum $V_{\textup{sg}}$ of $\cal{J}_{\textup{sg}}$, since there is no uniqueness theory for the McKean--Vlasov SDE \eqref{eq:relaxed_weak_mfl_struc}, which represents the singular model. In fact, pathwise and weak uniqueness are known to fail even for constant coefficients and without controls (unless solutions are restricted to a smaller class of so-called physical solutions). We refer the reader to \cite{delarue_supercooled_uniqueness_2022, hambly_mckean_vlasov_blow_up_2019} for more details. Consequently, it may not be possible to approximate every solution to McKean--Vlasov SDE \eqref{eq:relaxed_weak_mfl_struc} by the sequence of regularised models \eqref{eq:mfl_sequence}. If this is the case for all (nearly) optimally controlled solutions of \eqref{eq:relaxed_weak_mfl_struc}, then $J_n(\gamma^n)$ will not converge to $V_{\textup{sg}}$.

Therefore, we specialise to a simpler situation, where the diffusion coefficients $\sigma$ and $\sigma_0$ are only functions of time, $\lambda^n$ depends solely on $x$, the drift coefficient is given by $b(t, g) = b_0(t) + b_1(t) g$ for bounded and measurable functions $b_0$, $b_1 \define [0, T] \to \R$, and $G$ is a subset of $\R$. Then the McKean--Vlasov SDE \eqref{eq:relaxed_weak_mfl_struc} becomes
\begin{equation} \label{eq:struct_simple}
    \d X^{\ast}_t = \int_G (b_0(t) + b_1(t) g) \, \d \Gamma^{\ast}(t, g) + \sigma(t) \, \d W^{\ast}_t  + \sigma_0(t) \, \d B^{\ast}_t - \alpha(t) \, \d L^{\ast}_t.
\end{equation}
with $L^{\ast}_t = \pr_{\ast}(\tau \leq t \vert \F^0_T)$. While we still cannot guarantee uniqueness of the McKean--Vlasov SDE \eqref{eq:struct_simple}, we can exploit the monotonicity of the equation in the loss function. To this end, we formulate the following assumption.

\begin{assumption} \label{ass:monotonicity}
We assume that 
\begin{enumerate}[noitemsep, label = (\roman*)]
    \item for all $t$, $x$, $v$ we have $f(t, x, v, 0) = 0$;
    \item for all $t$, $x$, $x'$, $v$, $v'$, $g$, $g$ with $x \geq x'$, $v \geq v'$, $g \leq g'$ we have
\begin{equation*}
    f(t, x, v, g) \leq f(t, x', v', g'),\qquad \psi(t, v) \leq \psi(t, v').
\end{equation*}
\end{enumerate}
Here we say that $v \leq v'$ for two elements $v$, $v' \in \cal{M}^2_{\leq 1}(\R)$ if $v([x, \infty)) \leq v'([x, \infty))$ for all $x \in \R$.
\end{assumption}
Being small with respect to the partial order introduced in Assumption \ref{ass:monotonicity} heuristically means that mass is predominately found on the lower half-line or has already disappeared.

\begin{proposition} \label{prop:conv_struc_simple_coeff}
Let Assumptions \ref{ass:struc_limit} and \ref{ass:monotonicity} be satisfied and suppose that the coefficients and $G$ have the form outlined above Equation \eqref{eq:struct_simple}. Let $(X^n, \nu^n, \gamma^n)_n$ be as in Theorem \ref{thm:convergence_struc} and additionally assume that $J_n(\gamma^n) \leq V_n + \epsilon_n$ for a sequence $(\epsilon_n)_n$ that converges to zero. Then $\lim_{n \to \infty} V_n = \lim_{n \to \infty} J_n(\gamma^n) = V_{\textup{sg}}$, where $V_n$ denotes the infimum of $J_n(\gamma)$ over all admissible strong controls $\gamma$.
\end{proposition}

\begin{remark}
We can disintegrate $\Gamma^{\ast}$ as $\d \Gamma^{\ast}(t, g) = \Gamma^{\ast}_t (g) \, \d t$ for an $\bb{F}^{\ast}$-predictable process $(\Gamma^{\ast}_t)_t$ with values in $\P(G)$ (see \cite[Lemma 3.2]{lacker_mfg_controlled_mgale_2015}). Then, if we set $\gamma^{\ast}_t = \int_G g \, \d \Gamma^{\ast}_t$ for $t \in [0, T]$, the process $\gamma^{\ast} = (\gamma^{\ast}_t)_{0 \leq t \leq T}$ is $\bb{F}^{\ast}$-progressively measurable and the state equation \eqref{eq:struct_simple} becomes
\begin{equation*}
    \d X^{\ast}_t = (b_0(t) + b_1(t) \gamma^{\ast}_t) \, \d t + \sigma(t) \, \d W^{\ast}_t  + \sigma_0(t) \, \d B^{\ast}_t - \alpha(t) \, \d L^{\ast}_t.
\end{equation*}
Moreover, assuming that the running cost $f$ is convex in $g$, Jensen's inequality implies
\begin{equation*}
    \cal{J}_{\textup{sg}}(\pr_0) \geq \ev_{\ast}\biggl[\int_0^{\tau^{\ast} \land T} f(t, X^{\ast}_t, \nu^{\ast}_t, \gamma^{\ast}_t) \, \d t + \int_0^T \psi(t, \nu^{\ast}_t) \, \d t\biggr].
\end{equation*}
Consequently, in this natural setup, it is always possible to choose an optimal control that takes the form of an $\bb{F}^{\ast}$-progressively measurable $G$-valued process as opposed to a random measure on $[0, T] \times G$.
\end{remark}

As we mentioned before, even in the absence of controls, i.e.\@ if $b_1(t) = 0$, McKean--Vlasov SDE \eqref{eq:struct_simple} does not exhibit (weak or pathwise) uniqueness. Hence, obtaining a weak convergence result in the spirit of Corollary \ref{cor:prop_of_chaos} seems hopeless. However, if we assume that the $(\lambda^n)_n$ form a nondecreasing sequence in the sense that $\lambda^m(x) \geq \lambda^n(x)$ whenever $m \geq n$, then the same will be true for $(L^n)_n$, i.e.\@ $L^m_t \geq L^n_t$ if $m \geq n$. Consequently, the random variables $L^n_t$ converge a.s.\@ for each $t \in [0, T + 1]$. This allows us to define a c\`adl\`ag process $L = (L_t)_{0 \leq t \leq T + 1}$ as follows: we define $\ell_t$ as the a.s.\@ limit of $L^n_t$ for $t \in [0, T + 1] \cap (\mathbb{Q} \cup \{T + 1\})$, where $\bb{Q}$ are the rationals. Then we set $L_t = \lim_{s \searrow t, s \in (t, T + 1] \cap \bb{Q}} \ell_s$ for $t \in [0, T + 1)$ and $L_{T + 1} = \ell_{T + 1}$.
We show in the proof of Proposition \ref{eq:minimal_convergence} that $L^n_t \to L_t$ a.s.\@ for any $t \in [0, T + 1]$ with $\pr(\Delta L_t = 0) = 1$ and that there exists a c\`adl\`ag process $X$ such that
\begin{equation} \label{eq:minimal}
    \d X_t = b_0(t) \, \d t + \sigma(t) \, \d W_t + \sigma_0(t) \, \d W^0_t - \alpha(t) \, \d L_t
\end{equation}
and $L_t = \pr(\tau \leq t \vert W^0)$, where $\tau = \inf\{t > 0 \define X_t \leq 0\}$. That is, $(X, L)$ is a strong solution to McKean--Vlasov SDE \eqref{eq:struct_simple}. Moreover, we prove that the process $L$ is \textit{minimal} in the sense that if $(X', L')$ is any other solution of McKean--Vlasov SDE \eqref{eq:minimal} on the same probability space then $L_t \leq L'_t$ for all $t \in [0, T + 1]$.

For the purpose of the following proposition we extend $(X^n, L^n)$ and $(X, L)$ to the interval $[-1, T + 1]$ by setting $X^n_t = X_t = X_{0-}$ and $L^n_t = L_t = 0$ on $[-1, 0)$.

\begin{proposition} \label{eq:minimal_convergence}
Let Assumptions \ref{ass:struc_limit} and \ref{ass:monotonicity} be satisfied and suppose that the coefficients have the form outlined above Equation \eqref{eq:struct_simple}. Assume further that $b_1 = 0$ and that $\lambda^m(x) \geq \lambda^n(x)$ for $x \in \R$ whenever $m \geq n$. Then $(X^n, L^n)_n$ converges a.s.\@ in $D[-1, T + 1] \times D[-1, T + 1]$ to the minimal solution $(X, L)$ of McKean--Vlasov SDE \eqref{eq:minimal}.
\end{proposition}

\subsection{Future Work}

There are a various avenues for future work. Firstly, one may include controls in a subset of the currently uncontrolled coefficients $\sigma$, $\sigma_0$, $\alpha$, and $\lambda$, the latter being of particular interest in the context of the financial model from Section \ref{sec:toy_model}. For example, one could consider a situation, where the central bank controls the threshold at which $x \mapsto \lambda(t, x, v)$ becomes positive, corresponding to a change in the regulatory capital requirements. If we simultaneously adjust the terminal cost in \eqref{eq:cost_toy_model} so that it only counts banks with a sufficiently healthy capital position, then the central bank has to trade off the risk of contagion during a systemic crisis with the health of the banking system after the crisis. 

Another related topic for future research is the study of other types of control problems, such as impulse control, singular control or stochastic target problems, in the same framework. Particularly for the adjustment of the default threshold discussed in the preceding paragraph, impulse controls, that only allow for finitely (or countably) many changes to the capital requirements, are suitable. In the context of a stochastic target problem, the goal could be to keep a certain fraction of particles (or banks in the context of the model from Section \eqref{eq:cost_toy_model}) alive until time $T$. The controller's objective is then to find an initial position (corresponding to an up-front bailout at the initial time) that ensures the target is satisfied at the final time.

Lastly, one could allow the particles to control their own dynamics as well. This extension could be realised in the form of a Stackelberg mean-field game or a mean-field game with major and minor players. In the former scenario, the central planner is the leader who announces a policy to which the particles, which compete amongst each other, react. In the case of a mean-field game with major and minor players, the central planner does not have a first mover advantage. Instead, the central planner (major player) and the particles (minor players) all compete on an equal footing. Through this extension, moral hazard can be introduced into the financial model from Section \eqref{eq:cost_toy_model}. If the banks anticipate a bailout by the government, they will increase the riskiness of their investments to achiever higher returns because potential losses are partly externalised through the bailout.

\section{Convergence of the Nearly Optimally Controlled Particle System} \label{sec:ps}

The goal of this section is to establish the subsequential convergence of the (approximately optimally) controlled particle system to an (optimal) relaxed control rule. In particular, we prove the main results Theorem \ref{thm:convergence_ps}, Corollary \ref{cor:prop_of_chaos} as well as Theorems \ref{thm:smooth_equiv} and \ref{thm:all_equiv}.

\subsection{The Finite Particle System} \label{sec:finite_ps}

Our first objective is to analyse the particle system and deduce tightness on $\Omega_0$ of the sequence $(\mu^N, W^0)_N$ as defined in Theorem \ref{thm:convergence_ps}. Throughout this subsection Assumption \ref{ass:red_form_cont_prob} is in place.

Initially, we show that the particle system \eqref{eq:ps} is well-posed and has a unique strong solution for any $\bb{F}^N$-progressively measurable processes $\gamma^1$,~\ldots, $\gamma^N$. Note that these inputs do not necessarily need to be admissible controls as defined in Subsection \ref{sec:main_results_ps}. To be precise, by a strong solution we mean a solution adapted to the filtration $\bb{F}^N$, while uniqueness is understood in a pathwise sense. That is, any two solutions of SDE \eqref{eq:ps} with the same initial conditions, exponential clocks, Brownian motions, and inputs $\gamma^1$,~\ldots, $\gamma^N$ coincide almost surely. We may inductively construct such a solution on $[0, S)$ for any $S > 0$ as follows: set $\varrho_0 = 0$ and $D_0 = \varnothing$, and define $\xi_{0, i} = \xi_i$, $i = 1$,~\ldots, $N$. Next, suppose that a stopping time $\varrho_{n - 1} \leq S$, a random subset $D_{n - 1}$ of $\{1, \dots, N\}$, and the $\F^N_{\varrho_{n - 1}}$-measurable random variables $\xi_{n - 1, i}$, $i = 1$,~\ldots, $N$, are defined for $n \geq 1$. Then we let $Y^{n, i}$ denote the unique solution to
\begin{equation*}
    \d Y^{n, i}_t = b(t, Y^{n, i}_t, \nu^{n, N}_t, \gamma^i_t) \, \d t + \sigma(t, Y^{n, i}_t, \nu^{n, N}_t) \, \d W^i_t + \sigma_0(t, Y^{n, i}_t, \nu^{n, N}_t) \, \d W^0_t
\end{equation*}
on the random interval $[\varrho_{n - 1}, S]$, where $Y^{n, i}_{\varrho_{n - 1}} = \xi_{n - 1, i}$ and $\nu^{n, N}_t = \frac{1}{N}\sum_{i = 1}^N \bf{1}_{i \notin D_{n - 1}} \delta_{Y^{n, i}_t}$. Such a solution exists e.g.\@ by \cite[Theorem 3.3.1]{zhang_bsde_2017}, which can be readily extended to random initial times. Then, we define $\varrho_n$ as the minimum between $S$ and the first killing time of a previously alive particle, i.e.\@ the first time $t \geq \varrho_{n - 1}$ such that $\theta_i \leq \int_0^t \lambda(s, Y^{n, i}_s, \nu^{n, N}_s) \, \d s$ for some $i \notin D_{n - 1}$. We obtain the set $D_n$ by adding the index of the killed particle to $D_{n - 1}$. If two particles are removed at the same time, we deal with them iteratively, selecting the particle with the smaller index first. Finally, we set $\xi_{n, i} = Y^{n, i}_{\varrho_{n - 1}-} - \frac{1}{N}\alpha(\varrho_{n - 1}, Y^{n, i}_{\varrho_{n - 1}-}, \nu^{n, N}_{\varrho_{n - 1}-})$. Since at most $N$ killings can occur $\varrho_{N + 1} = S$ and it is not difficult to see that the processes $X^1$,~\ldots, $X^N$ defined by $X^i_t = Y^{n, i}_t$ for $t \in [\varrho_{n - 1}, \varrho_n)$, $n = 1$,~\ldots, $N + 1$, form a solution to \eqref{eq:ps} on $[0, S)$. Uniqueness is an immediate consequence of the uniqueness on the individual intervals $[\varrho_{n - 1}, \varrho_n)$.

From now on, we assume that $\gamma^1$,~\ldots, $\gamma^N$ are admissible controls. The proof of the following uniform square-integrability result for the particle trajectories is entirely standard, so we omit it.

\begin{lemma} \label{lem:l_2_bound_ps}
The probability distribution $\bigl(\frac{1}{N}\sum_{i = 1}^N \L(X^i)\bigr)_N$ on $C([0, T])$ is uniformly square-integrable.
\end{lemma}

Recall the processes $\Lambda^i_t = \int_0^t \lambda(s, X^i_s, \nu^N_s) \, \d s$ and $I^i_t = \bf{1}_{\theta_i > \Lambda^i_t}$, and for $i = 1$,~\ldots, $N$ define $M^i$ by
\begin{equation*}
    M^i_t = I^i_t + \int_0^t I^i_s \lambda(s, X^i_s, \nu^N_s) \, \d s
\end{equation*}
for $t \geq 0$. If the intensity function $\lambda$ were only to depend on time, it would be easy to see that $\int_0^{\cdot} I^i_t \lambda(t) \, \d t$ is the compensator of $I^i$, so that $M^i$ is a martingale. The same is still true if $\lambda$ depends on $\theta_i$ through the indicator $I^i_t$ as we show in Proposition \ref{prop:martingale_property} below. The main challenge in the proof is to disentangle $\theta_i$ from $I^i_t$, which we achieve by introducing an auxiliary particle system. 

\begin{proposition} \label{prop:martingale_property}
The processes $M^i$ for $i = 1$,~\ldots, $N$ are martingales with respect to the filtration $\bb{F}^{N, \bf{I}}$ defined in \eqref{eq:filtration_ps}.
\end{proposition}

\begin{proof}
Let us introduce the stochastic process $\bf{Y}$ given by $\bf{Y}_t = ((\xi_i, W^i_t)_i, W^0_t)$, so that the filtration $\bb{F}^{N, \bf{I}}$ is generated by the pair $(\bf{Y}, \bf{I})$. In particular, since $\gamma^1$,~\dots, $\gamma^N$ are admissible, for $i = 1$,~\ldots, $N$, we can find a measurable function $\Phi^i \define [0, T] \times (\R \times C([0, T]))^N \times C([0, T]) \times (D_{[0, 1]}[0, T])^N \to G$ such that $\gamma^i_t = \Phi^i_t(\bf{Y}_{\cdot \land t}, \bf{I}_{\cdot \land t})$ for $\textup{Leb} \otimes \pr$-a.e.\@ $(t, \omega) \in [0, T] \times \Omega$. Now, let us fix an index $i \in \{1, \dots, N\}$. We can modify the solution $(X^1, \dots, X^N)$ of the particle system \eqref{eq:ps} constructed at the beginning of the section such that the modification $(\tilde{X}^1, \dots, \tilde{X}^N)$ coincides with the original solution up to $\tau_i = \inf\{t > 0 \define \Lambda^i_t \geq \theta_i\}$ but is independent of $\theta_i$. We achieve this by obtaining $(\tilde{X}^1, \dots, \tilde{X}^N)$ as the solution of the SDE
\begin{align} \label{eq:ps_one_out}
\begin{split}
    \d \tilde{X}^j_t &= b\bigl(t, \tilde{X}^j_t, \tilde{\nu}^N_t, \Phi^j_t(\bf{Y}_{\cdot \land t}, \tilde{\bf{I}'}_{\cdot \land t})\bigr) \, \d t + \sigma(t, \tilde{X}^j_t, \tilde{\nu}^N_t) \, \d W^j_t + \sigma_0(t, \tilde{X}^j_t, \tilde{\nu}^N_t) \, \d W^0_t \\
    &\ \ \ - \alpha(t, \tilde{X}^j_{t-}, \tilde{\nu}^N_{t-}) \, \d \tilde{L}^N_t, \\
    \d \tilde{\Lambda}^j_t &= \lambda(t, \tilde{X}^j_t, \tilde{\nu}^N_t) \, \d t,
\end{split}
\end{align}
with initial conditions $\tilde{X}^j_0 = \xi_j$ and $\tilde{\Lambda}^j_0 = 0$, where $\tilde{\bf{I}}$ has entries $\tilde{I}^k$ given by $\tilde{I}^k_t = \bf{1}_{\theta_k > \tilde{\Lambda}^k_t}$ for $t \geq 0$, $\tilde{\bf{I}}'$ arises from $\tilde{\bf{I}}$ by replacing the $i$th entry with the function $[0, \infty) \to [0, 1]$ that is constantly equal to $1$, and $\tilde{\nu}^N_t = \frac{1}{N} \sum_{k \neq i} \tilde{I}^k_t \delta_{\tilde{X}^k_t} + \frac{1}{N} \delta_{\tilde{X}^i_t}$. Using a similar construction as at the beginning of the section, we can show that SDE \eqref{eq:ps_one_out} has a unique strong solution $(\tilde{X}^1, \dots, \tilde{X}^N)$, which is independent of $\theta_i$, as it does not appear anywhere in the construction. Moreover, since $(X^1, \dots, X^N)$ satisfies the same SDE up to $\tau_i$, it follows from the pathwise uniqueness that $(\tilde{X}^1, \dots, \tilde{X}^N)$ and $(X^1, \dots, X^N)$ coincide on $[0, \tau_i)$.
Hence, it holds that $\tilde{I}^j_t = I^j_t$ for $t \in [0, \tau_i)$ and $j = 1$,~\ldots, $N$. This in turn implies that the process $\tilde{M}^i_t = \tilde{I}^i_t + \int_0^t \tilde{I}^i_s \lambda(s, \tilde{X}^i_s, \tilde{\nu}^N_t) \, \d s$ coincides with $M^i$ up to $\tau_i$. But both $\tilde{M}^i$ and $M^i$ are constant after $\tau_i$, so $\tilde{M}^i$ and $M^i$ must be equal.

We now show that $\tilde{M}^i$ is an $\bb{F}^{N, \bf{I}}$-martingale. Firstly, $\tilde{M}^i$ is $\bb{F}^{N, \bf{I}}$-adapted and integrable at all times since it is bounded in $L^2$ by Lemma \ref{lem:l_2_bound_ps}. Thus, it remains to verify the martingale property. Let $\Psi \define (\R \times C([0, T]))^N \times C([0, T]) \times (D_{[0, 1]}[0, T])^N \to \R$ be bounded and measurable. Our goal is to show that $\ev[(\tilde{M}^i_t - \tilde{M}^i_s) \Psi(\bf{Y}_{\cdot \land s}, \bf{I}_{\cdot \land s})] = 0$ for all $t \geq s \geq 0$, as this implies the martingale property for $\tilde{M}^i$ if we let $\Psi$ range over all bounded and measurable functions. Note that if $s < \tau_i$, we have $\bf{I}_{\cdot \land s} = \tilde{\bf{I}}_{\cdot \land s}$, and, otherwise, $\tilde{M}^i_t - \tilde{M}^i_s$ vanishes. Consequently, we get
\begin{equation} \label{eq:equality_with_tilde}
    \ev\bigl[(\tilde{M}^i_t - \tilde{M}^i_s) \Psi(\bf{Y}_{\cdot \land s}, \bf{I}_{\cdot \land s})\bigr] = \ev\bigl[(\tilde{M}^i_t - \tilde{M}^i_s) \Psi(\bf{Y}_{\cdot \land s}, \tilde{\bf{I}}_{\cdot \land s})\bigr].
\end{equation}
The difference $\tilde{I}^i_t - \tilde{I}^i_s$ can only be nonzero if $\tilde{I}^i_s$, and therefore $\tilde{I}^i_u$ for $0 \leq u \leq s$, is equal to $1$. Thus, we get
\begin{equation*}
    \bb{E}\bigl[(\tilde{I}^i_t - \tilde{I}^i_s) \Psi(\bf{Y}_{\cdot \land s}, \tilde{\bf{I}}_{\cdot \land s})\bigr] = \bb{E}\bigl[(\tilde{I}^i_t - \tilde{I}^i_s) \Psi(\bf{Y}_{\cdot \land s}, \tilde{\bf{I}}'_{\cdot \land s})\bigr].
\end{equation*}
The process $\tilde{\Lambda}^i$ as well as the random variable $\Psi(\bf{Y}_{\cdot \land s}, \tilde{\bf{I}}'_{\cdot \land s})$ are independent of $\theta_i$. Since $\theta_i$ follows a standard exponential distribution, we can explicitly compute
\begin{align} \label{eq:exponential_calc}
\begin{split}
    \bb{E}\bigl[(\tilde{I}^i_t - \tilde{I}^i_s) \Psi(\bf{Y}_{\cdot \land s}, \tilde{\bf{I}}'_{\cdot \land s})\bigr] &= \bb{E}\bigl[\bigl(e^{-\tilde{\Lambda}^i_t} - e^{-\tilde{\Lambda}^i_s}\bigr) \Psi(\bf{Y}_{\cdot \land s}, \tilde{\bf{I}}'_{\cdot \land s})\bigr] \\
    &= \bb{E}\biggl[\Psi(\bf{Y}_{\cdot \land s}, \tilde{\bf{I}}'_{\cdot \land s}) \int_s^t -\lambda(u, \tilde{X}^i_u, \tilde{\nu}^N_u) e^{-\tilde{\Lambda}^i_u}\, \d u\biggr] \\
    &= \bb{E}\biggl[\Psi(\bf{Y}_{\cdot \land s}, \tilde{\bf{I}}'_{\cdot \land s}) \int_s^t -\tilde{I}^i_u \lambda(u, \tilde{X}^i_u, \tilde{\nu}^N_u) \, \d u\biggr].
\end{split}
\end{align}
Again, $-\int_s^t \tilde{I}_u\lambda(u, \tilde{X}^i_u, \tilde{\nu}^N_u) \, \d u$ is only nonzero if $\tilde{I}^i_s = 1$, whence
\begin{equation*}
    \bb{E}\biggl[\Psi(\bf{Y}_{\cdot \land s}, \tilde{\bf{I}}'_{\cdot \land s}) \int_s^t -\tilde{I}^i_u \lambda(u, \tilde{X}^i_u, \tilde{\nu}^N_u)\, \d u\biggr] = \bb{E}\biggl[\Psi(\bf{Y}_{\cdot \land s}, \tilde{\bf{I}}_{\cdot \land s}) \int_s^t -\tilde{I}^i_u \lambda(u, \tilde{X}^i_u, \tilde{\nu}^N_u)\, \d u\biggr].
\end{equation*}
Combining this equality with the previous two equations and rearranging implies that $\ev[(\tilde{M}^i_t - \tilde{M}^i_s) \Psi(\bf{Y}_{\cdot \land s}, \tilde{\bf{I}}_{\cdot \land s})] = 0$, so in view of Equation \eqref{eq:equality_with_tilde}, we get the desired equality $\ev[(\tilde{M}_t - \tilde{M}_s) \Psi(\bf{Y}_{\cdot \land s}, \bf{I}_{\cdot \land s})] = 0$.
\end{proof}

The martingale property of $M^i$ provides the crucial step in the proof of tightness of the sequence $(\mu^N, \nu^N, W^0)_N$, which is the subject of Proposition \ref{prop:tightness_ps} below. However, this does not imply that the loss and, therefore, the particle trajectories become continuous in the limit. To establish that we need to prove that the martingales $M^i$, $i = 1$,~\ldots, $N$, decorrelate asymptotically. 

\begin{lemma} \label{lem:compensated_martingale_l_2}
It holds that $\bb{E}\bigl\lvert\frac{1}{N}\sum_{i = 1}^N (M^i_t - 1)\bigr\rvert^2 = O(1/N)$ as $N \to \infty$.
\end{lemma}

\begin{proof}
For simplicity, we set $Z^i_t = M^i_t - 1$. We have that
\begin{equation*}
    \bb{E}\biggl\lvert \frac{1}{N} \sum\limits_{i = 1}^N Z^i_t \biggr\rvert^2 = \frac{1}{N^2}\sum\limits_{i = 1}^N \bb{E}\lvert Z^i_t \rvert^2 + \frac{1}{N^2} \sum\limits_{i = 1}^N \sum\limits_{j \neq i} \bb{E}[Z^i_t Z^j_t].
\end{equation*}
Lemma \ref{lem:l_2_bound_ps} implies that the expectations $\bb{E}\lvert Z^i_t \rvert^2$ are bounded uniformly in $N \geq 1$. Consequently, the first summand on the right-hand side is of order $1/N$ as required. Next, we prove that $\bb{E}[Z^i_t Z^j_t] = 0$ whenever $i \neq j$, so that the second summand vanishes. We have that
\begin{equation*}
    \bb{E}[Z^i_t Z^j_t] = \bb{E}[[Z^i, Z^j]_t] = \bb{E}\biggl[\sum_{0 \leq s \leq t} \Delta I^i_s \Delta I^j_s\biggr],
\end{equation*}
where $[\cdot, \cdot]$ denotes the quadratic covariation between two c\`adl\`ag semimartingales. Recall that $\tau_i = \inf\{t > 0 \define \Lambda^i_t \geq \theta_i\}$. We claim that $\pr(\tau_i = \tau_j < \infty) = 0$, which implies that the right-hand side above and, therefore, $\bb{E}[Z^i_t Z^j_t]$ equals zero. We can define a particle system $\tilde{X}^1$,~\ldots, $\tilde{X}^N$ similar to \eqref{eq:ps_one_out}, introduced in the proof of Proposition \ref{prop:martingale_property}. However, instead of only removing the possibility of killing the $i$th particle, we remove the possibility of killing particles $i$ and $j$, so the new particle system coincides with the original one up to time $\tau_i \land \tau_j$ and is independent of $(\theta_i, \theta_j)$. Hence, for $k = i$, $j$, the processes $\Lambda^k$ and $\tilde{\Lambda}^k = \int_0^{\cdot} \lambda(s, \tilde{X}^k_s, \tilde{\nu}^N_s) \, \d s$ agree up to $\tau_i \land \tau_j$, so that $\{\tau_i = \tau_j < \infty\} = \{\tilde{F}^i_{\theta_i} = \tilde{F}^j_{\theta_j} < \infty\}$, where $\tilde{F}^k_p = \inf\{t > 0 \define \tilde{\Lambda}^k_t \geq p\}$ for $p \in [0, \infty)$. However, $(\tilde{\Lambda}^i, \tilde{\Lambda}^j)$ and, therefore, $(\tilde{F}^i, \tilde{F}^j)$ are independent of $(\theta_i, \theta_j)$, so that
\begin{align*}
    \pr(\tau_i = \tau_j < \infty) &= \pr(\tilde{F}^i_{\theta_i} = \tilde{F}^j_{\theta_j} < \infty) \\
    &=\ev \int_{[0, \tilde{\Lambda}^i_{\infty}) \times [0, \tilde{\Lambda}^j_{\infty})} e^{-(p + q)} \bf{1}_{\tilde{F}^i_p = \tilde{F}^j_q} \, \d p \d q \\
    &= \ev \int_{[0, \infty)^2} \lambda(s, \tilde{X}^i_s, \tilde{\nu}^N_s)\lambda(r, \tilde{X}^j_r, \tilde{\nu}^N_r) e^{-(\tilde{\Lambda}^i_s + \tilde{\Lambda}^j_r)} \bf{1}_{s = r} \, \d s \d r \\
    &= 0.
\end{align*}
Here, we applied a change of variables in the third equality and used that at $p = \tilde{\Lambda}^k_s$, the function $p \mapsto \tilde{F}^k_p$ attains the value $\inf\{r > 0 \define \tilde{\Lambda}^k_r \geq \tilde{\Lambda}^k_s\} = s$. This establishes the claim and concludes the proof.
\end{proof}

Let us associate to each control $\gamma^i$ a random measure $\Gamma^i$ on $[0, T] \times G$ defined through $\d \Gamma^i(t, g) = \d \delta_{\gamma^i_t} \d t$. Obviously, $\Gamma^i$ takes values in $\bb{M}_T(G)$, the space measures on $[0, T] \times G$ with the Lebesgue measure as the time-marginal. Since $G$ is compact, the same is true for $\bb{M}_T(G)$. We can now prove the tightness of the particle system.

\begin{proposition} \label{prop:tightness_ps}
Let $(\gamma^{N, 1}, \dots, \gamma^{N, N})_{N \geq 1}$ be a sequence of admissible controls. Then, the sequences $(\mu^N, W^0)_N$ and $(\nu^N)_N$ are tight on $\Omega_0$ and $D_{\bf{M}^2}[0, T]$, respectively. Moreover, any limit point of $\bigl(\frac{1}{N}\sum_{i = 1}^N \L^{\pr}(X^{N, i})\bigr)_N$ is concentrated on the space of continuous functions $C([0, T])$.
\end{proposition}

\begin{proof}
We begin by proving tightness of $(\nu^N)_N$. By Lemma \ref{lem:weak_conv_upgrade}, the sequence $(\nu^N)_N$ is tight on $D_{\bf{M}^2}[0, T]$ if it is tight on $D_{\bf{M}^1}[0, T]$ (here $\bf{M}^1 = \cal{M}^1_{\leq 1}(\R)$) and is uniformly square-integrable in the sense that
\begin{equation*}
    \lim_{R \to \infty} \sup_{N \geq 1} \ev \sup_{t \in [0, T]} \int_{\R} \bf{1}_{\lvert x \rvert \geq R} \lvert x \rvert^2 \, \d \nu^N_t(x) = 0.
\end{equation*}
The latter holds true by Lemma \ref{lem:l_2_bound_ps}, so it remains to show tightness of $(\nu^N)_N$ on $D_{\bf{M}^1}[0, T]$. The space $\cal{M}^1_{\leq 1}(\R)$ equipped with the distance $d_1$ is a complete separable metric space. Thus by Theorem 8.6 of Chapter 3 in \cite{ethier_convergence_markov_1986} to prove tightness it is enough to show that (i) $(\nu^N_t)_N$ is tight for all $t \in [0, T]$ and (ii) there exists a sequence of random variables $(\zeta_N)_N$ with $\sup_N \bb{E} \zeta_N < \infty$ such that for all $0 \leq s \leq t \leq T$,
\begin{equation} \label{eq:d_1_bound}
    \bb{E}\bigl[d_1(\nu^N_t, \nu^N_s) \big\vert \F^{N, \bf{I}}_s\bigr] \leq \lvert t - s \rvert^{\frac{1}{2}} \bb{E}[\zeta_N \vert \F^{N, \bf{I}}_s].
\end{equation}
By the second part of Proposition \ref{prop:space_of_subprobs}, the sequence $(\nu^N_t)_N$ is tight on $\bf{M}^1$ if and only if the sequence $(\nu^N_t + (1 - \nu^N_t(\R))\delta_0)_N$ is tight on $\P^1(\R)$. By \cite[Proposition B.1]{lacker_mfg_controlled_mgale_2015}, the latter sequence is tight if $\ev[\nu^N_t + (1 - \nu^N_t(\R))\delta_0]$, $N \geq 1$, is precompact in $\P^1(\R)$ and the sequence of first moments $M_1(\nu^N_t + (1 - \nu^N_t(\R))\delta_0)$, $N \geq 1$, is uniformly integrable. Both follow from the uniform bound on $\frac{1}{N} \sum_{i = 1}^N \ev (\lvert X^i \rvert^{\ast}_T)^2$ that we established in Lemma \ref{lem:l_2_bound_ps}. This gives (i).  

For (ii), by definition of the metric $d_1$ from \eqref{eq:subprob_metric} and using the fact that $\nu^N_t + (1 - \nu^N_t(\R))\delta_0 = \frac{1}{N}\sum_{i = 1}^N \bigl(I^i_t \delta_{X^i_t} + (1 - I^i_t)\delta_0\bigr) = \frac{1}{N}\sum_{i = 1}^N \delta_{I^i_t X^i_t}$, we see that
\begin{align} \label{eq:split_of_metric}
    d_1(\nu^N_t, \nu^N_s) &= W_1\biggl(\frac{1}{N}\sum_{i = 1}^N \delta_{I^i_t  X^i_t}, \frac{1}{N}\sum_{i = 1}^N  \delta_{I^i_s X^i_s}\biggr) + \lvert L^N_t - L^N_s\rvert \notag \\
    &\leq \frac{1}{N} \sum_{i = 1}^N \bigl\lvert I^i_t X^i_t - I^i_s X^i_s\bigr\rvert + \frac{1}{N}\sum_{i = 1}^N \lvert I^i_t - I^i_s\rvert \notag \\
    &\leq \frac{1}{N} \sum\limits_{i = 1}^N \lvert X^i_t - X^i_s\rvert + \frac{1}{N} \sum\limits_{i = 1}^N (1 + \lvert X^i_s\rvert) \lvert I^i_t - I^i_s\rvert.
\end{align}
We will derive appropriate bounds in terms of a random variable $\zeta_N$ for both summands on the right-hand side above. We start with the first term. Define $\chi^k_N = \frac{1}{N} \sum_{i = 1}^N (\lvert X^i \rvert^{\ast}_T)^k$ and note that by Lemma \ref{lem:l_2_bound_ps} that $\sup_{N \geq 1}\bb{E} \chi^k_N < \infty$ for $k = 1$, $2$. Then we estimate using Assumption \ref{ass:red_form_cont_prob} \ref{it:growth} and the Burkholder-Davis-Gundy inequality
\begin{align*}
    \bb{E}\bigl[\lvert &X^i_t - X^i_s\rvert \big\vert \cal{F}^{N, \bf{I}}_s\bigr] \\
    &\leq \lvert t - s \rvert^{\frac{1}{2}} \Bigl(C_b \sqrt{T}\bb{E}\bigl[1 + \lvert X^i \rvert^{\ast}_T + \chi^1_N \big\vert \cal{F}^{N, \bf{I}}_s\bigr] + C_1 C_{\sigma} \Bigr)  + C_{\alpha} \bb{E}\bigl[\lvert L^N_t - L^N_s \rvert \big\vert \cal{F}^{N, \bf{I}}_s\bigr]
\end{align*}
for some constant $C_1 > 0$. Here we exploited that $\int_s^t \alpha(u, X^i_{u-}, \nu^N_{u-}) \, \d L^N_u \leq C_{\alpha} \lvert L^N_t - L^N_s \rvert$. Summing over $i = 1$,~\ldots, $N$, and dividing by $N$ yields
\begin{align} \label{eq:aux_ineq}
\begin{split}
    \frac{1}{N}\sum\limits_{i = 1}^N \bb{E}\bigl[\lvert X^i_t& - X^i_s\rvert \big\vert \cal{F}^{N, \bf{I}}_s\bigr] \\
    &\leq \lvert t - s\rvert^{\frac{1}{2}} \Bigl(C_b \bigl(\sqrt{T} + 2\sqrt{T} \bb{E}[\chi^1_N \vert \cal{F}^{N, \bf{I}}_s]\bigr) + C_1C_{\sigma}\Bigr) \\
    &\ \ \ + C_{\alpha} \bb{E}\bigl[\lvert L^N_t - L^N_s \rvert \big\vert \cal{F}^{N, \bf{I}}_s\bigr].
\end{split}
\end{align}
The first term on the right-hand side is of the desired form. We deal with the expression involving the loss increment next. We can do this together with the second summand in Equation \eqref{eq:split_of_metric}. First note that $\lvert I^i_t - I^i_s\rvert = I^i_s - I^i_t \leq M^i_s - M^i_t + \int_s^t \lvert \lambda(u, X^i_u, \nu^N_u)\rvert \, \d u$. The process $M^i$ is a martingale with respect to the filtration $\bb{F}^{N, \bf{I}}$ by Proposition \ref{prop:martingale_property}. Thus, since $\varphi(X^i_s)$ is $\F^{N, \bf{I}}_s$-measurable for any measurable function $\varphi \define \R \to \R$, assuming that $\varphi$ is of at most linear growth we obtain that
\begin{align*}
    \bb{E}\bigl[\lvert \varphi(X^i_s) \rvert \lvert I^i_t - I^i_s\rvert \big\vert \cal{F}^{N, \bf{I}}_s\bigr] &\leq \lvert \varphi(X^i_s)\rvert \bb{E}\bigl[ M^i_t - M^i_s \big\vert \cal{F}^{N, \bf{I}}_s\bigr] \\
    &\ \ \ + C_{\lambda} \lvert t - s\rvert \lvert \varphi(X^i_s)\rvert  \bb{E}\bigl[1 + \lvert X^i\rvert^{\ast}_T + \chi^1_N\big\vert \cal{F}^{N, \bf{I}}_s\bigr] \\
    &= C_{\lambda} \lvert t - s\rvert \lvert \varphi(X^i_s)\rvert  \bb{E}\bigl[1 + \lvert X^i\rvert^{\ast}_T + \chi^1_N\big\vert \cal{F}^{N, \bf{I}}_s\bigr].
\end{align*}
Averaging over all particles and choosing $\varphi(x) = x$ allows us to bound the second term on the right-hand side of Equation \eqref{eq:split_of_metric} by $2 C_{\lambda} \sqrt{T} \lvert t - s\rvert^{\frac{1}{2}} \bigl(1 + \bb{E}[\chi^2_N \vert \cal{F}^{N, \bf{I}}_s]\bigr)$. Setting $\varphi(x) = 1$ yields a bound for the loss in \eqref{eq:aux_ineq}:
\begin{equation*}
    C_{\alpha} \bb{E}\bigl[\lvert L^N_t - L^N_s \rvert \big\vert \cal{F}^{N, \bf{I}}_s\bigr] \leq 2 C_{\alpha} C_{\lambda} \sqrt{T} \lvert t - s\rvert^{\frac{1}{2}} \bigl(1 + \bb{E}[\chi^1_N \vert \cal{F}^{N, \bf{I}}_s]\bigr).
\end{equation*}
Hence, upon defining
\begin{equation*}
    \zeta_N = 2(C_b + C_{\alpha} C_{\lambda}) \sqrt{T} (1 + \chi^1_N) + C_1 C_{\sigma} + 2 C_{\lambda} \sqrt{T} (1 + \chi^2_N),
\end{equation*}
we obtain the desired bound \eqref{eq:d_1_bound} with $\sup_N \bb{E} \zeta_N < \infty$.

Next, we address the tightness of $\mu^N$. It is enough to prove tightness of the empirical laws of $(X^{N, i})_i$, $(W^i)_i$, $(\Gamma^{N, i})_i$, and $(I^{N, i})_i$ for $N \geq 1$ separately. For the Brownian motions, this is obvious, and for the controls tightness follows from the compactness of $\bb{M}_T(G)$. Thus it remains to show that $\frac{1}{N} \sum_{i = 1}^N \delta_{X^{N, i}}$ is tight on $\P^2(D[0, T])$ and $\frac{1}{N} \sum_{i = 1}^N \delta_{I^{N, i}}$ is tight on $\P(D_{[0, 1]}[0, T])$. However, by Proposition 2.2 (ii) of \cite{sznitman_prop_chaos_1991} this boils down to proving that $\frac{1}{N} \sum_{i = 1}^N \L(X^{N, i})$ and $\frac{1}{N} \sum_{i = 1}^N \L(I^{N, i})$ are relatively compact in $\P^2(D[0, T])$ and $\P(D_{[0, 1]}[0, T])$, respectively. The latter is equivalent to tightness of $\frac{1}{N} \sum_{i = 1}^N \L(X^{N, i})$ and $\frac{1}{N} \sum_{i = 1}^N \L(I^{N, i})$ on $D[0, T]$ and $D_{[0, 1]}[0, T]$, respectively, together with uniform square-integrability of $\frac{1}{N} \sum_{i = 1}^N \L(X^{N, i})$ by Theorem 7.12 of \cite{villan_ot_2003}. In both cases, tightness follows from essentially the same argument we used to establish tightness of $\nu^N$ above and uniform square-integrability holds by Lemma \ref{lem:l_2_bound_ps}.

Let us now move to the last statement, that any limit point of $\bigl(\frac{1}{N}\sum_{i = 1}^N \L^{\pr}(X^{N, i})\bigr)_N$ is concentrated on $C([0, T])$. We can write $X^{N, i}$ as the sum of a continuous process $C^{N, i}$ and the purely discontinuous martingale $J^{N, i}$ defined by
\begin{equation*}
    J^{N, i}_t = \frac{1}{N} \sum_{j = 1}^N\int_0^t \alpha(s, X^{N, i}_{s-}, \nu^N_{s-}) \, \d M^{N, j}_s.
\end{equation*}
We will prove below that $\ev \sup_{0 \leq t \leq T} \lvert J^{N, i}_t\rvert^2$ converges to zero as $N \to \infty$. Consequently, the families $\bigl(\frac{1}{N}\sum_{i = 1}^N \L^{\pr}(X^{N, i})\bigr)_N$ and $\bigl(\frac{1}{N}\sum_{i = 1}^N \L^{\pr}(C^{N, i})\bigr)_N$ have the same weak limit points on $D[0, T]$. However, the processes $C^{N, i}$ are continuous and the space of continuous functions $C([0, T])$ is closed in $D[0, T]$ with respect to the topology of convergence in $J1$. Thus, any limit point of $\bigl(\frac{1}{N}\sum_{i = 1}^N \L^{\pr}(C^{N, i})\bigr)_N$ and, therefore, $\bigl(\frac{1}{N}\sum_{i = 1}^N \L^{\pr}(X^{N, i})\bigr)_N$ is concentrated on $C([0, T])$ as desired. It remains to show that $J^{N, i}$ vanishes in the $L^2$-$\sup$-limit. For brevity, let us set $H^{N, i}_t = \alpha(t, X^{N, i}_{t-}, \nu^N_{t-})$ and $\bar{M}^N_t = \frac{1}{N}\sum_{j = 1}^N (M^{N, j}_t - 1)$. Then by the Burkholder-Davis-Gundy inequality,
\begin{align*}
    \ev\sup_{0 \leq t \leq T} \lvert J^{N, i}\rvert^2 &\leq 4 \ev \int_0^T \lvert H^{N, i}_t\rvert^2 \, \d [\bar{M}^N]_t \leq 4 C_{\alpha}^2 \ev [\bar{M}^N]_T = 4 C_{\alpha}^2 \ev\biggl\lvert \frac{1}{N}\sum_{j = 1}^N (M^{N, j}_T - 1) \biggr\rvert^2.
\end{align*}
The expression on the right-hand side is in $O(1/N)$ by Lemma \ref{lem:compensated_martingale_l_2}, which concludes the proof.
\end{proof}

\subsection{Properties of the Limit System} \label{sec:limit_system}

As in the previous subsection, we impose Assumption \ref{ass:red_form_cont_prob} and fix admissible controls $\gamma^1$,~\ldots, $\gamma^N$ for the particle system. Then, in view of Proposition \ref{prop:tightness_ps}, the sequence $(\mu^N, W^0)_N$ is tight. For notational convenience in this and the subsequent two subsections, we suppress the asterisk in the superscript of the canonical variables $\Theta^{\ast} = (X^{\ast}, W^{\ast}, \Gamma^{\ast}, I^{\ast}, \mu^{\ast}, B^{\ast})$ on the space $\Omega_{\ast}$. E.g.\@ we write $X$ instead of $X^{\ast}$. We also set
\begin{equation*}
    \Theta^{N, i} = (X^{N, i}, W^i, \Gamma^{N, i}, I^{N, i}, \mu^N, W^0)
\end{equation*}
for $i = 1$,~\ldots, $N$ and $\Theta^{N, 0} = (\mu^N, W^0)$. The first goal of this subsection is to verify that all subsequential limits $\pr_0$ of the particle system $(\mu^N, W^0)_N$ satisfy Property \ref{it:martingale} of Definition \ref{def:rel_cntrl}. Secondly, we will show that for any subsequential limit $\pr_0 \in \P(\Omega_0)$ of $(\mu^N, W^0)_N$, we have that $\L(\nu^N)$ converges weakly to $\L_{\ast}(\nu)$ on $D_{\bf{M}^2}[0, t]$ along the chosen subsequence, where $\pr_{\ast}$ is the probability distribution on $\Omega_{\ast}$ associated to $\pr_0$ via \eqref{eq:prob_associate} and $\L_{\ast}$ denotes the law under $\pr_{\ast}$.

Note that throughout this and the following subsections, we will generically use $\pr_{\ast}$ to denote the probability distribution associated to $\pr_0$ via \eqref{eq:prob_associate} without mentioning it explicitly. We start with the following lemma.

\begin{lemma} \label{lem:weak_conv_particle}
Let $\pr_0$ be a subsequential limit of $(\mu^N, W^0)_N$, then $\frac{1}{N}\sum_{i = 1}^N \L(\Theta^{N, i})$ converges weakly to $\pr_{\ast}$ along the same subsequence. In particular, the process $X$ has $\pr_{\ast}$-a.s.\@ continuous trajectories and $\pr_{\ast}(\Delta I_t = 1) = 0$ for all $t \in [0, T]$.
\end{lemma}

\begin{proof}
For ease of notation assume that convergence holds along the entire sequence. Let $\varphi \in C_b(\Omega_{\ast})$, so the map $\Omega_0 \ni (m, v, b) \mapsto \langle m, \varphi(\cdot, m, b)\rangle$ is continuous. Then, the continuous mapping theorem implies
\begin{equation*}
    \frac{1}{N}\sum_{i = 1}^N \ev \varphi(\Theta^{N, i}) = \ev \langle \mu^N, \varphi(\cdot, \Theta^{N, 0}) \rangle \to \int_{\Omega_0} \langle m, \varphi(\cdot, m, b)\rangle \, \d \pr_0(m, b).
\end{equation*}
However, $\int_{\Omega_0} \langle m, \varphi(\cdot, m, b)\rangle \, \d \pr_0(m, b) = \ev_{\ast}\langle \mu, \varphi(\cdot, \Theta^0)\rangle = \ev_{\ast}\varphi(\Theta)$, which gives the desired convergence.

Next, by Proposition \ref{prop:tightness_ps} any limit point of $\bigl(\frac{1}{N}\sum_{i = 1}^N \L^{\pr}(X^{N, i})\bigr)_N$ is concentrated on $C([0, T])$. By the above, the law of $X$ under $\pr_{\ast}$ is precisely such a limit point, so that $X$ has a.s.\@ continuous trajectories under $\pr_{\ast}$.

Lastly, since $\frac{1}{N}\sum_{i = 1}^N \L(I^{N, i})$ converges weakly to $\L_{\ast}(I)$ on $D_{[0, 1]}[0, 1]$, we can find a dense set of times $\bb{T} \subset [0, T]$ including $0$ and $T$ such that $\frac{1}{N}\sum_{i = 1}^N \L(I^{N, i}_t) \Rightarrow \L_{\ast}(I_t)$. Now, let us fix an arbitrary $t \in [0, T]$. Then, for any $t_0$, $t_1 \in \bb{T}$ with $t_0 \leq t \leq t_1$, we have that
\begin{align*}
    \pr_{\ast}(\Delta I_t = 1) &\leq \ev_{\ast}[I_{t_1} - I_{t_0}]\\
    &= \lim_{N \to \infty} \frac{1}{N}\sum_{i = 1}^N \ev\bigl[I^{N, i}_{t_1} - I^{N, i}_{t_0}\bigr] \\
    &= \lim_{N \to \infty} \frac{1}{N}\sum_{i = 1}^N \ev \int_{t_0}^{t_1} I^{N, i}_s \lambda(s, X^{N, i}_s, \nu^N_s) \, \d s \\
    &\leq \lim_{N \to \infty} \frac{1}{N}\sum_{i = 1}^N C_{\lambda}\bigl(1 + 2\ev(\lvert X^{N, i}\rvert^{\ast}_T)^2\bigr) (t_1 - t_0).
\end{align*}
By Lemma \ref{lem:l_2_bound_ps}, the right-hand side is bounded by $C(t_1 - t_0)$ for a constant $C > 0$ independent of $t_0$ and $t_1$. Hence, letting $t_0 \nearrow t$ and $t_1 \searrow t$, implies $\pr_{\ast}(\Delta I_t = 1) = 0$. 
\end{proof}

Let us recall the process $M$ given by $M_t = I_t + \int_0^t I_s \lambda(s, X_s, \nu_s) \, \d s$. We can verify the following properties for $M$ under any limiting probability $\pr_0$.

\begin{proposition} \label{prop:compensator_limit}
Let $\pr_0$ be a subsequential limit of $(\mu^N, W^0)_N$. Then $M$ is an $(\bb{F}^I \lor \bb{F}^{\ast})$-martingale under $\pr_{\ast}$ and $\ev_{\ast}[\int_0^T H_t \, \d M_t \vert \F^0_T] = 0$ holds $\pr_{\ast}$-a.s.\@ for all bounded $\bb{F}^{\ast}$-predictable processes $H$.
\end{proposition}

\begin{proof}
We will again assume that $(\mu^N, W^0)_N$ converges along the entire sequence. To show that $M$ is an $(\bb{F}^I \lor \bb{F}^{\ast})$-martingale under $\pr_{\ast}$ it suffices to prove that $\ev_{\ast}[(M_t - M_s) \Phi(\Theta_{\cdot \land s})] = 0$ for all $\Phi \in C_b(\Omega_{\ast})$ and $0 \leq s \leq t \leq T$, where $\Theta_{\cdot \land s} = (X_{\cdot \land s}, W_{\cdot \land s}, \Gamma_s, I_{\cdot \land s}, \pi^{\#}_s \mu, B_{\cdot \land s})$, $\pi_s \define \cal{S} \to \cal{S}$ is given by $(x, w, \mathfrak{g}, p) \mapsto (x_{\cdot \land s}, w, \mathfrak{g}_s, p_{\cdot \land s})$, and $\mathfrak{g}_s$ is defined in \eqref{eq:measure_stop}. We will write the expectation $\ev_{\ast}[(M_t - M_s) \Phi(\Theta_{\cdot \land s})]$ as the limit of the corresponding expression $\frac{1}{N}\sum_{i = 1}^N \ev\bigl[(M^{N, i}_t - M^{N, i}_s) \Phi(\Theta^{N, i}_{\cdot \land s})\bigr]$ for the particle system and then use that $M^{N, 1}$,~\ldots, $M^{N, N}$ are martingales by Proposition \ref{prop:martingale_property} to conclude that the expectation vanishes. To make this work, we need to verify that $\frac{1}{N}\sum_{i = 1}^N \ev\bigl[(M^{N, i}_t - M^{N, i}_s) \Phi(\Theta^{N, i}_{\cdot \land s})\bigr]$ indeed converges to $\ev_{\ast}[(M_t - M_s) \Phi(\Theta_{\cdot \land s})]$ as $N \to \infty$. We achieve this in two steps. First, we will show that the function $\Theta_{\cdot \land s}$ is continuous at $\pr_{\ast}$-a.e.\@ $\omega \in \Omega_{\ast}$ and then that $\frac{1}{N} \sum_{i = 1}^N \L(M^{N, i}_u) \Rightarrow \L_{\ast}(M_u)$ for $u \in [0, T]$. The former, in conjunction with the continuous mapping theorem, implies that $\frac{1}{N}\sum_{i = 1}^N \L(\Phi(\Theta^{N, i}_{\cdot \land s})) \Rightarrow \L_{\ast}(\Phi(\Theta_{\cdot \land s}))$. Thus, passing to a subsequence if necessary, we find that
\begin{equation*}
    \frac{1}{N}\sum_{i = 1}^N \L\bigl((M^{N, i}_t - M^{N, i}_s) \Phi(\Theta^{N, i}_{\cdot \land s})\bigr) \Rightarrow \L_{\ast}\bigl((M_t - M_s) \Phi(\Theta_{\cdot \land s})\bigr)
\end{equation*}
This, together with the uniform boundedness of $\frac{1}{N}\sum_{i = 1}^N \ev\lvert M^{N, i}_u\rvert^2$ in $N \geq 1$ from Lemma \ref{lem:l_2_bound_ps} and the boundedness of $\Phi$ implies that
\begin{equation*}
    \ev_{\ast}[(M_t - M_s) \Phi(\Theta_{\cdot \land s})] = \lim_{N \to \infty} \frac{1}{N} \sum_{i = 1}^N \ev\bigl[(M^{N, i}_t - M^{N, i}_s) \Phi(\Theta^{N, i}_{\cdot \land s})\bigr] = 0,
\end{equation*}
where the last equality holds because $M^{N, i}$ is a martingale with respect to the filtration $\bb{F}^{N, \bf{I}}$ defined in \eqref{eq:filtration_ps} and $\Phi(\Theta^{N, i}_{\cdot \land s}) \in \F^{N, \bf{I}}_s$.

So let us show that $\Theta_{\cdot \land s}$ is continuous at $\pr_{\ast}$-a.e.\@ $\omega \in \Omega_{\ast}$. Note that the map $\Theta_{\cdot \land s}$ is continuous at any point $(x, w, \mathfrak{g}, p, m, b)$, for which $x$ and $p$ are continuous at $s$ and $m(C([0, T]) \times C([0, T]) \times \bb{M}_T(G) \times D^s_{[0, 1]}[0, T]) = 1$, where $D^s_{[0, 1]}[0, T]$ denotes the set of paths in $D_{[0, 1]}[0, T]$ that are continuous at $s$. To verify the last assertion, simply note the stopping map $\pi_s$ is continuous on
\begin{equation*}
    C([0, T]) \times C([0, T]) \times \bb{M}_T(G) \times D^s_{[0, 1]}[0, T] \subset \cal{S},
\end{equation*}
so by the continuous mapping theorem the pushforward $\pi^{\#}_s$ is continuous at distributions $m$ with $m(C([0, T]) \times C([0, T]) \times \bb{M}_T(G) \times D^s_{[0, 1]}[0, T]) = 1$. Let us check that the desired continuity properties hold $\pr_{\ast}$-almost surely. The $\pr_{\ast}$-a.s.\@ continuity of $X$ and $I$ at $s$ are established in Lemma \ref{lem:weak_conv_particle}. Next, since $\mu = \L_{\ast}(X, W, \Gamma, I \vert \F^0_T)$ we get that
\begin{equation*}
    \mu\Bigl(C([0, T]) \times C([0, T]) \times \bb{M}_T(G) \times D^s_{[0, 1]}[0, T])\Bigr) = \pr_{\ast}\bigl(X \in C([0, T]),\, I \in D^s_{[0, 1]}[0, T] \big\vert \F^0_T\bigr).
\end{equation*}
But the expectation of the right-hand side equals $1$ by Lemma \ref{lem:weak_conv_particle}, so $\pr_{\ast}$-a.s.\@ the probability on the right-hand side above is $1$. Consequently, the pushforward $\pi^{\#}_s$ is continuous at $\L_{\ast}(\mu)$-a.e.\@ element of $\P^2(\cal{S})$ and we can conclude that $\Theta_{\cdot \land s}$ is $\pr_{\ast}$-a.s.\@ continuous.

Next, we prove that $\frac{1}{N}\sum_{i = 1}^N \L(M^{N, i}_u) \Rightarrow \L_{\ast}(M_u)$ for $u \in [0, T]$. Since $\pr(\Delta I_u = 1) = 0$ for all $u \in [0, T]$ by Lemma \ref{lem:weak_conv_particle}, we have that $\frac{1}{N} \sum_{i = 1}^N \L(I^{N, i}_u) \Rightarrow \L_{\ast}(I_u)$. The integral $\int_0^{\cdot} I_u \lambda(u, X_u, \nu_u) \, \d u$ is a continuous function of $I$, $X$, and $\nu$, where $\nu$ is given by $\nu_u = \int_{\cal{S}} p_u \delta_{x_u} \, \d \mu(x, w, \mathfrak{g}, p)$. Thus, if we can show that $\L(\nu^N) \Rightarrow \L_{\ast}(\nu)$ on $D_{\bf{M}^2}[0, T]$, then it follows that $\frac{1}{N}\sum_{i = 1}^N \L(M^{N, i}_u) \Rightarrow \L_{\ast}(M_u)$. We know from Proposition \ref{prop:tightness_ps} that $(\nu^N)_{N \geq 1}$ is tight on $D_{\bf{M}^2}[0, T]$. Hence, to show that $\nu^N$ converges weakly to $\nu$, it suffices to show that the finite-dimensional marginals of $\nu^N$ tend to $\nu$. As above, we can show that the map $m \mapsto \int_{\cal{S}} p_u \delta_{x_u} \, \d m(x, w, \mathfrak{g}, p)$ is continuous at any $m$ for which $m(C([0, T]) \times C([0, T]) \times \bb{M}_T(G) \times D^u_{[0, 1]}[0, T]) = 1$. Since $\mu$ is supported on the set of such measures, the continuous mapping theorem yields $\nu^N_u = \int_{\cal{S}} p_u \delta_{x_u} \, \d \mu^N(x, w, \mathfrak{g}, p) \Rightarrow \int_{\cal{S}} p_u \delta_{x_u} \, \d \mu(x, w, \mathfrak{g}, p) = \nu_u$. This weak convergence is easily extended to any finite collection $0 \leq t_1 < \dots t_n \leq T$ of time points in $[0, T]$, so we can conclude that $\L(\nu^N)$ tends to $\L_{\ast}(\nu)$ on $D_{\bf{M}^2}[0, T]$. Consequently, we get the desired convergence $\frac{1}{N}\sum_{i = 1}^N \L(M^{N, i}_u) \Rightarrow \L_{\ast}(M_u)$. This finishes the proof of the first part of the proposition.

For the second statement, we note that it suffices to show that $\ev_{\ast}[\int_0^T H_t \, \d M_t \vert \F^0_T] = 0$ holds $\pr_{\ast}$-a.s.\@ for all bounded $\bb{F}^{\ast}$-predictable simple processes $H = \sum_{k = 0}^{n - 1} H^k \bf{1}_{(t_k, t_{k + 1}]}$ with $0 = t_0 < t_1 < \dots < t_n = T$ and $H^k = h_k(\tilde{\Theta}_{\cdot \land t_k})$ for $h_k \in C_b(\Omega_{\ast})$, where $\tilde{\Theta}_{\cdot \land t_k}$ arises from $\Theta_{\cdot \land t_k}$ by replacing $I_{\cdot \land t_k}$ with the process that is constantly equal to $1$. Indeed, these simple processes are dense in the space of all bounded $\bb{F}^{\ast}$-predictable processes with respect to the metric $(K^1, K^2) \mapsto \ev_{\ast} \int_0^T \lvert K^1_t - K^2_t\rvert^2 \, \d [M]_t$. To prove that $\ev_{\ast}[\int_0^T H_t \, \d M_t \vert \F^0_T] = 0$ for such simple processes $H$, we again appeal to the particle system. As above, we have for all $\Phi \in C_b(\Omega_0)$ that
\begin{align} \label{eq:simple_int_zero}
    \ev_{\ast}\biggl[\biggl(\int_0^T H_t \, \d M_t\biggr)\Phi(\Theta^0)\biggr] &=  \sum_{k = 0}^{n - 1}\ev_{\ast}\Bigl[h_k(\tilde{\Theta}_{\cdot \land t_k})(M_{t_{k + 1}} - M_{t_k})\Phi(\Theta^0)\Bigr] \notag \\
    &= \lim_{N \to \infty} \sum_{k = 0}^{n - 1} \frac{1}{N}\sum_{i = 1}^N \ev\Bigl[h_k\bigl(\tilde{\Theta}^{N, i}_{\cdot \land t_k}\bigr)\bigl(M^{N, i}_{t_{k + 1}} - M^{N, i}_{t_k}\bigr)\Phi(\Theta^{N, 0})\Bigr],
\end{align}
where $\tilde{\Theta}^{N, i}_{\cdot \land t_k}$ arises from $\Theta^{N, i}_{\cdot \land t_k}$ by replacing $I^{N, i}_{\cdot \land t_k}$ with the process that is constantly equal to $1$. Now, by Lemma \ref{lem:compensated_martingale_l_2}, $\frac{1}{N}\sum_{i = 1}^N \bigl(M^{N, i}_{t_{k + 1}} - M^{N, i}_{t_k}\bigr)$ tends to zero in $L^2$, while $h_k\bigl(\tilde{\Theta}^{N, i}_{\cdot \land t_k}\bigr)$ and $\Phi(\Theta^{N, 0})$ are uniformly bounded. Consequently, the expression on the right-hand side of \eqref{eq:simple_int_zero} equals zero. Hence, $\ev_{\ast}[(\int_0^T H_t \, \d M_t)\Phi(\Theta^0)] = 0$ for all $\Phi \in C_b(\Omega_0)$, which yields $\ev_{\ast}[\int_0^T H_t \, \d M_t \vert \F^0_T] = 0$ as required.
\end{proof}

The following result shows that $\nu$ is a $\pr_{\ast}$-a.s.\@ continuous function on $\Omega_{\ast}$ and has $\pr_{\ast}$-a.s.\@ continuous trajectories. This does not only hold if $\pr_{\ast}$ arises as the limit of a particle system, but as soon as $\pr_0$ satisfies suitable regularity properties.

\begin{proposition} \label{prop:id_of_nu}
Let $\pr_0 \in \P(\Omega_0)$ be such that $X$ has $\pr_{\ast}$-a.s.\@ continuous trajectories and $\ev_{\ast}[\int_0^T H_t \, \d M_t \vert \F^0_T] = 0$ holds $\pr_{\ast}$-a.s.\@ for all bounded $\bb{F}^{\ast}$-predictable processes $H$. Then $\nu$ is continuous at $\pr_{\ast}$-a.e.\@ $\omega \in \Omega_{\ast}$ and has $\pr_{\ast}$-a.s.\@ continuous trajectories.
\end{proposition}

\begin{proof}
Let $(\tilde{X}, \tilde{W}, \tilde{\Gamma}, \tilde{I})$ denote the canonical random element on $\cal{S}$ and define $\cal{V} \define \P^2(\cal{S}) \to D_{\bf{M}^2}[0, T]$ by $\cal{V}(m) = \int_{\cal{S}} p_t \delta_{x_t} \, \d m(x, w, \mathfrak{g}, b)$ for $m \in \P^2(\cal{S})$ and $t \in [0, T]$, so that $\nu_t = \cal{V}_t(\mu)$. By assumption, to show that $\nu$ is $\pr_{\ast}$-a.s.\@ continuous on $\Omega_{\ast}$, it suffices to show that $\cal{V}$ is continuous at every $m \in \P^2(\cal{S})$ for which $\tilde{X}$ has $m$-a.s.\@ continuous trajectories and  
\begin{equation} \label{eq:compensate}
    \ev_m[\tilde{I}_t - \tilde{I}_s] = - \int_s^t \ev_m\bigl[\tilde{I}_u \lambda(u, \tilde{X}_u, \cal{V}_u(m))\bigr] \, \d u
\end{equation}
for all $0 \leq s \leq t \leq T$, where $\ev_m$ denotes the expectation on $\cal{S}$ under $m$. Indeed, $\mu$ is supported on measures of this type according to the assumptions of the proposition.

Now, let us define $t_k = kT/K$ for $k = 0$,~\ldots, $K$ and $K \geq 1$, and let $(m^n)_{n \geq 1}$ be a sequence in $\P^2(\cal{S})$ that converges to $m$. Then Equation \eqref{eq:compensate} implies that $m(\Delta \tilde{I}_t = 1) = 0$ for $t \in [0, T]$, so we can find a coupling $\pi^n$ of $m^n$ and $m$, such that $\ev_{\pi^n}[\lvert \tilde{I}^1_{t_k} - \tilde{I}^2_{t_k}\rvert + (\lvert \tilde{X}^1 - \tilde{X}^2\rvert^{\ast}_T)^2] \to 0$ for all $k \in \{0, \dots, K\}$. Here $(\tilde{X}^i, \tilde{W}^i, \tilde{\Gamma}^i, \tilde{I}^i)_{i = 1, 2}$ is the canonical random element on $\cal{S}^2$. Next, setting $v^n_t = \cal{V}_t(m^n)$ and $v_t = \cal{V}_t(m)$, we estimate similarly to \eqref{eq:split_of_metric},
\begin{align} \label{eq:d_1_bound_v}
    \sup_{0 \leq t \leq T} d_1(v^n_t, v_t) &\leq \sup_{0 \leq t \leq T}\Bigl(W_1\Bigl(v^n_t + (1 - v^n_t(\R))\delta_0, v_t + (1 - v_t(\R))\delta_0\Bigr) + \lvert v^n_t(\R) - v_t(\R)\rvert\Bigr) \notag \\
    &\leq \sup_{0 \leq t \leq T} \ev_{\pi^n}\lvert \tilde{I}^1_t\tilde{X}^1_t - \tilde{I}^2_t\tilde{X}^2_t\rvert + \sup_{0 \leq t \leq T}\ev_{\pi^n}\lvert\tilde{I}^1_t - \tilde{I}^2_t\rvert \notag \\
    &\leq \sup_{0 \leq t \leq T} \ev_{\pi^n}\lvert \tilde{X}^1_t - \tilde{X}^2_t\rvert + \sup_{0 \leq t \leq T}\ev_{\pi^n}\bigl[(1 + \lvert \tilde{X}^2_t\rvert) \lvert \tilde{I}^1_t - \tilde{I}^2_t\rvert\bigr].
\end{align}
The first term on the right-hand side tends to zero as $n \to \infty$ by the choice of the coupling $\pi^n$. We deal with the second expression next. By the Cauchy-Schwarz inequality, we have
\begin{align*}
    \sup_{0 \leq t \leq T}\ev_{\pi^n}\bigl[(1 + \lvert \tilde{X}^2_t\rvert) \lvert \tilde{I}^1_t - \tilde{I}^2_t\rvert\bigr]^2 \leq C^2\sup_{0 \leq t \leq T}\ev_{\pi^n}\lvert\tilde{I}^1_t - \tilde{I}^2_t\rvert,
\end{align*}
where $C^2 = \sup_{0 \leq t \leq T} \ev_m[(1 + \lvert \tilde{X}_t\rvert)^2]$. To bound the expression on the right-hand side, we use that
\begin{align*}
    \lvert \tilde{I}^1_t - \tilde{I}^2_t\rvert &\leq \bigl\lvert \tilde{I}^1_{t_k} - \tilde{I}^2_{t_{k + 1}}\bigr\rvert + \bigl\lvert\tilde{I}^1_{t_{k + 1}} - \tilde{I}^2_{t_k}\bigr\rvert \\
    &\leq \bigl\lvert\tilde{I}^1_{t_k} - \tilde{I}^2_{t_k}\bigr\rvert + \bigl\lvert\tilde{I}^1_{t_{k + 1}} - \tilde{I}^2_{t_{k + 1}}\bigr\rvert + 2\bigl\lvert\tilde{I}^2_{t_{k + 1}} - \tilde{I}^2_{t_k}\bigr\rvert
\end{align*}
whenever $t \in [t_k, t_{k + 1}]$. This implies that
\begin{align*}
    \sup_{0 \leq t \leq T}\ev_{\pi^n}\lvert\tilde{I}^1_t - \tilde{I}^2_t\rvert \leq 2\sup_{0 \leq k \leq K} \ev_{\pi^n}\bigl\lvert \tilde{I}^1_{t_k} - \tilde{I}^2_{t_k}\bigr\rvert + 2\sup_{0 \leq k \leq K - 1}\int_{t_k}^{t_{k + 1}} \ev_m\bigl[\tilde{I}_t \lambda(t, \tilde{X}_t, v_t)\bigr] \, \d t,
\end{align*}
where we made use of \eqref{eq:compensate}. The first summand on the right-hand side vanishes as $n \to \infty$ due to our choice of $\pi^n$. Hence, taking the limit superior as $n \to \infty$ on both sides of \eqref{eq:d_1_bound_v} implies that
\begin{equation*}
    \limsup_{n \to \infty} \biggl(\sup_{0 \leq t \leq T} d_1(v^n_t, v_t)\biggr) \leq 2C\sup_{0 \leq k \leq K - 1}\biggl(\int_{t_k}^{t_{k + 1}} \ev_m\bigl[\tilde{I}_t \lambda(t, \tilde{X}_t, v_t)\bigr] \, \d t\biggr)^{1/2}.
\end{equation*}
The right-hand side in turn tends to zero as $K \to \infty$. Consequently, we obtain that $\sup_{0 \leq t \leq T} d_1(v^n_t, v_t) \to 0$ as $n \to \infty$. This establishes continuity at $v$ with respect to the uniform distance on $D_{\bf{M}^2}[0, T]$, where $\bf{M}^2$ is equipped with the metric $d_1$. The continuity can be upgraded to the distance $d_2$ on $\bf{M}^2$ by Lemma \ref{lem:weak_conv_upgrade}. Here the condition \eqref{eq:uniform_square} is satisfied because $(m^n)_n$ is uniformly square-integrable as a convergent sequence in $\P^2(\cal{S})$.

Next, let us discuss the continuity of the trajectories of $\nu$ with respect to $d_2$. Proceeding similarly to \eqref{eq:d_1_bound_v}, we find
\begin{align*}
    d_1(\nu_t, \nu_s) &= W_1\Bigl(\nu_t + (1 - \nu_t(\R))\delta_0, \nu_s + (1 - \nu_s(\R))\delta_0\Bigr) + \lvert L_t - L_s\rvert \\
    &\leq \ev[\lvert I_t X_t - I_s X_s\rvert \vert \F^0_T] + \ev[\lvert I_t - I_s\rvert \vert \F^0_T] \\
    &\leq \ev[\lvert X_t - X_s \rvert \vert \F^0_T] + \ev[(1 + \lvert X_s\rvert) \lvert I_t - I_s \rvert \vert \F^0_T] \\
    &\leq \ev[\lvert X_t - X_s \rvert \vert \F^0_T] + \bigl(\ev\bigl[(1 + \lvert X\rvert^{\ast}_T)^2 \big\vert \F^0_T\bigr]\bigr)^{1/2} \biggl(\int_s^t \ev[I_u \lambda(u, X_u, \nu_u) \vert \F^0_T] \, \d u\biggr)^{1/2}.
\end{align*}
The first term on the right-hand side converges to zero $\pr_{\ast}$-a.s.\@ as $\lvert s - t\rvert \to 0$ by the $\pr_{\ast}$-a.s.\@ continuity of $X$ and square-integrability of $\lvert X \rvert^{\ast}_T$. The latter also implies that the integral with respect to time on the right-hand side vanishes $\pr_{\ast}$-a.s.\@ as $\lvert s - t\rvert \to 0$. Consequently, $\nu$ has $\pr_{\ast}$-a.s.\@ $d_1$-continuous trajectories. From this and the $\pr_{\ast}$-a.s.\@ uniform square-integrability of the family $(\nu_t)_{t \in [0, T]}$, we can conclude that $\nu$ has $\pr_{\ast}$-a.s.\@ continuous trajectories with respect to $d_2$. 
\end{proof}

\subsection{The Controlled Martingale Problem} \label{sec:mgale_problem}

In this subsection, we define a controlled martingale problem associated to the state equation \eqref{eq:relaxed_weak_mfl} of the relaxed formulation of the mean-field control problem from Definition \ref{def:rel_cntrl}. Subsequently, we prove that subsequential limits of the particle system solve the martingale problem and, hence, yield admissible relaxed control rules. Let us fix admissible controls $\gamma^1$,~\ldots, $\gamma^N$.

We introduce the differential operator $\bb{L}$ acting on twice continuously differentiable functions $\varphi \define \R^3 \to \R$ by
\begin{align*}
    \cal{L}\varphi(t, x, y&, z, v, g) \\
    &= \bigl(b(t, x, v, g) - \alpha(t, x, v) \langle v, \lambda(t, \cdot, v)\rangle\bigr) \partial_x \varphi(x, y, z) \\
    &\ \ \ + a(t, x, v) \partial_x^2 \varphi(x, y, z) + \sigma(t, x, v) \partial_{xy}^2 \varphi(x, y, z)+ \sigma_0(t, x, v) \partial_{xz}^2 \varphi(x, y, z) \\
    &\ \ \ + \frac{1}{2} \partial_y^2 \varphi(x, y, z) + \frac{1}{2} \partial_z^2 \varphi(x, y, z)
\end{align*}
for $(t, x, y, z, v, g) \in [0, T] \times \R^3 \times \cal{M}^1_{\leq 1}(\R) \times G$ with $a(t, x, v) = \frac{1}{2}(\sigma^2(t, x, v) + \sigma_0^2(t, x, v))$. The operator $\bb{L}$ is the infinitesimal generator of the process $(X, W, B)$ under an admissible relaxed control rule $\pr_0 \in \P(\Omega_0)$ (see Definition \ref{def:rel_cntrl}). Next, for any $\varphi \in C^2_c(\R^3)$ we define the process $\cal{M}^{\varphi}$ on $\Omega_{\ast}$ by
\begin{equation} \label{eq:mgale}
    \cal{M}^{\varphi}_t(\omega) = \varphi(x_t, w_t, b_t) - \int_{[0, t] \times G} \bb{L}\varphi(s, x_s, w_s, b_s, \nu_t(\omega), g) \, \d \mathfrak{g}(s, g)
\end{equation}
for $\omega = (x, w, \mathfrak{g}, p, m, b) \in \Omega_{\ast}$. We prove that $\cal{M}^{\varphi}$ is a martingale under any probability measure $\pr_{\ast}$ induced by a subsequential limit of $(\mu^N, W^0)_N$.

\begin{proposition} \label{prop:conv_to_ad_re_cr}
Let Assumption \ref{ass:red_form_cont_prob} be satisfied. Then for any $\varphi \in C^2_c(\R^3)$, the process $\cal{M}^{\varphi}$ is an $\bb{F}^{\ast}$-martingale under $\pr_{\ast}$ for any subsequential limit $\pr_0$ of $(\mu^N, W^0)_N$. In particular, the limit $\pr_0$ is an admissible relaxed control rule in the sense of Definition \ref{def:rel_cntrl}.
\end{proposition}

\begin{proof}
As usual, we assume that $(\mu^N, W^0)_N$ converges to $\pr_0$ along the whole sequence. We proceed similarly as in the proof of Proposition \ref{prop:compensator_limit}, showing that for any $\Phi \in C_b(\Omega_{\ast})$ and $0 \leq s \leq t \leq T$, the expectation $\ev_{\ast}\bigl[\bigl(\cal{M}^{\varphi}_t(\Theta) - \cal{M}^{\varphi}_s(\Theta)\bigr)\Phi(\Theta_{\cdot \land s})\bigr]$ vanishes. In the proof of Proposition \ref{prop:compensator_limit}, we already established that the map $\omega \mapsto \Theta_{\cdot \land s}(\omega)$ is continuous at $\pr_{\ast}$-a.e.\@ $\omega \in \Omega_{\ast}$. To show the same is true for $\cal{M}^{\varphi}_t(\Theta) - \cal{M}^{\varphi}_s(\Theta)$, we simply choose $E_1 = E_2 = E_3 = \R$, $E_4 = \bf{M}^2 = \cal{M}^2_{\leq 1}(\R)$, and $\Phi = \bb{L}\varphi$ in Corollary \ref{cor:mgale_integral_converges}. This together with the $\pr_{\ast}$-a.s.\@ continuity of $\nu$ on $\Omega_{\ast}$ from Proposition \ref{prop:id_of_nu} implies the continuity of the integral operator appearing in the definition \eqref{eq:mgale} of $\cal{M}^{\varphi}_t$. The continuity of the first summand in \eqref{eq:mgale} is clear, so $\cal{M}^{\varphi}_t$ and, similarly, $\cal{M}^{\varphi}_s$ are $\pr_{\ast}$-a.s.\@ continuous in $\omega \in \Omega_{\ast}$. Hence, by the continuous mapping theorem, we find that
\begin{equation*} \label{eq:conv_mgale_ps}
    \frac{1}{N}\sum_{i = 1}^N \L\Bigl(\bigl(\cal{M}^{\varphi}_t(\Theta^{N, i}) - \cal{M}^{\varphi}_s(\Theta^{N, i})\bigr)\Phi(\Theta^{N, i}_{\cdot \land s})\Bigr) \Rightarrow \L_{\ast}\Bigl(\bigl(\cal{M}^{\varphi}_t(\Theta) - \cal{M}^{\varphi}_s(\Theta)\bigr)\Phi(\Theta_{\cdot \land s})\Bigr)
\end{equation*}
by Lemma \ref{lem:weak_conv_particle}. Next, since $\Omega_{\ast} \ni \omega \mapsto \cal{M}^{\varphi}_u(\omega)$ is of linear growth for any $u \in [0, T]$ and $\Phi$ is bounded, while
\begin{equation*}
    \sup_{N \geq 1} \frac{1}{N}\sum_{i = 1}^N \ev\biggl[\sup_{0 \leq u \leq T} \bigl(\lvert X^{N, i}_u \rvert^2 + M_2^2(\nu^N_u)\bigr)\biggr] < \infty
\end{equation*}
by Lemma \ref{lem:l_2_bound_ps}, where $M_2(v)$ denotes the second moment of $v \in \bf{M}^2$, we deduce from \eqref{eq:conv_mgale_ps} that
\begin{equation*}
    \frac{1}{N}\sum_{i = 1}^N \ev \Bigl[\bigl(\cal{M}^{\varphi}_t(\Theta^{N, i}) - \cal{M}^{\varphi}_s(\Theta^{N, i})\bigr)\Phi(\Theta^{N, i}_{\cdot \land s})\Bigr] \to \ev_{\ast}\bigl[\bigl(\cal{M}^{\varphi}_t(\Theta) - \cal{M}^{\varphi}_s(\Theta)\bigr)\Phi(\Theta_{\cdot \land s})\bigr].
\end{equation*}
Thus, to conclude that $\ev_{\ast}\bigl[\bigl(\cal{M}^{\varphi}_t(\Theta) - \cal{M}^{\varphi}_s(\Theta)\bigr)\Phi(\Theta_{\cdot \land s})\bigr] = 0$, it suffices to show that the left-hand side above vanishes in the limit $N \to \infty$. Applying It\^o's formula for jump diffusions to $\varphi(X^i, W^i, W^0)$ shows that
\begin{align} \label{eq:ito_for_mgale}
\begin{split}
    \cal{M}^{\varphi}_t&(\Theta^{N, i}) - \cal{M}^{\varphi}_s(\Theta^{N, i})\\
    &= \int_s^t \bigl(\sigma(u, X^i_u, \nu^N_u) \partial_x \varphi(X^i_u, W^i_u, W^0_u) + \partial_y \varphi(X^i_u, W^i_u, W^0_u)\bigr)\, \d W^i_u \\
    &\ \ \ + \int_s^t \bigl(\sigma_0(u, X^i_u, \nu^N_u) \partial_x \varphi(X^i_u, W^i_u, W^0_u) + \partial_z \varphi(X^i_u, W^i_u, W^0_u)\bigr) \, \d W^0_u \\
    &\ \ \ + \frac{1}{N}\sum_{j = 1}^N \int_s^t \alpha(u, X^i_{u-}, \nu^N_{u-}) \partial_x\varphi(X^i_{u-}, W^i_u, W^0_u) \, \d M^j_u + E^{\varphi,i}_t - E^{\varphi,i}_s,
\end{split}
\end{align}
where the error term $E^{\varphi,i}$ is given by $E^{\varphi,i}_t = \sum_{0 \leq u \leq t} \Delta \varphi(X^i_u) - \partial_x \varphi(X^i_{u-}) \Delta X^i_u$. By Proposition \ref{prop:martingale_property}, the process
\begin{equation*}
    \frac{1}{N}\sum_{j = 1}^N \int_0^{\cdot} \alpha(u, X^i_{u-}, \nu^N_{u-}) \partial_x\varphi(X^i_{u-}, W^i_u, W^0_u) \, \d M^j_u
\end{equation*}
is a martingale with respect to the filtration $\bb{F}^{N, \bf{I}}$ defined in Equation \eqref{eq:filtration_ps}. By the assumptions on $\sigma$ and $\sigma_0$, the same is true for the integrals with respect to the Brownian motions $W^i$ and $W^0$. Since $\Phi(\Theta^{N, i}_{\cdot \land s})$ is $\F^{N, \bf{I}}_s$-measurable, multiplying both sides of \eqref{eq:ito_for_mgale} by $\Phi(\Theta^{N, i}_{\cdot \land s})$, taking expectation, and summing over $i = 1$,~\ldots, $N$, we obtain
\begin{equation*}
    \frac{1}{N}\sum_{i = 1}^N \ev \Bigl[\bigl(\cal{M}^{\varphi}_t(\Theta^{N, i}) - \cal{M}^{\varphi}_s(\Theta^{N, i})\bigr)\Phi(\Theta^{N, i}_{\cdot \land s})\Bigr] = \frac{1}{N}\sum_{i = 1}^N \ev\bigl[(E^{\varphi,i}_t - E^{\varphi,i}_s)\Phi(\Theta^{N, i}_{\cdot \land s})\bigr].
\end{equation*}
To show that the right-hand side converges to zero, we prove that the error terms $E^{\varphi,i}_t - E^{\varphi,i}_s$ are of order $\frac{1}{N}$. Using Taylor's theorem together with $\Delta X^i_u = - \alpha(u, X^i_{u-}, \nu^N_{u-}) \Delta L^N_u$ yields
\begin{align*}
    E^{\varphi,i}_t - E^{\varphi,i}_s = \frac{1}{2}\sum_{s \leq u \leq t} \alpha^2(u, X^i_{u-}, \nu^N_{u-}) \partial_x^2 \varphi(\zeta^i_u) (\Delta L^N_u)^2
\end{align*}
for some $\zeta^i_u \in [X^i_{u-}, X^i_u]$. Since $L^N$ jumps at most $N$ times and each jump has size $1/N$, this yields $\lvert E^{\varphi,i}_t - E^{\varphi,i}_s\rvert \leq \frac{C_{\alpha}^2}{2} \lVert \partial_x^2 \varphi \rVert_{\infty} N^{-1}$ as desired. Consequently, $\cal{M}^{\varphi}$ is a martingale under $\pr_{\ast}$.

Lastly, we verify that $\pr_0$ is an admissible relaxed control rule. First, note that $\cal{M}^{\varphi}$ is $\pr_{\ast}$-a.s.\@ continuous. Indeed, the integral appearing in Equation \eqref{eq:mgale} varies continuously with $t$, since the time marginal of $\Gamma$ is the Lebesgue measure. Moreover, the process $X$ is $\pr_{\ast}$-a.s.\@ continuous by Lemma \ref{lem:weak_conv_particle}, so the same holds for $t \mapsto \varphi(X_t, W_t, W^0_t)$. 
Then, combining the martingale property of $\cal{M}^{\varphi}$ with Theorem 4.5.2 from \cite{stroock_diffusion_2006}, we conclude that $X$ satisfies the McKean--Vlasov SDE \eqref{eq:relaxed_weak_mfl} under $\pr_{\ast}$. Note that in contrast to \cite[Theorem 4.5.2]{stroock_diffusion_2006} we do not need to enlarge the filtered probability space $(\Omega_{\ast}, \F^{\ast}_T, \bb{F}^{\ast}, \pr_{\ast})$ because all the randomness can be expressed through the Brownian motions $W$ and $B$ on $\Omega_{\ast}$ without the need for an external source of randomness. This together with Proposition \ref{prop:id_of_nu} gives Item \ref{it:sde} in Definition \ref{def:rel_cntrl}. The identity $\L_{\ast}(X_0) = \nu_0$ from Item \ref{it:law_integrability} is obvious and the bound on the second moment of $X$ under $\pr_{\ast}$ follows from Lemma \ref{lem:l_2_bound_ps}. Item \ref{it:independence} is a consequence of Lemma \ref{lem:weak_conv_particle} and the stability of independence under weak convergence, see \cite[Proposition 4.17]{mkv_control_limit_2020} for details. Item \ref{it:martingale} was addressed in Proposition \ref{prop:compensator_limit}. This concludes the proof.
\end{proof}

We have proved that any subsequential limit $\pr_0$ of $(\mu^N, W^0)_N$ converges to an admissible relaxed control rule. Next, we show that the associated costs converge as well.

\begin{proposition} \label{prop:conv_of_cost}
Let Assumption \ref{ass:red_form_cont_prob} be satisfied. Let $(\pr_0^n)_n$ be a sequence of probability measures on $\Omega_0$ that converges to a relaxed control rule $\pr_0 \in \P(\Omega_0)$. Assume further that the family $((X, \nu)^{\#}\pr_{\ast}^n)_n$ is uniformly square-integrable, where $\pr_{\ast}^n$ is the probability measure on $\Omega_{\ast}$ associated to $\pr^n_0$. Then we have $\lim_{n \to \infty} \cal{J}(\pr_0^n) = \cal{J}(\pr_0)$.
\end{proposition}

Here by uniform square-integrability of the laws $((X, \nu)^{\#}\pr_{\ast}^n)_n$ we mean that the family of laws of $\sup_{0 \leq t \leq T} \lvert X_t\rvert$ and $\sup_{0 \leq t \leq T} M_2(\nu_t)$ under $\pr_{\ast}^n$ are uniformly square-integrable.

\begin{proof}
We apply Corollary \ref{cor:mgale_integral_converges} with $E_1 = \R$, $E_2 = [0, 1]$, $E_3 = \bf{M}^2$, and $\Phi(t, x, p, v, g) = p f(t, x, v, g)$ to see that the law of $F = \int_{[0, T] \times G} I_t f(t, X_t, \nu_t, g) \, \d \Gamma(t, g)$ under $\pr_{\ast}^n$ converges weakly to $\L_{\ast}(F)$. Here, we use the continuity of $f$ in $(x, \nu, g)$ from Assumption \ref{ass:red_form_cont_prob} \ref{it:continuity_cost} and the $\pr_{\ast}$-a.s.\@ continuity of the trajectories of $\nu$ with respect to the uniform-topology on $D_{\bf{M}^2}[0, T]$ guaranteed by Proposition \ref{prop:id_of_nu}, which holds due to the admissibility of $\pr_0$. Then, since the family $((X, \nu)^{\#}\pr_{\ast}^n)_n$ is uniformly square-integrable and the running cost function $f$ has at most quadratic growth in its last three variables uniformly in $t \in [0, T]$ by Assumption \ref{ass:red_form_cont_prob} \ref{it:growth_cost}, the laws $F^{\#} \pr_{\ast}^n$, $n \geq 1$, are uniformly integrable. Consequently, we get that $\ev^{\pr_{\ast}^n}F \to \ev_{\ast} F$ as $n \to \infty$. 

Next, let us analyse the terminal cost. From the above, we know that the law of $\nu_T$ under $\pr_{\ast}^n$ converges to the its law under $\pr_{\ast}$. Now, we note again that the family $(\nu^{\#}\pr_{\ast}^n)_n$ is uniformly square-integrable and that the terminal cost function $\psi$ has at most quadratic growth and is continuous by Assumptions \ref{ass:red_form_cont_prob} \ref{it:growth_cost} and \ref{it:continuity_cost}, so that $\ev^{\pr_{\ast}^n}\psi(\nu_T) \to \ev_{\ast}\psi(\nu_T)$ as $n \to \infty$. Combining the convergence of both running and terminal cost yields $\lim_{n \to \infty} \cal{J}(\pr_0^n) = \cal{J}(\pr_0)$.
\end{proof}

We are now in a position to prove Theorem \ref{thm:convergence_ps}.

\begin{proof}[Proof of Theorem \ref{thm:convergence_ps}]
Let $(\gamma^{N, 1}, \dots, \gamma^{N, N})_{N \geq 1}$ be a sequence of admissible controls for the particle system. Then, by Proposition \ref{prop:tightness_ps}, the sequence $(\mu^N, W^0)_{N \geq 1}$ of particle systems is tight on $\Omega_0$ and Proposition \ref{prop:conv_to_ad_re_cr} guarantees that any subsequential limit $\pr_0$ is an admissible relaxed control rule in the sense of Definition \ref{def:rel_cntrl}. Lastly, Proposition \ref{prop:conv_of_cost} implies that along the convergent subsequence $(\mu^{N_k}, W^0)_{k \geq 1}$, we have
\begin{equation*}
    \lim_{k \to \infty} J^{N_k}(\gamma^{N_k, 1}, \dots, \gamma^{N_k, N_k}) = \lim_{k \to \infty} \cal{J}(\pr^{N_k}_0) = \cal{J}(\pr_0) \geq V_{\textup{rl}},
\end{equation*}
where $\pr^N_0 = \L(\mu^N, W^0)$. Since this is true for any convergent subsequence, it follows that $\liminf_{N \to \infty} V^N \geq V_{\textup{rl}}$. This concludes the proof.
\end{proof}

\subsection{Equivalence between Strong and Smooth Relaxed Mean-Field Control problem}

In this subsection, we establish the equivalence between the strong formulation of the mean-field control problem and the relaxed formulation with smooth control rules. 

\subsubsection{Existence, Uniqueness, and Stability of the Mean-Field Limit} \label{sec:existence_uniqueness}

We introduce an alternative formulation of the smooth relaxed control setup from Definition \ref{def:rel_cntrl}. Our goal is to then apply the theory in \cite{mkv_control_limit_2020}, which establishes the equivalence of the strong and smooth relaxed formulation for mean-field control. \cite{mkv_control_limit_2020} proves the equivalence of both formulations for McKean-Vlasov SDEs with Lipschitz continuous coefficients, whereas in our case the coefficients are only locally Lipschitz continuous in the measure argument. The key properties \cite{mkv_control_limit_2020} requires for their proof are existence, uniqueness, and stability in the control argument of the state equation. We show that these properties also hold for locally Lipschitz coefficients.

In the following, let us fix a probability space $(\Omega, \F, \pr)$ equipped with two filtrations $\bb{G}$ and $\bb{F}$ with $\cal{G}_t \subset \F_t$ for $t \in [0, T]$, an $\F_0$-measurable random variable $\xi$ with finite second moment, and two $\bb{F}$-Brownian motions $W$ and $W^0$. We assume that $W^0$ is adapted to $\bb{G}$ and that the pair $(\xi, W)$ is independent of $\cal{G}_T$. Finally, we require that for all $t \in [0, T]$ we have $\pr(A \vert \cal{G}_t) = \pr(A \vert \cal{G}_T)$ a.s.\@ for all $A \in \F_t \lor \sigma(W)$.

We say that an $\bb{M}_T(G)$-valued random variable $\Gamma$ is $\bb{F}$\textit{-progressively measurable} if for all $t \in [0, T]$, the random variable $\Gamma([0, s] \times A)$ is $\F_t$-measurable for any $s \in [0, t]$ and $A \in \cal{B}(G)$. The set of \textit{admissible smooth relaxed controls} consists of all $\bb{F}$-progressively measurable $\bb{M}_T(G)$-valued random variables $\Gamma$. For a given admissible smooth relaxed control $\Gamma$, we consider the McKean--Vlasov SDE
\begin{align} \label{eq:relaxed_weak_mfl_2}
\begin{split}
    X_t = \xi &+ \int_0^t b(s, X_s, \nu_s, g) \, \d \Gamma(s, g) + \int_0^t \sigma(s, X_s, \nu_s) \, \d W_s \\
    &+ \int_0^t \sigma_0(s, X_s, \nu_s) \, \d W^0_s - \int_0^t \alpha(s, X_s, \nu_s)\langle \nu_s, \lambda(s, \cdot, \nu_s)\rangle \, \d s
\end{split}
\end{align}
with $\nu_t = \ev[e^{-\Lambda_t} \delta_{X_t} \vert \cal{G}_T]$ and $\Lambda_t = \int_0^t \lambda(s, X_s, \nu_s) \, \d s$.

When say that the pair $(X, \nu)$ is a \textit{strong solution} of McKean--Vlasov SDE \eqref{eq:relaxed_weak_mfl_2} if (i) the process $X$ is a strong solution to \eqref{eq:relaxed_weak_mfl_2} when viewed as an SDE with random coefficients, the randomness coming from the mean-field component $\nu$ and the control $\Gamma$, and (ii) $\nu_t$ is the conditional subprobability distribution of $X_t$ with respect to $\cal{G}_T$, so there is no additional external information in the conditioning.

\begin{proposition} \label{prop:existence_uniqueness_stability}
Let Assumption \ref{ass:red_form_cont_prob} be satisfied. For any smooth relaxed control $\Gamma$, McKean--Vlasov SDE \eqref{eq:relaxed_weak_mfl_2} has a unique strong solution. If $b$ does not depend on the control, i.e.\@ $b(t, x, v, g) = b_0(t, x, v)$ for some function $b_0 \define [0, T] \times \R \times \cal{M}^1_{\leq 1}(\R) \to \R$, then $\nu_t = \ev[e^{-\Lambda_t} \delta_{X_t} \vert W^0]$ for all $t \in [0, T]$ almost surely.

Moreover, if $(\epsilon_n)_n$ is a sequence of positive real numbers tending to zero and $(\Gamma_n)_{n \geq 1}$ is a sequence of smooth relaxed controls such that $\ev W^2_2(\Gamma_n, \Gamma) \to 0$, then $\ev(\lvert X^n - X \rvert^{\ast}_T)^2 \to 0$ as $n \to \infty$. Here $X^n$ is the unique strong solution to SDE \eqref{eq:relaxed_weak_mfl_2} started from $\xi$ at time $\epsilon_n$ with control $\Gamma_n$. 
\end{proposition}

In the statement of the proposition starting $X^n$ from $\xi$ at time $\epsilon_n$ means that $X^n$ solves the SDE \eqref{eq:relaxed_weak_mfl_2} on the interval $[\epsilon_n, T]$ with initial condition $X^n_{\epsilon_n} = \xi$.

\begin{proof}[Proof of Proposition \ref{prop:existence_uniqueness_stability}]
Both statements follow from a simple application of Proposition \ref{prop:local_lipschitz}. To get SDE \eqref{eq:relaxed_weak_mfl_2} into the same form as SDE \eqref{eq:local_lipschitz}, we replace the coefficients $b$, $\sigma$, and $\sigma_0$ of the latter by the functions
\begin{align} \label{eq:coeff_redef}
\begin{split}
    (t, (x, y)&, m, g) \mapsto 
    \begin{pmatrix}
    b(t, x, \Phi(m), g) - \alpha(t, x, \Phi(m)) \langle \Phi(m), \lambda(t, \cdot, \Phi(m))\rangle \\
    \lambda(t, x, \Phi(m))
    \end{pmatrix},
    \\
    &(t, (x, y), m) \mapsto 
    \begin{pmatrix}
    \sigma(t, x, \Phi(m)) \\
    0
    \end{pmatrix},
    \quad (t, (x, y), m) \mapsto 
    \begin{pmatrix}
    \sigma_0(t, x, \Phi(m)) \\
    0
    \end{pmatrix},
\end{split}
\end{align}
where $\Phi \define \P^2(\R^2) \to \cal{M}^1_{\leq 1}(\R)$ is defined by $\Phi(m) = \int_{\R^2} e^{-(y \lor 0)} \delta_x \, \d m(x, y)$. 
Now, one simply needs to verify Assumption \ref{ass:local_lipschitz} for the coefficients in \eqref{eq:coeff_redef}. We will not provide the details here.
\end{proof}

\begin{corollary} \label{cor:approx_contr}
Let Assumption \ref{ass:red_form_cont_prob} be satisfied. Fix a smooth relaxed control $\Gamma$ and denote the associated solution to the McKean--Vlasov SDE \eqref{eq:relaxed_weak_mfl_2} by $X$. Let $n \geq 1$ and define $t^n_i = T \frac{i}{n}$ for $i = 0$,~\ldots $n$. Then there exists a sequence of bounded $\bb{F}$-progressively measurable $G$-valued processes $(\gamma^n)_n$ with the following properties
\begin{enumerate}[noitemsep, label = (\roman*)]
    \item $\gamma^n_0 = g_0$ for some $g_0 \in G$ and $\gamma^n$ is constant on the intervals $[t^n_i, t^n_{i + 1})$, $i = 0$,~\ldots, $n - 1$;
    \item $\ev W_2^2(\Gamma_n, \Gamma) \to 0$ as $n \to \infty$.
\end{enumerate}
Here $\Gamma_n$ is defined by $\d \Gamma_n(t, g) = \d \delta_{\gamma^n_t}(g) \d t$. In particular, it holds that $\ev(\lvert X^n - X \rvert^{\ast}_T)^2 \to 0$ as $n \to \infty$, where $X^n$ is the unique strong solution to SDE \eqref{eq:relaxed_weak_mfl_2} started from $\xi$ at time $t^n_1$ with control $\Gamma_n$. 
\end{corollary}

\begin{proof}
First, we reduce the problem to controls $\Gamma$ of the form $\d \Gamma(t, g) = \delta_{\gamma_t}(g) \d t$ for an $\bb{F}$-progressively measurable $G$-valued processes $\gamma$. If $\Gamma$ is not of that form, we can find a sequence $(\gamma^n)_n$ of $\bb{F}$-progressively measurable $G$-valued processes with associated measure $\Gamma_n$, for which almost surely $\lim_{n \to \infty} W_2(\Gamma_n, \Gamma) = 0$, so that $\ev W_2^2(\Gamma_n, \Gamma) \to 0$ by the dominated convergence theorem. Indeed, the former is guaranteed by the chattering lemma (see e.g.\@ \cite[Theorem 2.2(b)]{karoui_filter_1988}), which is applicable since $G$ is convex by Assumption \ref{ass:red_form_cont_prob}.

The second approximation step from piecewise constant to $\bb{F}$-progressively measurable controls is provided by \cite[Lemma 4.4]{lipster_statistics_rp_1977}.
\end{proof}

\subsubsection{Proof of Theorems \ref{thm:smooth_equiv} and Corollary \ref{cor:prop_of_chaos}} \label{sec:prove_ps}

We start with the proof of Theorem \ref{thm:smooth_equiv}.

\begin{proof}[Proof of Theorem \ref{thm:smooth_equiv}]
First, we establish the equality $V_{\textup{srl}} = V$, where we recall that $V_{\textup{srl}}$ is the infimum of $\cal{J}$ over smooth relaxed control rules and $V$ denotes the optimal cost of the strong mean-field control problem. We achieve this by approximating an arbitrary smooth relaxed control rule $\pr_0$ by a sequence of strong controls $(\gamma^n)_n$. Note that under the probability distribution $\pr_{\ast}$ on $\Omega_{\ast}$ associated to $\pr_0$, the random measure $\Gamma^{\ast}$ is a smooth relaxed control as defined in the paragraph above Proposition \ref{prop:existence_uniqueness_stability} and $X^{\ast}$ solves the McKean--Vlasov SDE \eqref{eq:relaxed_weak_mfl_2}. Now, Corollary \ref{cor:approx_contr} reformulates the conclusions of Lemma 4.3 in \cite{mkv_control_limit_2020} in the context of our setup. Thus, we can execute the subsequent programme in \cite{mkv_control_limit_2020}, Lemma 4.4 and Proposition 4.5, to find a sequence $(\gamma^n)_n$ of strong controls, i.e.\@ $\bb{F}^{\xi, W, W^0}$-progressively measurable $G$-valued processes, on the probability space $(\Omega, \F, \pr)$ with the following property: if we set $\mu^n = \L(X^n, W, \Gamma^n, I^n \vert W^0)$ and $\nu^n = (\ev[e^{-\Lambda^n_t} \delta_{X^n_t} \vert W^0])_{t \in [0, T]}$, where $\Gamma^n$ is the relaxed control associated to $\gamma^n$, $X^n$ is the unique strong solution to McKean--Vlasov SDE \eqref{eq:relaxed_weak_mfl_2} with $\bb{G} = \bb{F}^{W^0}$, and $I^n = \bf{1}_{\theta > \Lambda^n_t}$ for a standard exponential random variable $\theta$ independent of the remaining random variables, and further let $\pr^n_0$ denote the law of $(\mu^n, W^0)$, then $\pr^n_0 \Rightarrow \pr_0$ and $((X^{\ast}, \nu^{\ast})^{\#}\pr_{\ast}^n)_n$ is uniformly square-integrable, where $\pr_{\ast}^n$ is the law on $\P(\Omega_{\ast})$ associated with $\pr_0^n$. Thus, by Proposition \ref{prop:conv_of_cost}, we get that $\lim_{n \to \infty} \cal{J}(\pr^n_0) = \cal{J}(\pr_0)$. In particular, since $X^n$ solves McKean--Vlasov SDE \eqref{eq:mfl}, we get that $J(\gamma^n) \geq V$. Now, we specifically choose $\pr_0$ such that $\cal{J}(\pr_0) \leq V_0 + \epsilon$ for a fixed $\epsilon > 0$. This implies
\begin{equation*}
    V \leq \lim_{n \to \infty} J(\gamma^n) = \lim_{n \to \infty} \cal{J}(\pr^n_0) = \cal{J}(\pr_0) \leq V_{\textup{srl}} + \epsilon.
\end{equation*}
Letting $\epsilon \to 0$ gives $V \leq V_{\textup{srl}}$. Since every $\bb{F}^{\xi, W, W^0}$-progressively measurable $G$-valued process $\gamma$ induces a smooth control rule, we also have the reverse inequality $V_{\textup{srl}} \leq V$, whence $V_{\textup{srl}} = V$.

Next, we show that $\limsup_{N \to \infty} V^N \leq V$. Let us fix an arbitrary $\epsilon > 0$ and choose a strong control $\gamma$ with $J(\gamma) \leq V + \epsilon$. Since $\gamma$ is $\bb{F}^{\xi, W, W^0}$-progressively measurable, we can find a measurable function $g \define [0, T] \times \R \times C([0, T]) \times C([0, T]) \to \R$ such that $\gamma_t = g(t, \xi, W_{\cdot \land t}, W^0_{\cdot \land t})$ for $\textup{Leb} \otimes \pr$-almost every $(t, \omega) \in [0, T] \times \Omega$. Now, we set $\gamma^{N, i}_t = g(t, \xi_i, W^i_{\cdot \land t}, W^0_{\cdot \land t})$, so that the controls $\gamma^{N, 1}$,~\ldots, $\gamma^{N, N}$ are admissible for the particle system. Moreover, it is not difficult to see that the sequence $(\mu^N, W^0)_N$ induced by these controls converges weakly to the smooth relaxed control rule $\pr_0$ induced by the unique strong solution $X$ of the McKean--Vlasov SDE \eqref{eq:mfl} with control $\gamma$ provided by Proposition \ref{prop:existence_uniqueness_stability}. Hence, applying Proposition \ref{prop:conv_of_cost} once again shows that
\begin{align*}
    \limsup_{N \to \infty} V^N &\leq \lim_{N \to \infty} J^N(\gamma^{N, 1}, \dots, \gamma^{N, N}) = \lim_{N \to \infty} \cal{J}(\L(\mu^N, W^0)) \\
    &= \cal{J}(\pr_0) = J(\gamma) \leq V + \epsilon.
\end{align*}
Since $\epsilon > 0$ was arbitrary, we get $\limsup_{N \to \infty} V^N \leq V = V_{\textup{srl}}$ as desired.
\end{proof}

We can immediately move on to the proof of Corollary \ref{cor:prop_of_chaos}. Recall that Corollary \ref{cor:prop_of_chaos} assumes that $b$ is of the form $b(t, x, v, g) = b_0(t, x, v)$, so there is no control present.

\begin{proof}[Proof of Corollary \ref{cor:prop_of_chaos}]
Since $b$ does not depend on the control, the particle system is exchangeable. By Theorem \ref{thm:convergence_ps} the sequence $(\mu^N, W^0)$ subsequentially converges to a probability measure $\pr_0$ on $\Omega_0$ such that $(X^{\ast}, \nu^{\ast})$ solves the McKean--Vlasov SDE \eqref{eq:relaxed_weak_mfl} under $\pr_{\ast}$, where $\pr_{\ast}$ is the probability measure on $\Omega_{\ast}$ associated to $\pr_0$. Using the exchangeability of the particle system together with Lemma \ref{lem:weak_conv_particle} and Proposition \ref{prop:id_of_nu} gives that
\begin{equation*}
    \L(X^{N, 1}, \nu^N) = \frac{1}{N}\sum_{i = 1}^N \L(X^{N, i}, \nu^N) \Rightarrow \L_{\ast}(X^{\ast}, \nu^{\ast})
\end{equation*}
on $D[0, T] \times D_{\bf{M}^2}[0, T]$ along the subsequence from above. By Proposition \ref{prop:existence_uniqueness_stability}, the process $(X^{\ast}, \nu^{\ast})$ is the unique strong solution to the McKean--Vlasov SDE \eqref{eq:mfl} on the probability space $(\Omega_{\ast}, \F^{\ast}, \pr_{\ast})$ with idiosyncratic noise $W^{\ast}$ and common noise $B^{\ast}$. Hence, by the Yamada--Watanabe theorem $\L_{\ast}(X^{\ast}, \nu^{\ast})$ is equal to the law of the unique strong solution  $(X, \nu)$ to McKean--Vlasov SDE \eqref{eq:mfl} (on the setup $(\Omega, \F, \pr)$ with noises $W$ and $W^0$). Thus, $\L(X^{N, 1}, \nu^N)$ converges subsequentially to $\L(X, \nu)$ on $D[0, T] \times D_{\bf{M}^2}[0, T]$. However, by the weak uniqueness of $(X, \nu)$, the weak limit of $\L(X^{N, 1}, \nu^N)$ along any other subsequence must also coincide with $\L(X, \nu)$, so that the weak convergence actually holds along the entire sequence.
\end{proof}

\subsection{Closed-Loop Controls and Equivalence between all Formulations} \label{eq:proof_all_equiv}

In this subsection, we only give the proof of Theorem \ref{thm:all_equiv}.

\begin{proof}[Proof of Theorem \ref{thm:all_equiv}]
Let $\pr_0 \in \P(\Omega_0)$ be a relaxed control rule. Our first goal is to construct a closed-loop control rule $\pr_0'$ such that $\cal{J}(\pr_0') \leq \cal{J}(\pr_0)$. By \cite[Lemma 3.2]{lacker_mfg_controlled_mgale_2015}, we can find an $\bb{F}^{\ast}$-predictable $\P(G)$-valued process$(\Gamma_t)_{t \in [0, T]}$ such that $\d \Gamma(t, g) = \d \Gamma_t(g) \, \d t$. Next, set $\gamma_t = \int_G g \, \d \Gamma$ for $t \in [0, T]$, so that $\int_0^t b(s, X_s, \nu_s, g) \, \d \Gamma(s, g) = \int_0^t b(s, X_s, \nu_s, \gamma_s) \, \d s$ because $g \mapsto b(t, x, v, g)$ is affine. Then, we define measurable functions $r$, $g_{\ast} \define [0, T] \times \R \times \M^1_{\leq 1}(\R) \to \R$ by $r(t, x, v) = -\log\ev_{\ast}[I_t \vert X_t = x,\, \nu_t = v]$ and
\begin{equation*}
    g_{\ast}(t, x, v) = \ev_{\ast}\bigl[I_t e^{r(t, X_t, \nu_t)} \gamma_t \bigr\vert X_t = x,\, \nu_t = v\bigr],
\end{equation*}
so that $\langle \nu_t, \varphi\rangle = \ev[e^{-r(t, X_t, \nu_t)} \varphi(X_t) \vert \nu_t]$ a.s.\@ for any $\varphi \define \R \to \R$ measurable and bounded. Next, fix $\varphi \in C^2_c(\R)$ and apply It\^o's formula to $I_t \varphi(X_t)$, whereby
\begin{align*}
    \d (I_t \varphi(X_t)) &= I_t\L\varphi(t, X_t, \nu_t, \gamma_t) \, \d t + I_t \partial_x \varphi(X_t) \Bigl(\sigma(t, X_t, \nu_t) \, \d W_t + \sigma_0(t, X_t, \nu_t) \, \d B_t\Bigr) \\
    &\ \ \ + \varphi(X_t) \bigl(\d I_t + I_t \lambda(t, X_t, \nu_t) \, \d t\bigr).
\end{align*}
Here the differential operator $\L$ acts on $\varphi \in C^2_c(\R)$ by
\begin{equation} \label{eq:diff_op_sfpe}
    \L\varphi(t, x, v, g) = -\lambda(t, x, v) \varphi(x) + b_{\alpha}(t, x, v, g) \partial_x \varphi(x) + a(t, x, v) \partial^2_x\varphi(x)
\end{equation}
for $(t, x, v, g) \in [0, T] \times \R \times \cal{M}^1_{\leq 1}(\R) \times G$ with
\begin{equation*}
    b_{\alpha}(t, x, v, g) = b(t, x, v, g) - \alpha(t, x, v) \langle v, \lambda(t, \cdot, v)\rangle
\end{equation*}
and $a(t, x, v) = \frac{1}{2}(\sigma^2(t, x, v) + \sigma_0^2(t, x, v))$. We take conditional expectation with respect to $\F^0_T$ on both sides to obtain
\begin{equation*}
    \d \langle \nu_t, \varphi\rangle = \ev\bigl[\L\varphi(t, X_t, \nu_t, \gamma_t) \big\vert \F^0_T\bigr] \, \d t + \bigl\langle \nu_t, \sigma_0(t, \cdot, \nu_t) \partial_x \varphi\bigr\rangle \, \d B_t,
\end{equation*}
where we use the stochastic Fubini theorem (see e.g.\@ \cite[Lemma A.5]{hammersley_weak_ex_uni_2021} or \cite[Lemma B.1]{lacker_mimicking_2020}) to exchange the conditional expectation with respect to $\F^0_T$ and the stochastic integration against $B$. The stochastic integral with respect to $W$ vanishes since $W$ is independent of $\F^0_T$. Lastly, the integral with respect to $I_t - \int_0^t I_s \lambda(s, X_s, \nu_s) \, \d s$ vanishes because $\pr_0$ is a relaxed control, so satisfies Property \ref{it:martingale} of Definition \ref{def:rel_cntrl}.

Now, we may proceed as in the proof of Corollary 1.6 in \cite{lacker_mimicking_2020}, to show that the law $P_t = \L_{\ast}(\nu_t)$ satisfies the Fokker--Planck equation
\begin{align*}
    &\int_{\M^1_{\leq 1}(\R)} \Phi(\langle v, \varphi\rangle) \, \d (P_t - P_0)(v) \\
    &= \sum_{i = 1}^k \int_0^t \biggl(\int_{\M^1_{\leq 1}(\R)}  \partial_{x_i} \Phi(\langle v, \varphi\rangle) \bigl\langle v, \L\varphi_i\bigl(s, \cdot, v, g_{\ast}(s, \cdot, v)\bigr)\bigr\rangle \, \d P_s(v)\biggr) \, \d s \\
    &\ \ \ + \frac{1}{2}\sum_{i, j = 1}^k \int_0^t \biggl(\int_{\M^1_{\leq 1}(\R)}  \partial^2_{x_i x_j} \Phi(\langle v, \varphi\rangle) \langle v, \sigma_0(t, \cdot, v)\partial_x \varphi_i\rangle \langle v, \sigma_0(t, \cdot, v)\partial_x \varphi_j\rangle \, \d P_s(v)\biggr) \, \d s
\end{align*}
for $\varphi = (\varphi_1, \dots, \varphi_k) \in C^2_c(\R; \R^k)$ and $\Phi \in C^2_c(\R^k)$ on $\M^1_{\leq 1}(\R)$. Here we used that
\begin{align*}
    \ev_{\ast}\bigl[I_t b(t, X_t, \nu_t, \gamma_t) \big\vert X_t,\, \nu_t \bigr] &= e^{-r(t, X_t, \nu_t)} \ev_{\ast}\bigl[I_t e^{r(t, X_t, \nu_t)} b(t, X_t, \nu_t, \gamma_t) \big\vert X_t,\, \nu_t \bigr] \\
    &= e^{-r(t, X_t, \nu_t)} b\bigl(t, X_t, \nu_t, \ev_{\ast}\bigl[I_te^{r(t, X_t, \nu_t)}\gamma_t \big\vert X_t, \nu_t\bigr]\bigr) \\
    &= e^{-r(t, X_t, \nu_t)} b\bigl(t, X_t, \nu_t, g_{\ast}(t, X_t, \nu_t)\bigr),
\end{align*}
which holds because because $g \mapsto b(t, x, v, g)$ is affine and $\ev_{\ast}[I_t e^{r(t, X_t, \nu_t)} \vert X_t,\, \nu_t] = 1$, and in turn implies
\begin{align*}
    \ev_{\ast}\bigl[\partial_{x_i}\Phi&(\langle \nu_t, \varphi\rangle) I_t b(t, X_t, \nu_t, \gamma_t) \partial_x \varphi_i(X_t) \bigr] \\
    &= \ev_{\ast}\Bigl[\partial_{x_i}\Phi(\langle \nu_t, \varphi\rangle) \ev_{\ast}\bigl[I_t b(t, X_t, \nu_t, \gamma_t)\big\vert X_t,\, \nu_t\bigr] \partial_x \varphi_i(X_t)\Bigr] \\
    &= \ev_{\ast}\Bigl[\partial_{x_i}\Phi(\langle \nu_t, \varphi\rangle) e^{-r(t, X_t, \nu_t)} b\bigl(t, X_t, \nu_t, g_{\ast}(t, X_t, \nu_t)\bigr) \partial_x \varphi_i(X_t)\Bigr] \\
    &= \ev_{\ast}\Bigl[\partial_{x_i}\Phi(\langle \nu_t, \varphi\rangle) \bigl\langle \nu_t, b\bigl(t, \cdot, \nu_t, g_{\ast}(t, \cdot, \nu_t)\bigr) \partial_x\varphi_i\bigr\rangle\Bigr].
\end{align*}
Hence, we can apply the superposition principle for stochastic Fokker--Planck equations from \cite[Theorem 1.5]{lacker_mimicking_2020}, which can easily be extended to equations with zeroth-order terms in the generator $\L$, to show that there exists a filtered probability space $(\tilde{\Omega}_0, \tilde{\F}^0, \tilde{\bb{F}}^0, \tilde{\pr}_0)$ with $\tilde{\F}^0_T$ countably generated which carries an $\tilde{\bb{F}}^0$-Brownian motion $\tilde{B}$ and a continuous $\tilde{\bb{F}}^0$-adapted $\M_{\leq 1}(\R)$-valued process $\tilde{\nu}$ such that $\L^{\tilde{\pr}_0}(\tilde{\nu}_t) = P_t$ and
\begin{equation*}
    \d \langle \tilde{\nu}_t, \varphi\rangle = \bigl\langle \tilde{\nu}_t, \L\varphi\bigl(t, \cdot, \tilde{\nu}_t, g_{\ast}(t, \cdot, \tilde{\nu}_t)\bigr)\bigr\rangle \, \d t + \bigl\langle \tilde{\nu}_t, \sigma_0(t, \cdot, \tilde{\nu}_t) \partial_x \varphi\bigr\rangle \, \d \tilde{B}_t
\end{equation*}
for all $\varphi \in C^2_c(\R)$ and $t \in [0, T]$. 

Our goal is to construct a closed-loop control rule $\tilde{\pr}_0$ whose feedback function is given by $g_{\ast}$ and such that the law of $\nu_t$ under $\pr_{\ast}$ coincides with $P_t$. We can extend $(\tilde{\Omega}_0, \tilde{\F}^0, \tilde{\bb{F}}^0, \tilde{\pr}_0)$ to a probability space $(\tilde{\Omega}, \tilde{\F}, \tilde{\bb{F}}, \tilde{\pr})$ supporting an $\tilde{\bb{F}}$-Brownian motion $\tilde{W}$ and a $\tilde{\F}_0$-measurable random variable $\tilde{\xi}$ with distribution $\nu_0$, such that $(\tilde{\xi}, \tilde{W})$ is independent of $\tilde{\F}^0_T$ and $\tilde{B}$ is an $\tilde{\bb{F}}$-Brownian motion. Now, we solve the SDE with random coefficients
\begin{equation} \label{eq:sde_random_mckean}
    \d \tilde{X}_t = b_{\alpha}\bigl(t, \tilde{X}_t, \tilde{\nu}_t, g_{\ast}(t, \tilde{X}_t, \tilde{\nu}_t)\bigr) \, \d t + \sigma(t, \tilde{X}_t, \tilde{\nu}_t) \, \d \tilde{W}_t + \sigma_0(t, \tilde{X}_t, \tilde{\nu}_t) \, \d \tilde{B}_t,
\end{equation}
with initial condition $\tilde{X}_0 = \tilde{\xi}$.
Let us set $\tilde{\Lambda}_t = \int_0^t \lambda(s, \tilde{X}_s, \tilde{\nu}_s) \, \d s$ and then define the $\M^2_{\leq 1}(\R^d)$-valued process $\tilde{\nu}' = (\tilde{\nu}'_t)_{0 \leq t \leq T}$ by $\tilde{\nu}'_t = \tilde{\ev}[e^{-\tilde{\Lambda}_t} \delta_{\tilde{X_t}} \vert \tilde{\F}^0_T]$ for $t \in [0, T]$, where $\tilde{\ev}$ denotes the expectation with respect to $\tilde{\pr}$. Proceeding as at the beginning of the proof, we can show that $\tilde{\nu}'$ solves the linear SPDE
\begin{equation} \label{eq:spde_rd_lin}
    \d \langle \tilde{\nu}'_t, \varphi\rangle = \bigl\langle \tilde{\nu}'_t, \L\varphi\bigl(t, \cdot, \tilde{\nu}_t, g_{\ast}(t, \cdot, \tilde{\nu}_t)\bigr)\bigr\rangle \, \d t + \bigl\langle \tilde{\nu}'_t, \sigma_0(t, \cdot, \tilde{\nu}_t)\partial_x \varphi\bigr\rangle \, \d \tilde{B}_t
\end{equation}
for $\varphi \in C_c^2(\R^d)$, with initial condition $\tilde{\nu}'_0 = \nu_0$. Since Assumptions \ref{ass:red_form_cont_prob} and \ref{ass:aff_conv} are in force, this SPDE has a unique solution by Theorem \ref{thm:uniqueness_spde}. But $\tilde{\nu}$ satisfies SPDE \eqref{eq:spde_rd_lin} as well, so we conclude that $\tilde{\nu}'_t = \tilde{\nu}_t$, meaning that $(\tilde{X}, \tilde{\nu})$ solves \eqref{eq:sde_random_mckean} when viewed as a McKean--Vlasov SDE.

Let us now define $\pr_0'$ to be the law of $\bigl(\L^{\tilde{\pr}}(\tilde{X}, \tilde{W}, \tilde{\Gamma}, \tilde{I} \vert \tilde{\F}^0_T), \tilde{B}\bigr)$ under $\tilde{\pr}$, where $\d\tilde{\Gamma}(t, g) = \d \delta_{g_{\ast}(t, \tilde{X}_t, \tilde{\nu}_t)} \, \d t$ and $\tilde{I} = \bf{1}_{\tilde{\theta} > \tilde{\Lambda}_t}$ for a standard exponential random variable $\tilde{\theta}$ on $(\tilde{\Omega}, \tilde{\F}, \tilde{\pr})$, which is independent of the remaining random variables. It follows that $\pr_0'$ is a smooth relaxed control rule in the sense of Definition \ref{def:rel_cntrl} and by construction it is closed-loop. Moreover, by convexity of $f$ in the control argument (see Assumption \ref{ass:aff_conv} \ref{it:convex}), it holds that
\begin{align*}
    \ev_{\ast}&\biggl[\int_{[0, T] \times G} I_t f(t, X_t, \nu_t, g) \, \d \Gamma(t, g) + \psi(\nu_T)\biggr] \\
    &= \ev_{\ast}\biggl[\int_0^T \biggl(e^{-r(t, X_t, \nu_t)} \ev_{\ast}\biggl[\int_G I_t e^{r(t, X_t, \nu_t)} f(t, X_t, \nu_t, g)  \d \Gamma_t(g)\bigg\vert X_t, \, \nu_t\biggr]\biggr) \, \d t + \psi(\nu_T)\biggr] \\
    &\geq \ev_{\ast}\biggl[\int_0^T e^{-r(t, X_t, \nu_t)} f\biggl(t, X_t, \nu_t, \ev_{\ast}\biggl[I_t e^{r(t, X_t, \nu_t)} \int_G g \, \d \Gamma_t(g) \bigg\vert X_t, \, \nu_t\biggr]\biggr) \, \d t + \psi(\nu_T)\biggr] \\
    &= \ev_{\ast}\biggl[\int_0^T e^{-r(t, X_t, \nu_t)} f\bigl(t, X_t, \nu_t, g_{\ast}(t, X_t, \nu_t)\bigr) \, \d t + \psi(\nu_T)\biggr] \\
    &= \ev_{\ast}\biggl[\int_0^T \bigl\langle \nu_t, f\bigl(t, \cdot, \nu_t, g_{\ast}(t, \cdot, \nu_t)\bigr)\bigr\rangle \, \d t + \psi(\nu_T)\biggr] 
\end{align*}
where we applied Jensen's inequality using that
\begin{equation*}
    \F_{\ast} \otimes \cal{B}(G) \ni A \mapsto \ev_{\ast}\biggl[\int_G I_t e^{r(t, X_t, \nu_t)} \bf{1}_A(\cdot, g) \, \d \Gamma_t(g) \bigg\vert X_t, \nu_t\biggr]
\end{equation*}
is a (random) probability measure. But this implies
\begin{align} \label{eq:rel_to_cl}
    \cal{J}(\pr_0) = \ev_{\ast}&\biggl[\int_{[0, T] \times G} I_t f(t, X_t, \nu_t, g) \, \d \Gamma(t, g) + \psi(\nu_T)\biggr] \notag \\
    &\geq \ev_{\ast}\biggl[\int_0^T \bigl\langle \nu_t, f\bigl(t, \cdot, \nu_t, g_{\ast}(t, \cdot, \nu_t)\bigr)\bigr\rangle \, \d t + \psi(\nu_T)\biggr] \notag \\
    &= \tilde{\ev}\biggl[\int_0^T \bigl\langle \tilde{\nu}_t, f\bigl(t, \cdot, \tilde{\nu}_t, g_{\ast}(t, \cdot, \tilde{\nu}_t)\bigr)\bigr\rangle \, \d t + \psi(\tilde{\nu}_T)\biggr] \notag \\
    &= \ev_{\ast}'\biggl[\int_{[0, T] \times G} I_t f(t, X_t, \nu_t, g) \, \d \Gamma(t, g) + \psi(\nu_T)\biggr] \notag \\
    &= \cal{J}(\pr_0'),
\end{align}
where $\ev_{\ast}'$ denotes the expectation with respect to the probability measure on $\Omega_{\ast}$ associated to $\pr_0'$. Hence, $\pr_0'$ is the desired closed-loop control rule. 

Next, minimising the left-hand side of \eqref{eq:rel_to_cl} implies that $V_{\textup{rl}} \geq V_{\textup{cl}} \geq V_{\textup{srl}}$, which, together with the trivial inequality $V_{\textup{rl}} \leq V_{\textup{srl}}$, yields $V_{\textup{rl}} = V_{\textup{cl}} = V_{\textup{srl}}$. Combining this with Theorems \ref{thm:convergence_ps} and \ref{thm:smooth_equiv} furthermore implies that
\begin{equation*}
    \liminf_{N \to \infty} V^N \geq V_{\textup{rl}} = V_{\textup{cl}} = V_{\textup{srl}} = V \geq \limsup_{N \to \infty} V^N.
\end{equation*}
Consequently, all expressions in the above are in fact equal. This concludes the proof.
\end{proof}

\section{Convergence to the Singular Limit} \label{sec:sin_limit}

In this section, we analyse the convergence of the regularised system as we let the regularisation vanish. This will culminate in the proof of Theorem \ref{thm:convergence_struc} as well as Propositions \ref{prop:conv_struc_simple_coeff} and \ref{eq:minimal_convergence}.

\subsection{Subsequential Convergence of the Regularised Sequence} \label{sec:reg_to_sing}

Here as well as in the subsequent three subsections we adopt the convention from Subsection \ref{sec:limit_system} and drop the asterisk in the superscript of variables defined on $\Omega_{\ast}$. Note that this means that we will use the same notation for the idiosyncratic noises on $(\Omega, \F, \pr)$ and $(\Omega_{\ast}, \F_{\ast}, \pr_{\ast})$, but in each case it will be clear from context which random variable we are referring to.

We fix a sequence $(\gamma^n)_n$ of admissible strong controls, that is $\bb{F}^{\xi, W, W^0}$-progressively measurable $G$-valued processes $\gamma^n$. Then we let $X^n$ be the solution to McKean--Vlasov SDE \eqref{eq:mfl_sequence} with input $\gamma^n$, set $\nu^n = \ev[e^{-\Lambda^n_t} \delta_{X^n_t}\vert W^0]$, and denote the relaxed control associated to $\gamma^n$ by $\Gamma^n$. The objective of this section is to establish the subsequential weak convergence of the family $(\mu^n, W^0)_n$, with $\mu^n = \L(X^n, W, \Gamma^n \vert W^0)$, on $\Omega_0$ to an admissible relaxed control rule (for the singular model), cf.\@ Definition \ref{def:rel_cntr_struc}. First, we note that similar to Lemma \ref{lem:l_2_bound_ps} we can show that the family $((\lvert X^n \rvert^{\ast}_T)^2)_n$ is uniformly integrable. This allows us to deduce the following tightness result.

\begin{proposition} \label{prop:sequence_tightness}
Let Assumption \ref{ass:struc_limit} be satisfied. Then the sequence $(\mu^n, W^0)_n$ is tight on $\Omega_0$.
\end{proposition}

\begin{proof}
Clearly $(W^0)_n$ is tight, so we only need to worry about tightness of $(\mu^n)_n$. By \cite[Proposition 2.2 (ii)]{sznitman_prop_chaos_1991} this follows from tightness of $(X^n, W, \Gamma^n)_n$, where $X^n$ and the corresponding subprobability distribution $\nu^n$ are defined above the statement of the proposition. We establish the tightness of each component separately. Tightness of $(\Gamma^n)_n$ follows from compactness of $G$, while $(W)_n$ is trivially tight. Thus, it remains to prove tightness of $(X^n)_n$. We write $X^n = \xi + \beta^n + Z^n - F^n_t$ with
\begin{equation*}
    \beta^n_t = \int_0^t b(s, X^n_s, \nu^n_s, \gamma^n_s) \, \d s, \qquad Z^n_t = \int_0^t \sigma(s, X^n_s, \nu^n_s) \, \d W_s + \int_0^t \sigma_0(s, X^n_s, \nu^n_s) \, \d W^0_s,
\end{equation*}
and $F^n_t = \int_0^t \alpha(s) \, \d L^n_s$ for $t \in [0, T + 1]$, where we recall that we extended the coefficients $b$, $\sigma$, $\sigma_0$, and $\alpha$ to $[0, T + 1]$ in a suitable way, see the paragraph above Definition \ref{def:rel_cntr_struc}. Further setting $\beta^n_t = Z^n_t = F^n_t = 0$ for $t \in [-1, 0)$ turns $\beta^n$, $Z^n$, and $F^n$ into random variables with values in $C([-1, T + 1])$, $C([-1, T + 1])$, and $D[-1, T + 1]$, respectively. We will prove that as such the sequences $(\beta^n)_n$, $(Z^n)_n$, $(F^n)_n$ are tight, which immediately implies tightness of $(X^n)_n$ on $D[-1, T + 1]$. We begin by analysing the process $\beta^n$. For any $s$, $t \in [0, T]$ it holds that
\begin{equation*}
    \lvert \beta^n_t - \beta^n_s \rvert^2 = \biggl\lvert \int_s^t b^n_u \, \d u \biggr\rvert^2 \leq \lvert t - s \rvert \int_s^t \lvert b^n_u \rvert^2 \, \d u \leq \lvert t - s \rvert \int_0^T \lvert b^n_u \rvert^2 \, \d u,
\end{equation*}
where $b^n_u = b(u, X^n_u, \nu^n_u, \gamma^n_u)$. As we mentioned at the beginning of the section, the family $((\lvert X^n \rvert^{\ast}_T)^2)_n$ is uniformly integrable, which implies that the quantity $\ev \int_0^T \lvert b^n_u \rvert^2 \, \d u$ is bounded uniformly in $n \geq 1$. Thus, for any $\epsilon > 0$, we can find $R > 0$ large enough, such that $\pr\bigl(\int_0^T \lvert b^n_u \rvert^2 \, \d u \leq R^2\bigr) \geq 1 - \epsilon$. Now we define $\cal{K}_{\epsilon} \subset C([-1, T + 1])$ to be the set of functions $f \define [-1, T + 1] \to \R$, such that $f(t) = 0$ for $t \in [-1, 0]$ and
\begin{equation*}
    \lvert f(t) - f(s) \rvert \leq R \lvert t - s \rvert^{\frac{1}{2}}
\end{equation*}
for all $s$, $t \in [0, T + 1]$. Clearly, the set $\cal{K}_{\epsilon}$ is a compact subset of $C([-1, T + 1])$, and
\begin{equation*}
    \pr(\beta^n \in \cal{K}_{\epsilon}) \geq \pr\biggl(\int_0^T \lvert b^n_u \rvert^2 \, \d u \leq R^2\biggr) \geq 1 - \epsilon.
\end{equation*}
Consequently, the laws $\L^{\pr}(\beta^n)$, $n \geq 1$, form a tight sequence on $C([-1, T + 1])$ and any subsequential limit is concentrated on the space of functions with $\frac{1}{2}$-H\"older continuous trajectories. 

Next, the bound $\sigma^2(t, x, \nu) + \sigma_0^2(t, x, \nu) \leq 2C_{\sigma}^2$ implied by Assumption \ref{ass:red_form_cont_prob} \ref{it:growth} together with Kolmogorov's tightness criterion imply that the sequence $(\L^{\pr}(Z^n))_n$ is tight on $C([-1, T + 1])$.

Finally, we show that $(\L(F^n))_n$ form a tight family on $D[-1, T + 1]$. It follows from Theorem 12.12.2 of \cite{whitt_stoch_limits_2002} that the set of nondecreasing c\`adl\`ag functions on $[-1, T + 1]$, which are started from zero, constant on $[-1, 0)$ and $[T + 1/2, T]$ (recall that $\alpha$ vanishes on $[T + 1/2, T]$), and bounded by $\lVert \alpha \rVert_{\infty}$, form a compact subset of $D[-1, T + 1]$. Every member of the sequence $(F^n)_n$ is a.s.\@ a member of this set because $\alpha$ is nonnegative by Assumption \ref{ass:struc_limit} \ref{it:alpha}, so $(\L(F^n))_n$ is tight. 
\end{proof}

Recall that $\bf{M}^2 = \cal{M}^2_{\leq 1}(\R)$ and denote by $L^2([-1, T + 1]; \bf{M}^2)$ the space of $\bf{M}^2$-valued measurable functions $v$ on $[-1, T + 1]$ with $\int_{-1}^{T + 1} M_1^2(v_t) \, \d t < \infty$. We endow $L^2([-1, T + 1]; \bf{M}^2)$ with the metric
\begin{equation*}
    (v^1, v^2) \mapsto \biggl(\int_{-1}^{T + 1} d_2^2(v^1_t, v^2_t) \, \d t\biggr)^{1/2},
\end{equation*}
which turns it into a separable complete metric space. Moreover, the inclusion $D_{\bf{M}^2}[-1, T + 1] \to L^2([-1, T + 1]; \bf{M}^2)$ is continuous, so that we can view $\nu$, defined in Equation \eqref{eq:subprobability_sin}, as a random variable with values in $L^2([-1, T + 1]; \bf{M}^2)$.

The reason for introducing the space $L^2([-1, T + 1]; \bf{M}^2)$ is that it is not clear whether the sequence $(\nu^n)_n$ is tight on $D_{\bf{M}^2}[-1, T + 1]$ nor whether it converges subsequentially to $\nu$ in the space $D_{\bf{M}^2}[-1, T + 1]$, cf.\@ Remark \ref{rem:conv_of_nu}. We can, however, establish both statements for the space $L^2([-1, T + 1]; \bf{M}^2)$.

\begin{proposition} \label{prop:subprobabilities_converge}
Let Assumption \ref{ass:struc_limit} be satisfied. Then for any subsequential limit $\pr_0$ of $(\mu^n, W^0)_n$, the map $\nu \define \Omega_{\ast} \to L^2([-1, T + 1]; \bf{M}^2)$ is $\pr_{\ast}$-a.s.\@ continuous. Moreover, $\nu$ is the subsequential weak limit of $(\nu^n)_n$ on $L^2([-1, T + 1]; \bf{M}^2)$.
\end{proposition}

\begin{proof}
We proceed in three steps.

\textit{Step 1}: We begin by showing that the random time $\tau = \inf\{0 < t \leq T + 1 \define X_t \leq 0\}$ is $\pr_{\ast}$-a.s.\@ continuous. Recall the sequence of processes $(\beta^n, Z^n, F^n)_n$ introduced in Proposition \ref{prop:sequence_tightness}. Since the law of these processes is tight on $C([-1, T + 1]) \times C([-1, T + 1]) \times D[-1, T + 1]$, $X^n = \xi + \beta^n + Z^n + F^n$, and $X^n$ converges subsequentially to the law of $X$ under $\pr_{\ast}$, by enlarging the probability space $(\Omega_{\ast}, \F_{\ast}, \pr_{\ast})$ if necessary, we can find processes $\beta$, $Z$, and $F$, such that $(\beta, Z, F)$ is the subsequential weak limit of $(\beta^n, Z^n, F^n)_n$ and $X = X_0 + \beta + Z + F$. We let $\bb{G} = (\cal{G}_t)_{0 \leq t \leq T + 1}$ be the right-continuous extension of the filtration generated by $X$, $\beta$, $Z$, and $F$. 

The proof of Proposition \ref{prop:sequence_tightness} shows that $\beta$ has $\pr_{\ast}$-a.s.\@ $\frac{1}{2}$-H\"older continuous trajectories. The process $Z$ is a continuous $\bb{G}$-martingale since the local martingale property is conserved by weak convergence and $Z$ is square-integrable. Moreover, by Theorem 6.26 of \cite{jacod_limit_theorems_2003} it holds that $\langle Z^n\rangle$ converges weakly on $C([0, T])$ to $\langle Z\rangle$ along a subsequence, which in view of the nondegeneracy condition, Assumption \ref{ass:struc_limit} \ref{it:ellipticity}, allows us to deduce that $(C_{\sigma}^{-1} \land 1) (t - s) \leq \langle Z \rangle_t - \langle Z \rangle_s$ for $0 \leq s \leq t \leq T + 1$. The minimum with $1$ comes from having $Z$ diffuse according to a Brownian motion between time $T$ and $T + 1$. Lastly, the process $F$ is nondecreasing as the weak limit of nondecreasing functions in $D[-1, T + 1]$. We will use this decomposition of $X$ under $\pr_{\ast}$ to prove that $X$ has the downcrossing property. That is, for $\pr_{\ast}$-a.e.\@ $\omega \in \{\tau \leq T + 1\}$, given any $0 \leq \delta \leq T + 1 - \tau$ there exists $t \in [\tau, \tau + \delta]$ such that $X_t < 0$. 

First, note that $X^n$ diffuses according to a Brownian motion on the interval $[T + 1/2, T + 1]$, so the same holds for $X$ under $\pr_{\ast}$. Consequently, we have $\pr_{\ast}(\tau = T + 1) = 0$ and we may assume that $\tau < T + 1$. Now, let us denote the smallest (random) constant $C > 0$ for which $\lvert \beta_t - \beta_s\rvert \leq C \lvert t - s \rvert^{1/2}$ for all $s$, $t \in [0, T_1]$ by $C_{\beta}$. Here for notational simplicity, we set $T_1 = T + 1$. Then, we define $Z^{(T_1 \land \tau)}_t = Z_{T_1 \land (\tau + t)} - Z_{T_1 \land \tau}$ for $t \in [0, T_1]$, which is a martingale with respect to the filtration $(\cal{G}_{T_1 \land (\tau + t)})_{0 \leq t \leq T_1}$ by the optional stopping theorem. It has quadratic variation $q_t = \langle Z^{(T_1 \land \tau)} \rangle_t = \langle Z \rangle_{T_1 \land (\tau + t)} - \langle Z \rangle_{T_1 \land \tau}$, so by the Dambis-Dubins-Schwarz theorem there is a Brownian motion $M$ with $Z^{(T_1 \land \tau)}_t = M_{q_t}$. Hence, on some event $A_0 \subset \{\tau < T_1\}$ with $\pr_{\ast}(A_0) = \pr_{\ast}(\tau < T_1)$, we have
\begin{equation*}
    \liminf_{t \to 0} \frac{Z^{(T_1 \land \tau)}_t}{\sqrt{t}} = \liminf_{t \to 0} \frac{M_{q_t}}{\sqrt{q_t}} \biggl(\frac{q_t}{t}\biggr)^{1/2} \leq \liminf_{t \to 0} (C_{\sigma}^{-1/2} \land 1)\frac{M_{q_t}}{\sqrt{q_t}} = -\infty,
\end{equation*}
where we used that $\liminf_{t \to 0} \frac{M_t}{\sqrt{t}} = -\infty$ a.s.\@ together with the bound $(C_{\sigma}^{-1} \land 1) t \leq q_t$ for $0 \leq t \leq T_1 - \tau$. Consequently, for any $0 < \delta \leq T_1 - \tau$ we find $t \in [0, \delta]$, such that
\begin{equation*}
    X_{\tau + t} \leq X_{\tau + t} - X_{\tau} = \beta_{\tau + t} - \beta_{\tau} + Z^{(T_1 \land \tau)}_t + F_{\tau + t} - F_{\tau} \leq \sqrt{t} \biggl(C_{\beta} + \frac{Z^{(T_1 \land \tau)}_t}{\sqrt{t}}\biggr) < 0,
\end{equation*}
where we used that $X_{\tau} \leq 0$ and that $F$ is nondecreasing. Hence, $X(\omega)$ has the downcrossing property for any $\omega \in A_0$. We shall use this property to prove that $\tau$ is continuous at every $\omega \in A = A_0 \cup \{\tau = \infty\}$. 

First, let us fix $\omega \in A_0$ as well as a sequence $(\omega_n)_n$ in $\Omega_{\ast}$ which converges to $\omega$. On the one hand, employing the downcrossing property, for any $0 < \delta \leq T + 1 - \tau$ there is a $t \in [\tau(\omega), \tau(\omega) + \delta/2]$ with $X_t(\omega) < 0$. Now $X$ has right-continuous trajectories, so we can find a continuity point $s \in [t, \tau(\omega) + \delta]$ for which $X_s(\omega)$ is negative as well. Since $s$ is a continuity point of $X(\omega)$, it holds that $X_s(\omega_n) \to X_s(\omega)$ as $n \to \infty$. In particular, we have that $X_s(\omega_n) < 0$ for all sufficiently large $n$. Consequently, $\tau(\omega_n) \leq \tau(\omega) + \delta$. Then we let $\delta$ go to zero to see $\limsup_{n \to \infty} \tau(\omega_n) \leq \tau(\omega)$. Next we prove the $\liminf_{n \to \infty} \tau(\omega_n) \geq \tau(\omega)$. First note that
\begin{equation*}
    t_0 = \liminf_{n \to \infty} \tau(\omega_n) \leq \limsup_{n \to \infty} \tau(\omega_n) \leq \tau(\omega) < T + 1.
\end{equation*}
Then we select a subsequence $(\tau(\omega_{n_k}))_k$ with $\lim_{k \to \infty} \tau(\omega_{n_k}) = t_0$. By the definition of $M1$-convergence, choosing a further subsequence if necessary, it holds that
\begin{equation*}
    0 \geq \lim_{k \to \infty} X_{\tau(\omega_{n_k})}(\omega_{n_k}) \in [X_{t_0}(\omega), X_{t_0-}(\omega)],
\end{equation*}
where we implicitly use that $X_{t_0}(\omega) \leq X_{t_0-}(\omega)$ since $X(\omega)$ can only have negative jumps. In particular, we obtain $X_{t_0}(\omega) \leq 0$, whence $\liminf_{n \to \infty} \tau(\omega_n) = t_0 \geq \tau(\omega)$ as required. 

To prove continuity of $\tau$ on $A$, it remains to show that $\tau(\omega_n) \to \infty$ whenever $\tau(\omega) = \infty$. But the latter means that $X(\omega)_t > 0$ for all $t \in [0, T + 1]$ and implies that $X(\omega_n) > 0$ on $[0, T + 1]$ if $n$ is large enough, so we get $\tau(\omega_n) = \infty$.

\textit{Step 2}: We proceed by proving that $\nu \define \Omega_{\ast} \to L^2([-1, T + 1]; \bf{M}^2)$ is $\pr_{\ast}$-a.s.\@ continuous. Let us define the inclusion map $\iota \define \cal{S} \to \Omega_{\ast}$, $(x, w, \gamma) \mapsto (x, w, \gamma, \bar{\omega}_0)$ for some arbitrary element $\bar{\omega}_0 \in \Omega_0$. Then we introduce the indicator process $I$ defined by $I_t = \bf{1}_{\tau > t}$, and set $I^0 = I \circ \iota$ as well as $X^0 = X \circ \iota$. Note that $I^0$ and $X^0$ are independent of the choice of $\bar{\omega}_0 \in \Omega_0$, since $I$ and, trivially, $X$ are completely determined by the trajectory of $X$. We saw above that $\pr_{\ast}(\tau = T + 1) = 0$ and by definition $\tau \geq 0$. From this and the continuity of $\tau$ on $A$ it is easy to deduce that $I$ is continuous as a function $\Omega_{\ast} \to D_{[0, 1]}[-1, T + 1]$ on the set $A$. Now for any element $\omega_0 \in \Omega_0$, we define the measurable set $A_{\omega_0} = \{s \in \cal{S} \define (s, \omega_0) \in A\}$. Then,
\begin{equation} \label{eq:tonelli_argument}
    1 = \pr_{\ast}(A) = \ev_{\ast} \pr_{\ast}(A \vert \F^0_T) = \int_{\Omega_0} \int_{\cal{S}} \bf{1}_A(s, \omega_0) \, \d m(s) \d \pr_0(\omega_0) = \int_{\Omega_0} m(A_{\omega_0}) \, \d \pr_0(\omega_0),
\end{equation}
which means that $m(A_{\omega_0}) = 1$ for $\pr_0$-a.e.\@ $\omega_0 \in \Omega_0$. However, $I^0$ is continuous on $A_{\omega_0}$, which means that $\pr_{\ast}$-a.s.\@ the map $(X^0, I^0)$ is continuous at $\mu$-a.e.\@ element of $\cal{S}$. Thus, the continuous mapping theorem implies that $(X^0, I^0)^{\#}\mu$ is $\pr_{\ast}$-a.s.\@ continuous. The random variable $(X^0, I^0)^{\#}\mu$ takes values in $\P^2(D[-1, T + 1] \times D_{[0, 1]}[-1, T + 1])$, so we may invoke Lemma \ref{lem:space_for_sub} together with the continuous mapping theorem to deduce that $\nu \define \Omega_{\ast} \to L^2([-1, T + 1]; \bf{M}^2)$ is $\pr_{\ast}$-a.s.\@ continuous.

\textit{Step 3}: Next, let us define $\tilde{\nu}^n$ by $\tilde{\nu}^n_t = \pr(X^n_t \in \cdot,\, \inf_{0 \leq s \leq t} X^n_s > 0 \vert W^0)$ for $t \in [0, T + 1]$ and $\tilde{\nu}^n_t = \L(\xi)$ if $t \in [-1, 0)$. Since $\nu \define \Omega_{\ast} \to L^2([-1, T + 1]; \bf{M}^2)$ is $\pr_{\ast}$-a.s.\@ continuous, the continuous mapping theorem implies that along a subsequence the flow of subprobabilities $(\tilde{\nu}^n)_n$ converges weakly to $\nu$ on $L^2([-1, T + 1]; \bf{M}^2)$. Indeed, simply note that $\tilde{\nu}^n = \nu(\Theta^n)$ with $\Theta^n = (X^n, W, \Gamma^n, \mu^n, W^0)$. Thus, if we can show that $\ev \int_0^{T_1} d_2^2(\nu^n_t, \tilde{\nu}^n_t) \, \d t \to 0$ in the limit as $n \to \infty$, the second statement of the lemma follows. Let us set $\tau_n = \inf\{0 < t \leq T + 1 \define \Lambda^n_t \geq \theta\}$ and $\tilde{\tau}_n = \inf\{0 < t \leq T + 1 \define X^n_t \leq 0\}$ as well as $I^n_t = \bf{1}_{\tau_n > t}$ and $\tilde{I}^n_t = \bf{1}_{\tilde{\tau}_n > t}$. To establish the desired convergence, we shall use that $\lim_{n \to \infty} (\tau_n - \tilde{\tau}_n) = 0$ almost surely whenever $\limsup_{n \to \infty} \tilde{\tau}_n < \infty$. We will prove this fact first. 

Since $\lambda^n(t, x, v) = 0$ for $x \geq 0$ by Assumption \ref{ass:struc_limit} \ref{it:intensity_convergence}, it holds that $\Lambda^n_t = 0$ for $t \in [0, \tilde{\tau}_n \land (T + 1)]$, which implies that $\tau_n \geq \tilde{\tau}_n$. Hence, it suffices to prove that $\limsup_{n \to \infty} (\tau_n - \tilde{\tau}_n) \leq 0$. Let us fix a subsequence $(n_k)_k$ for which $(X^{n_k}, \tilde{\tau}_{n_k}) \Rightarrow (X, \tau)$. This is possible owing to the $\pr_{\ast}$-a.s.\@ continuity of $\tau$ established above. Then appealing to Skorokhod's representation theorem (and changing the probability space if necessary), we may assume that there exists a $D[-1, T + 1]$-valued random variable $\tilde{X}$ and a random time $\tilde{\tau}$ with values in $[0, T + 1] \cup \{\infty\}$, such that $(X^{n_k}, \tilde{\tau}_{n_k})$ converges to $(\tilde{X}, \tilde{\tau})$ almost surely. On the event $\{\limsup_{n \to \infty} \tilde{\tau}_n < \infty\}$, we can extract a further (random) subsequence, which for simplicity we denote again by $(n_k)_k$, such that $\lim_{k \to \infty} (\tau_{n_k} - \tilde{\tau}_{n_k}) = \limsup_{n \to \infty} (\tau_n - \tilde{\tau}_n)$. Now by the above, with probability one either $\tilde{\tau} = \infty$ or $\tilde{X}$ has the downcrossing property. If $\tilde{\tau} = \infty$ then $\lim_{k \to \infty} \tilde{\tau}_{n_k} = \tilde{\tau} = \infty$, which contradicts $\limsup_{n \to \infty} \tilde{\tau}_n < \infty$. Hence, for a.e.\@ element of $\{\limsup_{n \to \infty} \tilde{\tau}_n < \infty\}$ the process $\tilde{X}$ possesses the downcrossing property. Consequently, given any $\delta > 0$, we find a continuity point $t \in [\tilde{\tau}, \tilde{\tau} + \delta/2]$ of $\tilde{X}$ with $\tilde{X}_t < 0$. Hence, we can choose a sufficiently small $0 < \epsilon \leq \delta/2$ so that for all large enough $k \geq 1$ we have $X^{n_k}_s \leq \tilde{X}_t/2 < 0$ whenever $s \in [t, t + \epsilon]$. This implies that $\lambda^n(s, X^{n_k}_s, \nu^{n_k}_s) \to \infty$ for all $s \in [t, t + \epsilon]$ by Assumption \ref{ass:struc_limit} \ref{it:intensity_convergence}. Hence, we see 
\begin{equation*}
    \Lambda^{n_k}_{\tilde{\tau} + \delta} \geq \Lambda^{n_k}_{t + \epsilon} \geq \Lambda^{n_k}_{t + \epsilon} - \Lambda^{n_k}_t = \int_t^{t + \epsilon} \lambda^{n_k}(s, X^{n_k}_s, \nu^{n_k}_s) \, \d s \to \infty
\end{equation*}
as $n \to \infty$. But this means $\limsup_{k \to \infty} \tau_{n_k} \leq \tilde{\tau} + \delta$. Upon letting $\delta$ to zero, we find that
\begin{equation*}
    \limsup_{n \to \infty} (\tau_n - \tilde{\tau}_n) = \limsup_{k \to \infty} \tau_{n_k} - \tilde{\tau}_{n_k} = \limsup_{k \to \infty} \tau_{n_k} - \tilde{\tau} \leq 0
\end{equation*}
as required. 

In summary, we have $\limsup_{n \to \infty} \tau_n = \limsup_{n \to \infty} \tilde{\tau}_n = \infty$ or $\lim_{n \to \infty} (\tau_n - \tilde{\tau}_n) = 0$ almost surely, which means a.s.\@ $\tau_n \land T_1 - \tilde{\tau}_n \land T_1 \to 0$. Finally, this gives
\begin{align*}
    \ev \int_0^{T_1} d_2^2(\nu^n_t, \tilde{\nu}^n_t) \, \d t &\leq \int_0^{T_1} 2\ev\bigl[(1 + \lvert X^n_t\rvert^2) \lvert I^n_t - \tilde{I}^n_t\rvert\bigr] \, \d t \\
    &\leq 2\ev\bigl[\bigl(1 + (\lvert X^n \rvert^{\ast}_{T_1})^2\bigr) (\tau_n \land T_1 - \tilde{\tau}_n \land T_1)\bigr],
\end{align*}
where we proceed similarly to \eqref{eq:split_of_metric} in the first inequality. In view of Vitali's convergence theorem, the expression on the right-hand side vanishes as we let $n \to \infty$, since the sequence $((\lvert X^n \rvert^{\ast}_T)^2)_n$ is uniformly integrable. This concludes the proof.
\end{proof}

From Proposition \ref{prop:subprobabilities_converge} we can immediately deduce that $(L^n)_n$ converges weakly to $L$ on $L^2([-1, T + 1])$ along a subsequence, say $(n_k)_k$. Here $L$ is the process defined in the paragraph above Remark \ref{rem:j1_to_m1}. However, as in the proof of Proposition \ref{prop:sequence_tightness} we can show that the sequence $(L^n)_n$ is tight on $D[-1, T + 1]$, where we set $L^n_t = 0$ for $t \in [-1, 0)$ and $L^n_t = L^n_T$ for $t \in (T, T + 1]$. Thus $L^n$ converges weakly on $D[-1, T + 1]$ along a further subsequence. But since elements of $D[-1, T + 1]$ coincide if they are equal as elements of $L^2([0, T])$, the subsequential weak limit (along the subsequence of $(n_k)_k$) of $(L^n)_n$ on $D[-1, T + 1]$ must be $L$. Since the limit of $(L^n)_n$ along all such further subsequences of $(n_k)_k$ coincide, we can conclude that $L^n$ converges weakly to $L$ on $D[-1, T + 1]$ along the original subsequence $(n_k)_k$. 

Next we wish to identify the limit of the integrals $(\int_0^{\cdot} \alpha(t) \, \d L^n_t)_n$, which are tight on $D[-1, T + 1]$ according to the proof of Proposition \ref{prop:sequence_tightness}. Let $\cal{I}_0[-1, T + 1]$ denote the set of nondecreasing c\`adl\`ag functions $\ell \define [-1, T + 1] \to \R$ with $\ell_t = 0$ for $t \in [-1, 0)$. This is a closed subset of $D[-1, T + 1]$ and hence Polish. We define the map $I^{\alpha} \define \cal{I}_0[-1, T + 1] \to D[-1, T + 1]$ by $I^{\alpha}_t(\ell) = \int_0^{t \lor 0} \alpha(s) \, \d \ell_s$ for $t \in [-1, T + 1]$, so that $\int_0^{\cdot} \alpha(t) \, \d L^n_t = I^{\alpha}(L^n)$.

\begin{corollary} \label{cor:conv_of_integrals_nondecreasing}
Under the assumptions of Proposition \ref{prop:subprobabilities_converge}, for any subsequential limit $\pr_0$ of $(\mu^n, W^0)_n$, the processes $(L^n)_n$ and $(\int_0^{\cdot} \alpha(t) \, \d L^n_t)_n$ converge weakly to $L$ and $I^{\alpha}(L)$, respectively, on $D[-1, T + 1]$ along the given subsequence. 
\end{corollary}

Note that $I^{\alpha}(L)$ is well-defined since $L$ is a nondecreasing c\`adl\`ag process with $L_t = 0$ on $[-1, 0)$ by construction. 

\begin{proof}[Proof of Corollary \ref{cor:conv_of_integrals_nondecreasing}]
Assume that the weak convergence of $(\mu^n, W^0)_n$ to $\pr_0$ holds along the subsequence $(n_k)_k$. We discussed the weak convergence of $(L^n)_n$ along $(n_k)_k$ to $L$ on $D[-1, T + 1]$ above the statement of the corollary, so let us turn to the convergence of $(\int_0^{\cdot} \alpha(t) \, \d L^n_t)_n$. Since both $L^n$ and $L$ are nondecreasing, the weak convergence actually holds on $\cal{I}_0[-1, T + 1]$. By Assumption \ref{ass:struc_limit} \ref{it:alpha} the coefficient $\alpha$ is continuous. Hence, we can apply Lemma \ref{lem:conv_int_non_decr} whereby the map $I^{\alpha}$ is $\L_{\ast}(L)$-a.s.\@ continuous on $\cal{I}_0[-1, T + 1]$, so that by the continuous mapping theorem we have $\int_0^{\cdot} \alpha(t) \, \d L^{n_k}_t = I^{\alpha}(L^{n_k}) \Rightarrow I^{\alpha}(L)$ on $D[-1, T + 1]$.
\end{proof}

\subsection{The Controlled Martingale Problem for the Singular Model} \label{sec:mgproblem_singular}

In this subsection we use a slight variation of the martingale problem from Subsection \ref{sec:mgale_problem} to show that $(\mu^n, W^0)_n$ converges subsequentially to an admissible relaxed control rule. Since the loss process $L_t = 1 - \nu_t(\R)$ might jump with positive probability under the limiting probabilities, we remove it from the state $X$ to obtain a continuous process. More precisely, we introduce the differential operator $\bb{L}$ given by
\begin{align*}
    \bb{L}&\varphi(t, x, y, z, v, g) \\
    &= b(t, x, v, g)\partial_x \varphi(x, y, z) + a(t, x, v) \partial_x^2 \varphi(x, y, z) + \frac{1}{2} \partial_y^2 \varphi(x, y, z) + \frac{1}{2} \partial_z^2 \varphi(x, y, z) \\
    &\ \ \ + \sigma(t, x, v) \partial_{xy}^2 \varphi(x, y, z) + \sigma_0(t, x, v) \partial_{xz}^2 \varphi(x, y, z)
\end{align*}
for any twice continuously differentiable function $\varphi \define \R^3 \to \R$ and $(t, x, y, z, v, g) \in [0, T + 1] \times \R^3 \times \cal{M}^1_{\leq 1}(\R) \times G$, where $a(t, x, v) = \frac{1}{2}(\sigma^2(t, x, v) + \sigma_0^2(t, x, v))$. Now, for any $\varphi \in C^2_c(\R^3)$ we define the process $\cal{M}^{\varphi}$ on $\Omega_{\ast}$ by 
\begin{align} \label{eq:mgale_struc}
\begin{split}
    \cal{M}^{\varphi}_t(\omega) &= \varphi\bigl(x_t - I^{\alpha}_t(L(\omega)), w_t, b_t\bigr) \\
    &\ \ \ - \int_{[0, t] \times G} \bb{L}\varphi\bigl(s, x_s - I^{\alpha}_s(L(\omega)), w_s, b_s, \nu_s(\omega), g\bigr) \, \d \mathfrak{g}(s, g)
\end{split}
\end{align}
for $t \in [0, T + 1]$ and $\omega = (x, w, \mathfrak{g}, m, b) \in \Omega_{\ast}$. Here $\mathfrak{g}$ is extended to $(t, g) \in (T, T + 1] \times G$ by $\d \mathfrak{g}(t, g) = \d \delta_{g_0}(g) \d t$ for some arbitrary $g_0 \in G$. We prove that $\cal{M}^{\varphi}$ is a martingale under any probability measure $\pr_{\ast}$ induced by a subsequential limit of $(\mu^n, W^0)_n$, which -- as we demonstrate in the proof of Proposition \ref{prop:conv_to_ad_re_cr_stru} -- implies that $Y = X - I^{\alpha}(L) \define \Omega_{\ast} \to D[-1, T + 1]$ satisfies
\begin{align} \label{eq:continuous_sde_y}
\begin{split}
    \d Y_t &= \int_G b(t, Y_t + I^{\alpha}_t(L), \nu_t, g) \, \d \Gamma(t, g) + \sigma(t, Y_t + I^{\alpha}_t(L), \nu_t) \, \d W_t \\
    &\ \ \ + \sigma_0(t, Y_t + I^{\alpha}_t(L), \nu_t) \, \d B_t
\end{split}
\end{align}
under $\pr_{\ast}$. Note, however, that the martingale problem under consideration is not associated to the SDE
\begin{align} \label{eq:continuous_sde_y_sequence}
\begin{split}
    \d Y^n_t &= b(t, Y^n_t + I^{\alpha}_t(L^n), \nu^n_t, \gamma^n_t) \, \d t + \sigma(t, Y^n_t + I^{\alpha}_t(L^n), \nu^n_t) \, \d W_t \\
    &\ \ \ + \sigma_0(t, Y^n_t + I^{\alpha}_t(L^n), \nu^n_t) \, \d W^0_t
\end{split}
\end{align}
satisfied by the process $Y^n = X^n - I^{\alpha}(L^n)$. Indeed, the processes
\begin{equation*}
    L_t(X^n, W, \Gamma^n, \mu^n, W^0) = \pr\biggl(\inf_{0 \leq s \leq t} X^n_s < 0 \bigg\vert W^0\biggr)
\end{equation*}
and $L^n_t = \pr(\Lambda^n_t \leq \theta \vert W^0)$ are distinct, which leads us to introduce the process $\cal{M}^{n, \varphi}$ on $\Omega$ defined as
\begin{equation*}
    \cal{M}^{n, \varphi}_t = \varphi(Y^n_t, W_t, W^0_t) - \int_0^t \bb{L}\varphi(s, Y^n_s, W_s, W^0_s, \nu^n_s, \gamma^n_s) \, \d s
\end{equation*}
for $t \in [0, T + 1]$ and $\varphi \in C^2_c(\R^3)$. Now, $\cal{M}^{n, \varphi}$ is a martingale for all $\varphi \in C^2_c(\R^3)$ if and only if $Y^n$ solves SDE \eqref{eq:continuous_sde_y_sequence}.

\begin{proposition} \label{prop:conv_to_ad_re_cr_stru}
Let Assumption \ref{ass:struc_limit} be satisfied. Then, for any $\varphi \in C^2_c(\R^3)$, the process $\cal{M}^{\varphi}$ is an $\bb{F}^{\ast}$-martingale under any probability measure $\pr_{\ast}$ induced by a subsequential limit of $(\mu^n, W^0)_n$. In particular, the limit $\pr_0$ is an admissible relaxed control rule in the sense of Definition \ref{def:rel_cntr_struc}.
\end{proposition}

\begin{proof}
Recall the decomposition $X = X_{0-} + \beta + Z + F$ from the proof of Proposition \ref{prop:subprobabilities_converge}. Since $F$ is the subsequential weak limit of the sequence $(I^{\alpha}(L^n))_n$, we can use Corollary \ref{cor:conv_of_integrals_nondecreasing} to identify $F$ as $I^{\alpha}(L)$. Consequently, $Y = X_{0-} + \beta + Z$ has $\pr_{\ast}$-a.s.\@ continuous trajectories, so the same is true for $\cal{M}^{\varphi}$. Moreover, we deduce that along a subsequence 
\begin{equation*}
    (X^n, W, \Gamma^n, \mu^n, W^0, Y^n, \nu^n, I^{\alpha}(L^n)) \Rightarrow (X, W, \Gamma, \mu, B, Y, \nu, I^{\alpha}(L))
\end{equation*}
on $\Omega_{\ast} \times C([-1, T + 1]) \times L^2([-1, T + 1]; \bf{M}^2) \times D[-1, T + 1]$. (Here with slight abuse of notation we use the same symbol $W$ for the idiosyncratic noise on $(\Omega, \F, \pr)$ and $(\Omega_{\ast}, \F_{\ast}, \pr_{\ast})$, which are generally distinct.) Hence, with the choices $E = \R^3 \times \bf{M}^2$ and $\Phi = \bb{L}\varphi$ in Lemma \ref{lem:mgale_integral_converges}, we can deduce that for any bounded and measurable $\Phi \define \Omega_{\ast} \to \R$, which is $\pr_{\ast}$-a.s.\@ continuous, and all $s$, $t \in [0, T + 1]$, we have
\begin{equation} \label{eq:martingale_convergence}
    \ev\bigl[(\cal{M}^{n, \varphi}_t - \cal{M}^{n, \varphi}_s)\Phi(X^n, W, \Gamma^n, \mu^n, W^0)\bigr] \to \ev_{\ast}\bigl[\bigl(\cal{M}^{\varphi}_t(\Theta) - \cal{M}^{\varphi}_s(\Theta)\bigr)\Phi(\Theta)\bigr]
\end{equation}
along a suitable subsequence. Here we used the linear growth of $\Omega_{\ast} \ni \omega \mapsto \cal{M}^{\varphi}_u(\omega)$ for any $u \in [0, T + 1]$ and the boundedness of $\Phi$ together with $\sup_{n \geq 1} \ev\bigl[\sup_{0 \leq u \leq T}\bigl(\lvert X^n_u\rvert^2 + M_2^2(\nu^n_u)\bigr)\bigr] < \infty$ to upgrade weak convergence to convergence in mean, where we recall that $M_2^2(v) = \int_{\R} \lvert x\rvert^2 \, \d v(x)$ for $v \in \bf{M}^2$. We shall use the result \eqref{eq:martingale_convergence} to establish the martingale property of $\cal{M}^{\varphi}$ under $\pr_{\ast}$.

Similarly to the proof of Proposition \ref{prop:conv_to_ad_re_cr}, for any $\Phi \in C_b(\Omega_{\ast})$ and $t \in [0, T + 1]$ let us define the map $\Phi_t \define \Omega_{\ast} \to \R$ by
\begin{equation*}
    \Phi_t(x, w, \mathfrak{g}, m, b) = \Phi(x_{\cdot \land t}, w_{\cdot \land t}, \mathfrak{g}_t, \pi_t^{\#}m, b_{\cdot \land t}),
\end{equation*}
where $\pi_t \define \cal{S} \to \cal{S}$ is given by $(x, w, \mathfrak{g}) \mapsto (x_{\cdot \land t}, w, \mathfrak{g}_t)$. The $\sigma$-algebra $\F^{\ast}_t$ is generated by the random variables $\Phi_t(\Theta)$ for $\Phi \in C_b(\Omega_{\ast})$. Hence, if we can prove that
\begin{equation} \label{eq:martingale_property_dense}
    \ev_{\ast}\bigl[\bigl(\cal{M}^{\varphi}_t(\Theta) - \cal{M}^{\varphi}_s(\Theta)\bigr)\Phi_s(\Theta)\bigr] = 0
\end{equation}
for all $\Phi \in C_b(\Omega_{\ast})$ and all $s \leq t$ in some dense set of times $\bb{T} \subset [0, T + 1]$ which includes $T + 1$, then it clearly follows 
that $\cal{M}^{\varphi}$ is a martingale under $\pr_{\ast}$. Let us choose $\bb{T}$ to be the set of times $t \in [0, T + 1]$ for which $\pr_{\ast}(\Delta X_t = 0) = 1$, which we know to be cocountable and, hence, dense in $[0, T + 1]$ and inclusive of $T + 1$. Now, we have to verify Equation \eqref{eq:martingale_property_dense} for all $s \leq t$ in $\bb{T}$. Note that $\ev\bigl[(\cal{M}^{n, \varphi}_t - \cal{M}^{n, \varphi}_s)\Phi_s(X^n, W, \Gamma^n, \mu^n, W^0)\bigr] = 0$ because $\cal{M}^{n, \varphi}$ is a martingale. Further, we claim that $\Phi_s$ is $\pr_{\ast}$-a.s.\@ continuous for any $s \in \bb{T}$. Then, in view of \eqref{eq:martingale_convergence} we get the desired equality 
\begin{align*}
    0 &= \lim_{n \to \infty} \ev\bigl[(\cal{M}^{n, \varphi}_t - \cal{M}^{n, \varphi}_s)\Phi_s(X^n, W, \Gamma^n, \mu^n, W^0)\bigr] \\
    &= \ev_{\ast}\bigl[\bigl(\cal{M}^{\varphi}_t(\Theta) - \cal{M}^{\varphi}_s(\Theta)\bigr)\Phi_s(\Theta)\bigr].
\end{align*}
Let us prove the claim. Clearly, it is enough to establish continuity of the maps $\omega \mapsto x_{\cdot \land s}$ and $\omega \mapsto \pi_s^{\#}m$ at $\pr_{\ast}$-a.e.\@ $\omega = (x, w, \mathfrak{g}, m, b) \in \Omega_{\ast}$. For the first map $\pr_{\ast}$-almost sure continuity follows from our choice of $s$ as a $\pr_{\ast}$-almost sure continuity point of $X$. To show the desired continuity for $\omega \mapsto \pi_s^{\#}m$, let us choose a Borel set $A \subset \Omega_{\ast}$ of full measure on which the trajectories of $X$ are continuous at $s$. Then for any $\omega_0 \in \Omega_0$, we set $A_{\omega_0} = \{u \in \cal{S} \define (u, \omega_0) \in A\}$ and just as in Equation \eqref{eq:tonelli_argument} find that  $m(A_{\omega_0}) = 1$ for $\pr_0$-a.e.\@ $\omega_0 = (m, b) \in \Omega_0$ and, thus, at $\pr_{\ast}$-a.e.\@ $\omega = (u, \omega_0)$ of $\Omega_{\ast}$. However, the map $\pi_s$ is continuous on $A_{\omega_0}$ by our choice of $A$, so that $\omega \mapsto \pi_s^{\#}m$ is 
continuous at $\pr_{\ast}$-a.e.\@ $\omega \in \Omega_{\ast}$ by the continuous mapping theorem. This proves the claim. Consequently, $\cal{M}^{\varphi}$ is a martingale under $\pr_{\ast}$.

As we mentioned above the martingale $\cal{M}^{\varphi}$ has $\pr_{\ast}$-a.s.\@ continuous trajectories, so by Theorem II.7.2 from \cite{ikeda_sde_1989} it follows that $Y$ solves SDE \eqref{eq:continuous_sde_y}. This implies that $X$ is a solution to McKean--Vlasov SDE \eqref{eq:relaxed_weak_mfl_struc} under $\pr_{\ast}$. From there we proceed as at the end of the proof of Proposition \ref{prop:conv_to_ad_re_cr} to show that $\pr_0$ is a relaxed control rule.
\end{proof}

\subsection{Proof of Theorem \ref{thm:convergence_struc}} \label{sec:proves_singular}

\begin{proof}[Proof of Theorem \ref{thm:convergence_struc}]
By Lemma \ref{prop:sequence_tightness} the sequence $(\mu^n, W^0)_n$ is tight and Proposition \ref{prop:conv_to_ad_re_cr_stru} states that any subsequential limit $\pr_0$ of $(\mu^n, W^0)_n$ is an admissible relaxed control rule. Thus it remains to prove that the costs $J_n(\gamma^n) = \ev[\int_0^{\tau^n \land T} f(t, X^n_t, \nu^n_t, \gamma^n_t) \, \d t + \int_0^T \psi(t, \nu^n_t) \, \d t]$ converge subsequentially to $\cal{J}_{\textup{sg}}(\pr_0)$. However, this is a simple application of Lemma \ref{lem:mgale_integral_converges} along the lines of the proof of Proposition \ref{prop:conv_of_cost}.
\end{proof}

\subsection{The Case of Time-Dependent Coefficients} \label{sec:sing_simple}

We now consider the situation where $\lambda_n$ is just a function of $x$, the coefficients $\sigma$ and $\sigma_0$ only depend on time, and $b(t, x, v, g) = b_0(t) + b_1(t) g$ for measurable functions $b_0$, $b_1 \define [0, T] \to \R$. We can immediately proceed to the proof of Proposition \ref{prop:conv_struc_simple_coeff}.

\begin{proof}[Proof of Proposition \ref{prop:conv_struc_simple_coeff}]
We first show that there exists a sequence of controls admissible controls $(\gamma^n)_n$ for the regularised model with $\limsup_{n \to \infty} J_n(\gamma^n) \leq V_{\textup{sg}}$, where $J_n(\gamma^n)$ is defined in Equation \eqref{eq:cost_reg_seq}. Fix $\epsilon > 0$ and let $\pr_0 \in \P(\Omega_0)$ be a relaxed control rule with $\cal{J}_{\textup{sg}}(\pr_0) \leq V_{\textup{sg}} + \epsilon$. By Proposition \ref{prop:existence_uniqueness_stability}, on the probability space $(\Omega_{\ast}, \F_{\ast}, \pr_{\ast})$, we can uniquely solve the McKean--Vlasov SDE
\begin{equation} \label{eq:sde_simple_approx}
    \d X^n_t = \int_G (b_0(t) + b_1(t) g) \, \d \tilde{\Gamma}(t, g) + \sigma(t) \, \d W_t  + \sigma_0(t) \, \d B_t - \alpha(t) \, \d L^n_t, \quad X^n_0 = X_{0-},
\end{equation}
with $L^n_t = 1 - \ev_{\ast}[e^{-\Lambda^n_t} \vert \F^0_T]$, $\Lambda^n_t = \int_0^t \lambda_n(X^n_s) \, \d s$, and admissible relaxed control $\tilde{\Gamma}$ defined by $\int_{[0, T] \times G} \varphi(t, g) \, \d \tilde{\Gamma}(t, g) = \int_{[0, T] \times G} \varphi(t, \bf{1}_{\tau > t} g) \, \d \Gamma(t, g)$ for $\varphi \in C_b([0, T] \times G)$. It is important to add the indicator $\bf{1}_{\tau > t}$, where we recall $\tau = \inf\{0 < t \leq T + 1 \define X_t \leq 0\}$, to the control $\Gamma$ to ensure that $X^n$ does not have larger running cost than $X$. 
We claim that $X^n_t \geq X_t$ for all $t \in [0, \tau]$, which by Assumption \ref{ass:monotonicity} implies that
\begin{equation} \label{eq:reg_to_sin}
    J_0^n := \ev_{\ast}\biggl[\int_{[0, T] \times G} e^{-\Lambda^n_t} f(t, X^n_t, \nu^n_t, g) \, \d \tilde{\Gamma}(t, g) + \int_0^T \psi(t, \nu^n_t) \, \d t\biggr] \leq \cal{J}_{\textup{sg}}(\pr_0) \leq V_{\textup{sg}} + \epsilon,
\end{equation}
where $\nu^n$ is defined by $\langle \nu^n_t, \varphi\rangle = \ev_{\ast}[e^{-\Lambda^n_t}\varphi(X^n_t) \vert \F^0_T]$ for $\varphi \in C_b(\R)$ and $t \in [0, T]$. We prove the claim through an approximation argument.

We set $L^{n, 0}_t = 0$ and then for $m \geq 1$ define
\begin{equation} \label{eq:approximation_scheme}
    \d X^{n, m}_t = \int_G (b_0(t) + b_1(t) g) \, \d \tilde{\Gamma}(t, g) + \sigma(t) \, \d W_t + \sigma_0(t) \, \d B_t - \alpha(t) \, \d L^{n, m - 1}_t
\end{equation}
with $X^{n, m}_0 = X_{0-}$, and set $L^{n, m}_t = 1 - \ev_{\ast}[e^{-\Lambda^{n, m}_t} \vert \F^0_T]$, where $\Lambda^{n, m}_t = \int_0^t \lambda_n(X^{n, m}_s) \, \d s$. It is not difficult to see that $L^{n, m}_t \leq L_t$ for all $t \in [0, T]$. This is certainly true for $m = 0$. Then, for the purpose of induction let us assume the statement holds for some $m \geq 0$. It follows that $X^{n, m}_t \geq X_t$ on $[0, \tau]$, but $\tau$ is the first time that $X$ visits $(-\infty, 0]$, so that the first hitting time $\tilde{\tau}_{n, m}$ of $X^{n, m}_t$ on $(-\infty, 0]$ must be at least as large as $\tau$. Next, by Assumption \ref{ass:struc_limit} \ref{it:intensity_convergence}, $\lambda^n$ vanishes on $[0, \infty)$, so $\Lambda^{n, m}_t = 0$ on $[0, \tilde{\tau}^{n, m} \land (T + 1)]$. Hence, we see that
\begin{align*}
    L^{n, m}_t = 1 - \ev_{\ast}[e^{-\Lambda^{n, m}_t} \vert \F^0_T] \leq \pr_{\ast}(\tilde{\tau}_{n, m} \leq t \vert \F^0_T) \leq \pr_{\ast}(\tau \leq t \vert \F^0_T) = L_t,
\end{align*}
which concludes the induction. From $L^{n, m}_t \leq L_t$ we obtain that $X^{n, m}_t \geq X_t$ for all $t \in [0, T]$ and $m \geq 1$. Now we simply take the limit as $m \to \infty$. Since $(L^{n, m})_m$ is an increasing sequence, it converges and its limit coincides with $L^n$. Similarly, the processes $(X^{n, m})_m$ form a decreasing sequence with limit $X^n$, which implies the desired inequality $X^n_t \geq X_t$ for $0 \leq t \leq T$ and proves the claim.

Next we invoke Theorem \ref{thm:convergence_ps}, whereby the minimal cost $V_n$ achievable for McKean--Vlasov SDE \eqref{eq:mfl_sequence} with intensity function $\lambda^n$ over admissible strong controls (i.e.\@ those which are square-integrable and $\bb{F}^{\xi, W, W^0}$-progressively measurable) and over admissible relaxed control in the sense of Definition \ref{def:rel_cntrl} coincides. Consequently, for any sequence $(\epsilon_n)_n$ of positive numbers tending to zero, we can find a sequence of admissible strong controls $(\gamma^n)_n$, such that if $\tilde{X}^n$ denotes the solution to McKean--Vlasov SDE \eqref{eq:mfl_sequence} with control $\gamma^n$, then
\begin{equation}
     J_n(\gamma^n) \leq V_n + \epsilon_n \leq J_0^n + \epsilon_n \leq V_{\textup{sg}} + \epsilon + \epsilon_n,
\end{equation}
where $J_n(\gamma^n)$ is defined in Equation \eqref{eq:cost_reg_seq} and the last inequality is simply Equation \eqref{eq:reg_to_sin}. Taking the limit superior as $n \to \infty$ implies $\limsup_{n \to \infty} J_n(\gamma^n) \leq V_{\textup{sg}} + \epsilon$.

Let us next deduce the lower bound $V_{\textup{sg}} \leq \liminf_{n \to \infty} J_n(\gamma^n)$. Take a subsequence $(n_k)_k$ which realises the limit inferior of the sequence $(J_n(\gamma^n))_n$ and for which the random variables $(\cal{L}(\tilde{X}^{n_k}, W, \Gamma^{n_k} \vert W^0), W^0)_k$ converge weakly to some relaxed control rule $\pr_0^{\ast} \in \P(\Omega_0)$ (for the singular model), where $\Gamma^n$ is the relaxed control associated to $\gamma^n$. This can be achieved according to Theorem \ref{thm:convergence_struc}, which further states that the costs $J_{n_k}(\gamma^{n_k})$ converge to the cost $\cal{J}_{\textup{sg}}(\pr_0^{\ast})$ of the limiting relaxed control rule. Consequently, we have
\begin{equation*}
    V_{\textup{sg}} \leq \cal{J}_{\textup{sg}}(\pr_0^{\ast}) \leq \lim_{k \to \infty} J_{n_k}(\gamma^{n_k}) = \liminf_{n \to \infty} J_n(\gamma^n) \leq \limsup_{n \to \infty} J_n(\gamma^n) \leq V_{\textup{sg}} + \epsilon.
\end{equation*}
Letting $\epsilon \to 0$ we obtain $\lim_{n \to \infty} V_n = \lim_{n \to \infty} J_n(\gamma^n) = V_{\textup{sg}}$, which concludes the proof.
\end{proof}

Finally, we assume that $b_1 = 0$, so we are considering a problem without controls.

\begin{proof}[Proof of Proposition \ref{eq:minimal_convergence}]
Let $L$ be the process constructed above Proposition \ref{eq:minimal_convergence} and define $X$ through Equation \eqref{eq:minimal}. Our first goal is to verify that for every time $t \in [-1, T + 1]$ the sequence $(L^n_t)_n$ converges to $L_t$ on the set $\{\Delta L_t = 0\}$. Note that we may assume that $t \in [0, T + 1)$, since the convergence holds for $t \in [-1, 0)$ and $t = T + 1$ by definition. We split the convergence into two inequalities: $\lim_{n \to \infty} L^n_t \leq L_t$ and $\lim_{n \to \infty} L^n_t \geq L_t$. The first one holds even outside of $\{\Delta L_t = 0\}$. Indeed, let $\epsilon > 0$ and fix $s \in (t, T + 1] \cap \bb{Q}$ with $\ell_s - L_t \leq \epsilon$, which is possible by construction of $L$. Then by monotonicity of $L^n$ we have
\begin{equation*}
    \lim_{n \to \infty} L^n_t \leq \lim_{n \to \infty} L^n_s = \ell_s \leq L_t + \epsilon.
\end{equation*}
Letting $\epsilon \to 0$ gives $\lim_{n \to \infty} L^n_t \leq L_t$. Next, assuming that $\Delta L_t = 0$, for any $\epsilon > 0$ we can find $s \in [-1, t) \cap \bb{Q}$ with $L_t - \ell_s \leq \epsilon$. Then
\begin{equation*}
    \lim_{n \to \infty} L^n_t \geq \lim_{n \to \infty} L^n_s = \ell_s \geq L_t - \epsilon.
\end{equation*}
Now we let $\epsilon \to 0$ again to see that $\lim_{n \to \infty} L^n_t \leq L_t$ and thus $\lim_{n \to \infty} L^n_t = L_t$. Since the jumps in $X$ are caused by $L$, we can transfer the convergence of $(L^n)_n$ to $(X^n)_n$, so altogether we have $(X^n_t, L^n_t) \to (X_t, L_t)$ on $\{\Delta L_t = 0\}$.

By \cite[Theorem 12.5.1 (iv)]{whitt_stoch_limits_2002} a sequence of nondecreasing maps $(x^n)_n$ in $D[-1, T + 1]$ converges to $x \in D[-1, T + 1]$ if $x^n_t \to x_t$ for a dense set of times $t \in [-1, T + 1]$ including $-1$ and $T + 1$. This is a.s.\@ the case for $(L^n)_n$, so that $L^n \to L$ a.s.\@ in $D[-1, T + 1]$. Since $X^n$ is the sum of a continuous process (that does not change with $n$) and $L^n$, we deduce that $X^n \to X$ a.s.\@ in $D[-1, T + 1]$ as well.

It remains to show that $(X, L)$ is minimal. For that, we use the approximation scheme from \eqref{eq:approximation_scheme}. Let $(X', L')$ be any other solution to McKean--Vlasov SDE \eqref{eq:minimal}. Then defining $(X^{n, m}, L^{n, m})$ as in \eqref{eq:approximation_scheme} we can inductively show that $L^{n, m}_t \leq L'_t$ for all $t \in [0, T + 1]$. Since $L^{n, m}_t \to L^n_t$ for all $t \in [0, T + 1]$ we have that
\begin{equation*}
    L_t = \lim_{n \to \infty} L^n_t = \lim_{n \to \infty}\lim_{m \to \infty} L^{n, m}_t \leq L'_t
\end{equation*}
on $\{\Delta L_t = 0\}$. From this we can conclude that $L_t \leq L'_t$ for all $t \in [0, T + 1]$. Since $(X', L')$ was arbitrary it follows that $(X, L)$ is minimal.
\end{proof}

\begin{remark}
The minimality of $(X, L)$ does not only hold among the class of strong solutions $(X', L')$ for which $L'$ is $\bb{F}^{W^0}$-adapted. Indeed, if we assume that $\bb{G} = (\cal{G}_t)_{0 \leq t \leq T + 1}$ is another filtration that satisfies the conditions outlined at the beginning of Subsection \ref{sec:existence_uniqueness} and that $(X', L')$ is a solution of McKean--Vlasov SDE \eqref{eq:minimal} with $L' = \pr(\tau' \leq t \vert \cal{G}_T)$, where $\tau' = \inf\{t > 0 \define X'_t \leq 0\}$, then the arguments from the proof of Proposition \ref{eq:minimal_convergence} above still apply. Hence, we can conclude that $L_t \leq L'_t$ for all $t \in [0, T + 1]$. 

In fact, since $(X, L)$ is a strong solution, we can transfer it to any other probability space that can accommodate random variables with the law $\L(\xi, W, W^0)$. We simply use that $(X, L) = S(\xi, W, W^0)$ for a measurable function $S \define \R \times C([0, T + 1]) \times C([0, T + 1]) \to D[0, T + 1] \times D[0, T + 1]$. So if $(X', L')$ is a weak solution to McKean--Vlasov SDE \eqref{eq:minimal} (in the above sense that $L'$ is not necessarily adapted to the filtration of the common noise but a possibly larger filtration satisfying the conditions outlined at the beginning of Subsection \ref{sec:existence_uniqueness}) on any given probability space equipped with random variables $\xi'$, $W'$, and $B'$ such that $(\xi', W', B') \sim (\xi, W, W^0)$, then $S(\xi', W', B')$ is a strong solution to McKean--Vlasov SDE \eqref{eq:minimal} and $L'$ will be lower bounded by the second component of $S(\xi', W', B')$.
\end{remark}

\section{Numerical Simulation} \label{sec:num_sim}

In this section, we discuss a scheme for simulating the strong formulation of the regularised mean-field control problem \eqref{eq:mfl}. We do this in the context of the model for government interventions from Subsection \ref{sec:toy_model}. Since the coefficients and cost functions of this model satisfy the assumptions of Theorem \ref{thm:all_equiv}, the particle (or finite banking system) system is well approximated by the strong mean-field control problem. We use a policy gradient method to search for optimal controls. That is, we approximate the dynamics \eqref{eq:mfl} for a parametrised control $\gamma$, compute the associated cost $J(\gamma)$, and then update the parameters of $\gamma$ based on the gradient of $J(\gamma)$ with respect to the parameters. 

\subsection{Finite Element Scheme for the Stochastic Fokker--Planck Equation}

It is extremely costly to use a particle system to simulate \eqref{eq:mfl}, as this requires a two-fold approximation: for each realisation of the common noise we need to estimate $\nu_t = \ev[$ through a Monte-Carlo approximation based on $K \geq 1$ samples of the idiosyncratic noise $W$. Then, we approximate the cost $J(\gamma)$ by averaging over these estimates for $K_0 \geq 1$ realisations of the common noise $W^0$. The resulting Monte-Carlo error is at best of order $(K_0 \land K)^{-1/2}$ while the computational complexity is $O(K_0 K)$. This is too expensive, so instead we proceed via the stochastic Fokker--Planck equation satisfied by the flow $\nu = (\nu_t)_{0 \leq t \leq T}$. 

To pursue this approach, we have to assume that the control $\gamma$ is in \textit{closed-loop} form, i.e.\@ there exists a measurable function $g \define [0, T] \times \R \times \cal{M}^1_{\leq 1}(\R) \to G$ with $\gamma_t = g(t, X_t, \nu_t)$ for $t \in [0, T]$. From Theorem \ref{thm:all_equiv} we know that the optimal cost for (relaxed) closed-loop controls coincides with the value of the strong mean-field control problem, justifying our approach. However, to ensure that McKean--Vlasov SDE \eqref{eq:mfl} is (strongly) well-posed for any closed-loop control $g$, we have to assume that $g$ satisfies the assumptions imposed on $\lambda$ in Assumptions \ref{ass:red_form_cont_prob} \ref{it:growth} and \ref{it:continuity}. Well-posedness then follows from Proposition \ref{prop:existence_uniqueness_stability}. We call such $g$ \textit{strong closed-loop controls} and denote their collection by $\bb{G}_{\text{cl}}$. The \textit{cost functional} for closed-loop controls $g \in \bb{G}_{\text{cl}}$ is defined by
\begin{align} \label{eq:cost_cl}
    J_{\text{cl}}(g) &= \ev\biggl[\int_0^T e^{-\Lambda_t} f\bigl(t, X_t, \nu_t, g(t, X_t, \nu_t)\bigr) \, \d t + \psi(\nu_T)\biggr] \notag \\
    &= \ev\biggl[\int_0^T \bigl\langle \nu_t, f\bigl(t, \cdot, \nu_t, g(t, \cdot, \nu_t)\bigr)\bigr\rangle \, \d t + \psi(\nu_T)\biggr].
\end{align}
As mentioned above, the flow $\nu$ satisfies a stochastic Fokker--Planck equation, which takes the form
\begin{equation} \label{eq:mfl_spde}
    \d \langle \nu_t, \varphi\rangle = \bigl\langle \nu_t, \L\varphi\bigl(t, \cdot, \nu_t, g(t, \cdot, \nu_t)\bigr)\bigr\rangle \, \d t + \langle \nu_t, \sigma_0(t, \cdot, \nu_t) \partial_x\varphi\rangle \, \d W^0_t
\end{equation}
for $\varphi \in C_c^2(\R)$, with initial condition $\nu_0 = \L(\xi)$. The definition of the differential operator $\L$ is given in Equation \eqref{eq:diff_op_sfpe}. For details of the derivation of this SPDE, see the proof of Theorem \ref{thm:all_equiv} in Subsection \ref{eq:proof_all_equiv}.

\begin{remark}
Under Assumption \ref{ass:red_form_cont_prob} it is possible to show that the conditional subprobability $\nu_t = \ev[e^{-\Lambda_t} \delta_{X_t} \vert W^0]$ is the unique strong solution to SPDE \eqref{eq:mfl_spde}, where uniqueness is understood in a pathwise sense. This can be achieved through arguments similar to those in \cite{coghi_sfpe_unique_2019}.
\end{remark}

We will next propose a discretisation scheme for SPDE \eqref{eq:mfl_spde}. Our discussion will be informal and we do not justify the method theoretically. We discretise SPDE \eqref{eq:mfl_spde} in space by a finite element scheme with hat basis and in time by a semi-implicit Euler--Maruyama scheme. More precisely, fix $x_0 < 0 < x_{n + 1}$ and let $x_0 < x_1 < \dots x_{n + 1}$ be an equidistant grid on $[x_0, x_{n + 1}]$ with mesh size $h = x_1 - x_0$. For $1 \leq i \leq n$ we define the hat functions $v_i$ by $v_i(x) = 0 \lor (1 - h^{-1} \lvert x - x_i\rvert)$. Next, let $0 = t_0 < t_1 < \dots < t_m = T$ be a uniform time grid with mesh size $t_1$. We define the map $A \define [0, T] \times \cal{M}^1_{\leq 1}(\R) \times \cal{B}_1(\R; G) \to \R^{n \times n}$ by
\begin{align} \label{eq:bilin_mat}
\begin{split}
    A_{ij}(t, \nu, u) &= \int_{\R}  \Bigl(b(t, x, \nu, u(x)) - \alpha(t, x, \nu)\langle \nu, \lambda(t, \cdot, \nu)\rangle\Bigr) v_j(x) \partial_x v_i(x) \, \d x \\
    &\ \ \ - \int_{\R} \partial_x( a(t, x, \nu)  v_j(x)) \partial_x v_i(x) \, \d x - \int_{\R} \lambda(t, x, \nu)v_j(x) v_i(x) \, \d x
\end{split}
\end{align}
for $1 \leq i, j \leq n$. Here $\cal{B}_1(\R; G)$ denotes the space of measurable functions $\R \to G$ of at most linear growth. In practice the integrals in \eqref{eq:bilin_mat} have to be approximated via some numerical integration method. 

Next, we consider the SDE
\begin{equation} \label{eq:spde_semidiscrete}
    \d \rho^{ni}_t = \sum_{j = 1}^n A_{ij}(t, \nu^n_t, g(t, \cdot, \nu^n_t)) \rho^{nj}_t \, \d t + \sum_{j = 1}^n \bigl\langle v_j, \sigma_0(t, \cdot, \nu^n_t) \partial_x v_i\bigr\rangle \rho^{nj}_t \, \d W^0_t
\end{equation}
with initial condition $\rho^{ni}_0 = \langle \nu_0, v_i\rangle$. The processes $\rho^{ni}$ are the coefficients of the finite-element approximation of the density of $\nu_t$ in the hat basis $v_1$,~\ldots, $v_n$. Unfortunately, there is no guarantee that $\rho^{ni}_t \geq 0$ nor that $r^n_t = \sum_{n = 1}^n (\rho^{ni}_t)_+ \leq 1$. So if we were to naively define the finite-element approximation $\rho^n_t$ of the density through $\langle \rho^n_t, v_i\rangle = \rho^{ni}_t$, then $\d \nu^n_t(x) = \rho^n_t(x) \, \d x$ is not a subprobability distribution. Instead, we define $\rho^n_t$ as the unique element in the linear span of $(v_1, \dots, v_n)$ with $\langle \rho^n_t, v_i) = (\rho^{ni})t)_+/(r^n_t \lor 1)$ and define $\nu^n_t$ by $\d \nu^n_t(x) = \rho^n_t(x) \, \d x$.

We discretise SDE \eqref{eq:spde_semidiscrete} in time with a semi-implicit Euler-Maruyama scheme: we set $p^{ni}_0 = \rho^{ni}_0$ and for $k = 0$,~\ldots, $m - 1$ let $p^{ni}_{k + 1}$, $i = 1$,~\ldots, $n$, be the solution to
\begin{align} \label{eq:spde_discrete}
\begin{split}
    p^{ni}_{k + 1} + t_1 \sum_{j = 1}^n A_{ij}&\bigl(t_k, \nu^{nm}_k, g(t_k, \cdot, \nu^{nm}_k)\bigr) p^{nj}_{k + 1} \\
    &= p^{ni}_k + \sum_{j = 1}^n \bigl\langle v_j, \sigma_0(t_k, \cdot, \nu^{nm}_k) v_i\bigr\rangle p^{nj}_k (W^0_{t_{k + 1}} - W^0_{t_k}).
\end{split}
\end{align}
As above the measure $\nu^{nm}_k$ has density $p^n_k$ which is defined as the unique element in the linear span of $(v_j)_{1 \leq j \leq n}$ with $\langle p^n_k, v_j\rangle = (p^{nj}_k)_+/(r^{nm}_k \lor 1)$,
where $r^{nm}_k = \sum_{j = 1}^n (p^{nj}_k)_+$. With this discretisation we can approximate the cost $J_{\text{cl}}(g)$ by
\begin{equation} \label{eq:cost_discrete}
     \ev \biggl[\sum_{k = 0}^{m - 1} \int_{\R}f\bigl(t_k, x, \nu^{nm}_k, g(t_k, x, \nu^{nm}_k)\bigr) \d \nu^{nm}_k(x) + \psi(\nu^{nm}_m)\biggr].
\end{equation}
We can estimate this expectation using the Monte-Carlo method with $K$ independent samples of the common noise $W^0$.

Our discretisation scheme involves a threefold approximation: in time, space, and the randomness of the common noise. For each of the $K$ realisations of the common noise and at each of the $m$ time steps we have to solve the linear system \eqref{eq:spde_discrete}. Since the matrix $A\bigl(t_k, \nu^{nm}_k, g(t_k, \cdot, \nu^{nm}_k)\bigr)$ is tridiagonal this can be achieved with a cost of $O(n)$. Assuming that the complexity of computing the entries of $A\bigl(t_k, \nu^{nm}_k, g(t_k, \cdot, \nu^{nm}_k)\bigr)$ is of order $n$, we therefore get a total cost of $O(nmK)$.

\subsection{Policy Gradient Method for the Discretisation of the Stochastic Fokker--Planck Equation}

To make use of the gradient descent algorithm to search for optimal closed-loop controls, we require an (approximate) parametrisation of the space $\bb{G}_{\text{cl}}$. An application of the Stone-Weierstrass theorem implies that functions $g \define [0, T] \times \R \times \cal{M}^1_{\leq 1}(\R) \to G$ of the form $(t, x, v) \mapsto g_1(t, x, \langle v, g_0\rangle)$ for Lipschitz continuous maps $g_0 \define \R \to \R^{d_0}$ and $g_1 \define [0, T] \times \R \times \R^{d_0} \to \R$ with $d_0 \geq 1$ are locally dense in $\bb{G}_{\text{cl}}$ with respect to the supremum metric. For a given $d_0 \geq 1$ we can approximate the functions $g_0$ and $g_1$ by neural networks $g^{\vartheta}_0 \define \R \to \R^{d_0}$ and $g^{\vartheta}_1 \define [0, T] \times \R \times \R^{d_0} \to \R$ with parameters $\vartheta \in \R^{d_{\vartheta}}$. Let us fix such a $d_0 \geq 1$ as well as neural networks $g^{\vartheta}_0$ and $g^{\vartheta}_1$ and set $g^{\vartheta}(t, x, v) = g^{\vartheta}_1(t, x, \langle v, g^{\vartheta}_0\rangle)$. Then we can compute the finite element approximation of $\nu$ for the control $g = g^{\vartheta}$ for $K$ realisations of the common noise. We denote the resulting Monte-Carlo estimate of the discretised cost by $\hat{J}_{\text{cl}}(\vartheta)$. For a given learning rate $\eta > 0$, we can iteratively update the parameters $\vartheta$ according to $\vartheta \mapsto \vartheta - \eta \nabla_{\vartheta} \hat{J}_{\text{cl}}(\vartheta)$. We terminate this scheme after a predetermined number of iterations.

\subsection{Numerical Experiments}

For our numerical experiments we assume that the coefficients $\sigma$, $\sigma_0$, and $\alpha$ are constant, $\lambda(t, x, v) = \lambda_0 \max\{-x, 0\}$ for some positive number $\lambda_0$, and $b(t, x, v, g) = g$, which leads to the state
\begin{equation*}
    \d X_t = \bigl(g^{\vartheta}(t, X_t, \nu_t) - \alpha \lambda_0 \langle \nu_t, (\cdot)_-\rangle\bigr) \, \d t + \sigma \, \d W_t + \sigma_0 \, \d W^0_t, \quad \d \Lambda_t = \lambda_0 (X_t)_- \, \d t
\end{equation*}
with $\nu_t = \ev[e^{-\Lambda_t} \delta_{X_t} \vert W^0]$ for $\vartheta \in \R^{d_{\vartheta}}$. We set $G = [0, \infty)$, $f(t, x, v, g) = w g$, and $\psi(v) = 1- v(\R)$, so the cost functional is
\begin{equation*}
    J_{\text{cl}}(g^{\vartheta}) = \ev\biggl[\int_0^T w e^{-\Lambda_t} g^{\vartheta}(t, X_t, \nu_t) \, \d t + L_T\biggr],
\end{equation*}
In the context of the financial model discussed in Section \ref{sec:toy_model}, the weight $w > 0$ captures the trade-off between the cost of capital injection and the cost of insolvencies. We fix $w = 5$ and set the time horizon $T = 1$. The choice of $w$ ensures that the running and terminal cost are of the same order of magnitude and that the neural networks produce sensible results. The remaining parameters vary based on the experiment. Following \cite{cuchiero_optimal_bailout_2023}, we assume the initial condition $\xi$ is distributed according to a gamma distribution with density
\begin{equation*}
    f(x; k, \theta) = \frac{1}{\Gamma(k)\theta^k}x^{k - 1}e^{-\frac{x}{\theta}}.
\end{equation*}
We set the shape $k = 6$ and scale $\theta = 1/60$, meaning that mass is very tightly concentrated around the mean $1/10$ and all on the positive half-line. (Here $\theta$ is of course not the exponential random variable from above but the scale parameter of the gamma distribution.) For the finite element scheme we truncate space to the interval $[-1, 1]$.

We choose feedforward neural networks with two hidden layers for the functions $g^{\vartheta}_0$ and $g^{\vartheta}_1$. The dimension $d_0$ as well as the size of the hidden layers of $g^{\vartheta}_0$ are equal to $10$ while the number of units in the hidden layers of $g^{\vartheta}_1$ is $50$. We follow a two-level approach to train the neural networks in each experiment: first, we train for $500$ epochs with $n = 128$ grid points, $m = 128$ time steps, and $K = 128$ realisations of the common noise and then a further $100$ epochs with $n = 256$, $m = 256$, and $K = 256$. The code for our numerical experiments can be found on GitHub\footnote{\url{https://github.com/philkant/systemic-risk-model}}.

Next, we discuss the different numerical experiments. We frame the experiments in terms of the financial model that described in Section \ref{sec:toy_model}. The particles represent commercial banks with mutual obligations. Killing corresponds to a default of an institution and the contagion results from failure to repay obligations. The control $g^{\vartheta}$ can be interpreted as a capital injection by a government, intended to prevent default cascades.

\textit{Feedback.}  First, we vary the feedback parameter $\alpha$ within $\{0.5, 1.0, 1.5, 2.0, 2.5\}$ while fixing $\lambda_0 = 10$ and $\sigma = \sigma_0 = 0.1$. In Figure \ref{fig:feedback_cost} we report the convergence of the stochastic gradient descent algorithm for the two levels of discretisation. Further, we compare how the running cost
\begin{equation*}
    \ev \int_0^T w e^{-\Lambda_t} g^{\vartheta}(t, X_t, \nu_t) \, \d t = \ev \int_0^T w \langle \nu_t, g^{\vartheta}(t, \cdot, \nu_t)\rangle \, \d t
\end{equation*}
and terminal cost $\ev L_T = 1 - \ev \nu_T(\R)$ change for increasing levels of feedback $\alpha$. As expected their sum $J_{\text{cl}}(g^{\vartheta})$ becomes larger. However, this trend is fully driven by the running cost whereas the terminal cost at first shrinks for higher $\alpha$. We believe this is because at low levels of feedback, individual insolvencies barely impact the larger network, so the central agent is content with letting institutions default. As the interconnectedness mounts, insolvencies become more costly as they precipitate further defaults, so the controller injects larger levels of capital to prevent default cascades from materialising. 

\begin{figure}[tb] 
    \makebox[\linewidth][c]{
    \begin{subfigure}[b]{0.5\columnwidth}
        \centering
        \includegraphics[width=\columnwidth]{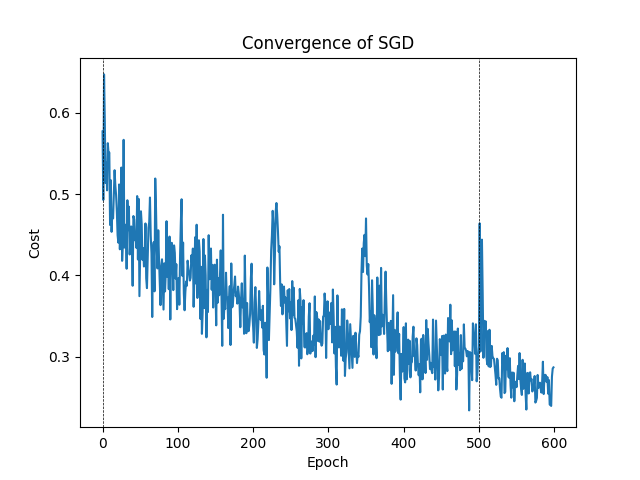}
    \end{subfigure}

    \hspace{-0.5cm}

    \begin{subfigure}[b]{0.5\columnwidth}
        \centering
        \includegraphics[width=\columnwidth]{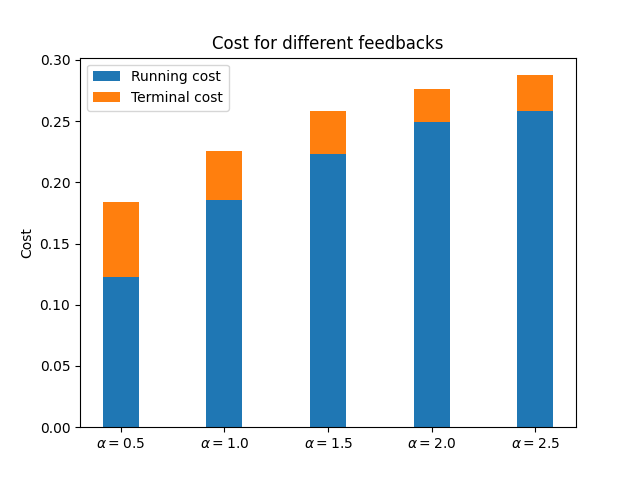}
    \end{subfigure}
    }
    \caption{Convergence of stochastic gradient descent and cost for different feedback parameters $\alpha$. The dotted line in the left panel indicates the split between the two training regimes with different values for $n$, $m$, and $K$.}
    \label{fig:feedback_cost}
 \end{figure}

\begin{figure}[tb] 
    \makebox[\linewidth][c]{
    \hspace{0.55cm}
    \begin{subfigure}[b]{0.53\columnwidth}
        \centering
        \includegraphics[width=\columnwidth]{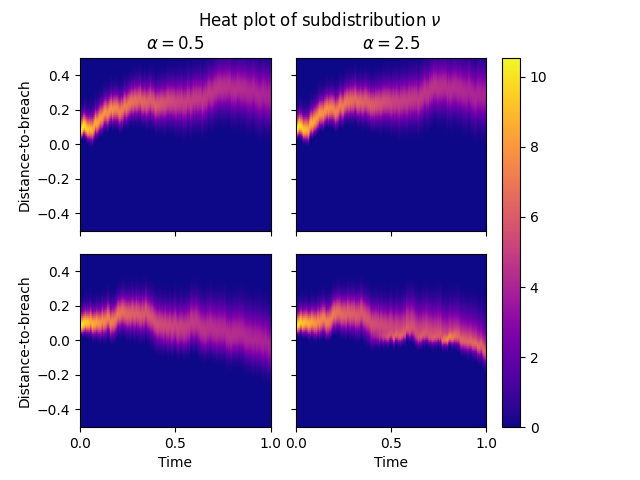}
    \end{subfigure}

    \hspace{-1.1cm}

    \begin{subfigure}[b]{0.53\columnwidth}
        \centering
        \includegraphics[width=\columnwidth]{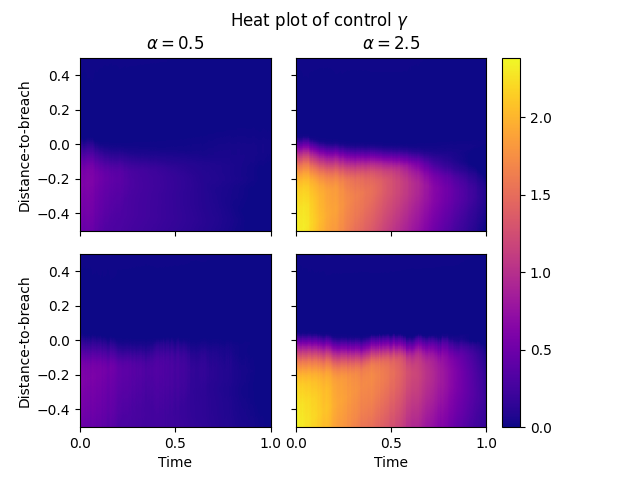}
    \end{subfigure}
    }
    \caption{Heat plot of flow of subprobability distributions $\nu = (\nu_t)_{0 \leq t \leq T}$ and the control $(g^{\vartheta}(t, \cdot, \nu_t))_{0 \leq t \leq T}$ for $\alpha = 0.5$, $\alpha = 2.5$ and two different realisations of the common noise $W^0$ in the upper and lower panel.}
    \label{fig:feedback_heatplot}
\end{figure}

We provide a heat plot of the flow of subprobability distributions $\nu = (\nu_t)_{0 \leq t \leq T}$ (or rather their densities) and the control $(g^{\vartheta}(t, \cdot, \nu_t))_{0 \leq t \leq T}$ in Figure \ref{fig:feedback_heatplot} for two realisations of the common noise. Bright colours indicate a high concentration of institutions with a given distance-to-breach $X_t$ (cf.\@ Section \ref{sec:toy_model}). The idiosyncratic noise spreads banks apart, while their mean is directed by the common noise, the interaction term $\int_0^t \alpha \lambda_0 \langle \nu_s, (\cdot)_- \rangle \, \d s$, and the control. The realisation of the second noise (bottom row) is more adverse as it pushes banks below the breach threshold. For the larger feedback parameter $\alpha = 2.5$ we can clearly perceive the controller's activity, which attempts to keep institutions above the capital threshold, and as a result squeezes the banks together (cf.\@ the bottom right depiction on the left plot of Figure \ref{fig:feedback_heatplot}). 

The different levels of activity by the central agent for $\alpha = 0.5$ and $\alpha = 2.5$ are contrasted by the heat plot of the control in the right panel of Figure \ref{fig:feedback_heatplot}. The controls are of bang-bang type, meaning that a bailout only occurs when a bank breaches the capital threshold, at which point the central agent injects heavily. Finally, we see that an adverse realisation of the noise leads to prolonged activity by the controller.

\textit{Correlation.} Next, we vary the correlation $\rho = \frac{\sigma_0^2}{\sigma^2 + \sigma_0^2}$ between the institutions' outside assets within $\{0.2, 0.4, 0.6, 0.8\}$ while keeping the total volatility $\sigma^2 + \sigma_0^2$ fixed at $0.04$ and $\lambda_0 = 10$. We contrast the impact of increased volatility with that of mounting interconnectedness in Figure \ref{fig:correlation_cost}. While the controller is capable of fighting the endogenous feedback even for large $\alpha$, there is a marked phase transition for asset correlation. As the correlation moves up from $0.6$ to $0.8$, the central agent precipitously cuts capital injections. We interpret this as an inability to combat the exogenous effect of the common exposures once they reach sufficiently large levels. In other words, if the controller keeps the number of insolvencies in check, the endogenous feedback $\int_0^t \alpha \lambda_0 \langle \nu_s, (\cdot)_- \rangle \, \d s$ is contained. However, this has no impact on the exogenous common exposures, which can still drag on banks' balance sheets.

\begin{figure}[tb] 
    \makebox[\linewidth][c]{
    \begin{subfigure}[b]{0.5\columnwidth}
        \centering
        \includegraphics[width=\columnwidth]{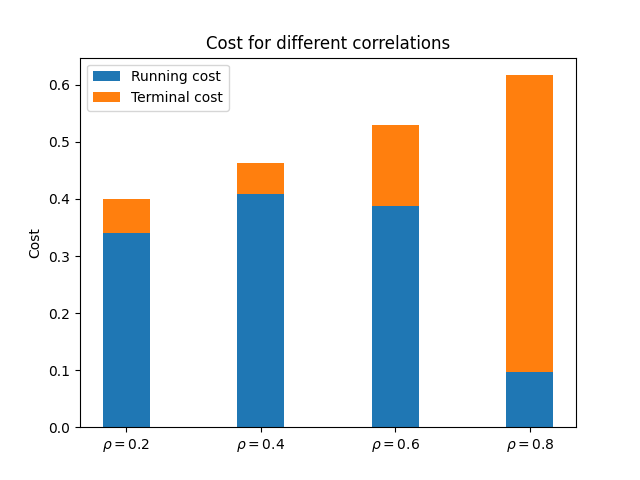}
    \end{subfigure}

    \hspace{-0.5cm}

    \begin{subfigure}[b]{0.5\columnwidth}
        \centering
        \includegraphics[width=\columnwidth]{plots/plot_feedback_cost.png}
    \end{subfigure}
    }
    \caption{Cost for different correlation parameters $\rho$ versus different feedback parameters $\alpha$ (note the different scales).}
    \label{fig:correlation_cost}
\end{figure}




\textit{Intensity.} Finally, we vary the intensity $\lambda_0$ within $\{5, 10, 25, 50\}$ while fixing $\alpha = 1.5$ and $\sigma = \sigma_0 = 0.1$. As the heat plot of the controls in Figure \ref{fig:intensity_heatplot} indicates, the control boundary in green, which separates the regions of activity and inactivity of the controller, is pushed outwards for larger intensities, because for large $\lambda_0$ banks default almost immediately upon breaching the capital threshold. By lifting the control boundary the central agent lowers the risk of breaches by financial institutions. This is also illustrated by the heat plot of the subprobability distribution on the left panel of Figure \ref{fig:intensity_heatplot}, where we plot the flow $\nu$ for different $\lambda_0$ and the same realisation of the common noise. For $\lambda_0 = 5$ and to a lesser extent $\lambda_0 = 10$ breaches are tolerated, whereas for higher intensity parameters institutions are kept above or even at a distance from the threshold.

\begin{figure}[tb] 
    \makebox[\linewidth][c]{
    \hspace{0.55cm}
    \begin{subfigure}[b]{0.53\columnwidth}
        \centering
        \includegraphics[width=\columnwidth]{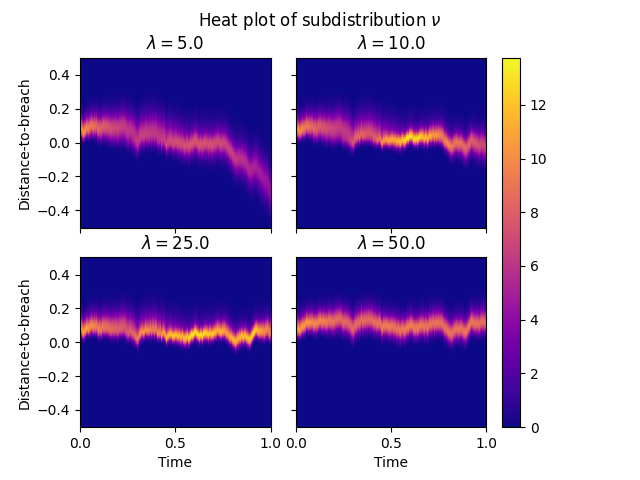}
    \end{subfigure}

    \hspace{-1.1cm}

    \begin{subfigure}[b]{0.53\columnwidth}
        \centering
        \includegraphics[width=\columnwidth]{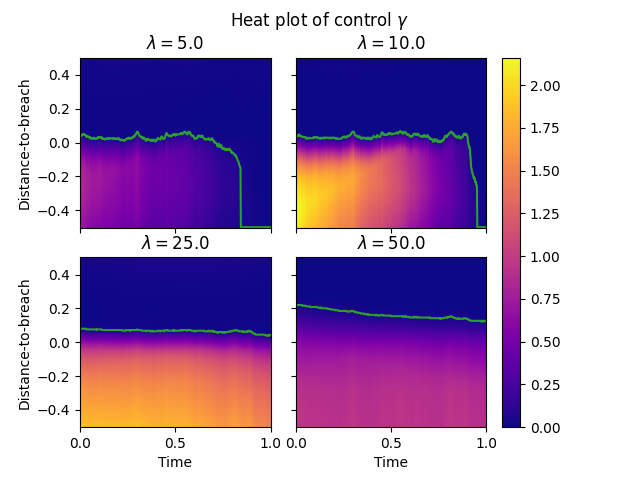}
    \end{subfigure}
    }
    \caption{Heat plot of flow of subprobability distributions $\nu = (\nu_t)_{0 \leq t \leq T}$ and the control $(g^{\vartheta}(t, \cdot, \nu_t))_{0 \leq t \leq T}$ for different $\lambda_0$. The green lines in the plot on the right-hand side indicate the control boundary.}
    \label{fig:intensity_heatplot}
\end{figure}

\appendix

\section{Appendix}

\subsection{Existence, Uniqueness and Stability for McKean--Vlasov SDEs with Common Noise under Local Lipschitz Condition} \label{app:mvsde_loc_lip}

In the following, let us fix a probability space $(\Omega, \F, \pr)$ equipped with two filtrations $\bb{G}$ and $\bb{F}$ with $\cal{G}_t \subset \F_t$ for $t \in [0, T]$, an $\F_0$-measurable $\R^{d_X}$-valued random variable $\xi$ for $d_X \geq 1$, and two $d_W$-dimensional $\bb{F}$-Brownian motions $W$ and $W^0$ for $d_W \geq 1$. We assume that $W^0$ is adapted to $\bb{G}$ and the pair $(\xi, W)$ is independent of $\cal{G}_T$. Finally, we require that for all $t \in [0, T]$ we have $\pr(A \vert \cal{G}_t) = \pr(A \vert \cal{G}_T)$ a.s.\@ for all $A \in \F_t \lor \sigma(W)$.

Let $\bb{M}_T^2(G)$ be the space of square-integrable measures on $[0, T] \times G$, for a non-empty and closed subset $G \subset \R^d$ with $d \geq 1$, which have the Lebesgue measure as their first marginal. We equip $\bb{M}_T^2(G)$ with the $2$-Wasserstein distance. We say that a $\bb{M}_T^2(G)$-valued random variable $\Gamma$ is $\bb{F}$-progressively measurable if for all $t \in [0, T]$, the random variable $\Gamma([0, s] \times A)$ is $\F_t$-measurable for any $s \in [0, t]$ and $A \in \cal{B}(G)$. Fix such a random measure $\Gamma$ and assume that $\ev \int_0^T \lvert g \rvert^2 \, \d \Gamma(t, g) < \infty$. Then we consider the McKean--Vlasov SDE
\begin{equation} \label{eq:local_lipschitz}
    \d X_t = \int_G b(t, X_t, \mu_t, g) \, \d \Gamma(t, g) + \sigma(t, X_t, \mu_t) \, \d W_t + \sigma_0(t, X_t, \mu_t) \, \d W^0_t,
\end{equation}
started from $X_0 = \xi$ with $\mu_t = \L(X_t \vert \cal{G}_T)$. The coefficients are functions $b \define [0, T] \times \R^{d_X} \times \P^2(\R^{d_X}) \times G \to \R^{d_X}$ and $\sigma$, $\sigma_0 \define [0, T] \times \R^{d_X} \times \P^2(\R^{d_X}) \to \R^{d_X \times d_W}$. Here $\P^2(\R^{d_X})$ denotes the space of square-integrable probability measures on $\R^{d_X}$ equipped with the $2$-Wasserstein distance $W_2$. We also set $M_2(\mu) = \bigl(\int_{\R^{d_X}} \lvert x \rvert^2 \, \d \mu(x)\bigr)^{1/2}$ for $\mu \in \P^2(\R^{d_X})$.

The pair $(X, \mu)$ is called a \textit{strong solution} of McKean--Vlasov SDE \eqref{eq:local_lipschitz} if (i) the process $X$ is a strong solution to \eqref{eq:local_lipschitz} when viewed as an SDE with random coefficients, the randomness coming from the mean-field component $\mu$ and the control $\Gamma$, and (ii) $\mu_t$ is the conditional law of $X_t$ with respect to $\cal{G}_T$, so there is no additional information other than $\bb{G}$ in the conditioning. Under the above hypothesis $\pr(A \vert \cal{G}_t) = \pr(A \vert \cal{G}_T)$ for all $A \in \F_t \lor \sigma(W)$, we have $\mu_t = \L(X_t \vert \cal{G}_T) = \L(X_t, \vert \cal{G}_t)$, so that $\mu$ is $\bb{G}$-adapted.

\begin{assumption} \label{ass:local_lipschitz}
Let $b \define [0, T] \times \R^{d_X} \times \P^2(\R^{d_X}) \times G \to \R^{d_X}$ and $\sigma$, $\sigma_0 \define [0, T] \times \R^{d_X} \times \P^2(\R^{d_X}) \to \R^{d_X \times d_W}$ be measurable and $\L(\xi) \in \P^2(\R^{d_X})$. We assume there exists a constant $C > 0$, such that
\begin{enumerate}[noitemsep, label = (\roman*)]
    \item \label{it:growth_general} the norms of the coefficients $\sigma$ and $\sigma_0$ are bounded by $C$ and for all $t$, $x$, $\mu$, $g$ we have
    \begin{align*}
        \lvert b(t, x, \mu, g) \rvert \leq C(1 + \lvert x \rvert + M_2(\mu) + \lvert g\rvert);
    \end{align*}
    \item \label{it:continuity_general} the coefficient $b$ is continuous in $g$ and for $h = \sigma$, $\sigma_0$ and all $t$, $x$, $x'$, $\mu$, $\mu'$, $g$ we have
    \begin{align*}
        \lvert b(t, x, \mu, g) - b(t, x', \mu'&, g)\rvert + \lvert h(t, x, \nu) - h(t, x', \nu')\rvert \\
        &\leq C\bigl(1 + M_2^2(\mu) \lor M_2^2(\mu')\bigr)(\lvert x - x'\rvert + W_2(\mu, \mu')).
    \end{align*}
\end{enumerate}
\end{assumption}

\begin{proposition} \label{prop:local_lipschitz}
Let Assumption \ref{ass:local_lipschitz} be satisfied. Then for any $\bb{F}$-progressively measurable $\bb{M}_T^2(G)$-valued $\Gamma$ with $\ev \int_0^T \lvert g \rvert^2 \, \d \Gamma(t, g) < \infty$, the McKean--Vlasov SDE \eqref{eq:local_lipschitz} has a unique strong solution. If $b$ does not depend on $g$, i.e.\@ $b(t, x, m, g) = b_0(t, x, m)$ for some function $b_0 \define [0, T] \times \R^{d_X} \times \P^2(\R^{d_X}) \to \R^{d_X}$, then $\mu_t = \L(X_t \vert W^0)$ for all $t \in [0, T]$ almost surely.

Moreover, if $(\epsilon_n)_n$ is a sequence of positive real numbers tending to zero and $(\Gamma_n)_{n \geq 1}$ is a sequence of $\bb{M}_T^2(G)$-valued $\bb{F}$-progressively measurable processes such that $\ev W_2^2(\Gamma_n, \Gamma) \to 0$, then $\ev(\lvert X^n - X \rvert^{\ast}_T)^2 \to 0$ as $n \to \infty$. Here $X^n$ is the unique strong solution to SDE \eqref{eq:local_lipschitz} started from $\xi$ at $\epsilon_n$ with input $\Gamma_n$. 
\end{proposition}

In the statement of the proposition starting $X^n$ from $\xi$ at time $\epsilon_n$ means that $X^n$ solves the SDE \eqref{eq:local_lipschitz} on the interval $[\epsilon_n, T]$ with initial condition $X^n_{\epsilon_n} = \xi$. We extend $X^n$ to the whole interval $[0, T]$ by setting $X^n_t = \xi$ for $t \in [0, \epsilon_n)$. Here for a continuous function $h \define [0, T] \to \R^{d_X}$ we denote by $\lvert h\rvert^{\ast}_T$ the running supremum $\sup_{0 \leq t \leq T} \lvert h_t\rvert$.

\begin{proof}[Proof of Proposition \ref{prop:local_lipschitz}]
\textit{Existence and uniqueness}: Our proof strategy is as follows: we introduce a sequence of nonlinearities $b^n$, $\sigma^n$, and $\sigma_0^n$, $n \geq 1$, which coincide with $b$, $\sigma$, and $\sigma_0$ for elements $\mu \in \P^2(\R^{d_X})$ with $M_2(\mu) \leq n$ and are uniformly Lipschitz continuous in $x$ and $\mu$. Then, we replace the coefficients of the McKean--Vlasov SDE \eqref{eq:local_lipschitz} by $b^n$, $\sigma^n$, and $\sigma_0^n$ and obtain a unique strong solution $\tilde{X}^n$. Here we use the same notion of strong solution as for McKean--Vlasov SDE \eqref{eq:local_lipschitz}. Clearly, up to the first time $\varrho_n$ at which the square of the second moment of $\L(\tilde{X}^n_t \vert \cal{G}_T)$ exceeds $n$, $X^n$ solves the original SDE \eqref{eq:local_lipschitz}. Here we use that the conditional law $\L(\tilde{X}^n_{t \land \varrho_n} \vert \cal{G}_T)$ of the process stopped at $\varrho_n$   coincides with the conditional law $\L(\tilde{X}^n_s \vert \cal{G}_T)\vert_{s = t \land \varrho_n}$ stopped at $\varrho_n$ because $\varrho_n$ is a $\bb{G}$-stopping. Hence, setting $X$ equal to $\tilde{X}^n$ on the interval $[0, \varrho_n]$ yields the desired solution. Let us expand on this sketch.

Let $B_n(0)$ denote the centered ball of radius $n$ in the space $L^2(\Omega, \F, \pr; \R^{d_X})$ and let $\pi_n$ be the projection onto $B_n(0)$, that is
\begin{equation*}
    \pi_n(X) = \biggl(\frac{n}{\lVert X \rVert_{L^2}} \land 1\biggr) X \quad \text{for } X \in L^2(\Omega, \F, \pr; \R^{d_X}).
\end{equation*}
The map $\pi_n$ is $1$-Lipschitz continuous and we can push it forward to $\P^2(\R^{d_X})$ through the map $L^2(\Omega, \F, \pr; \R^{d_X}) \to \P^2(\R^{d_X})$, $X \mapsto \L(X)$. That is, we define $\pi^{\ast}_n \define \P^2(\R^{d_X}) \to \P^2(\R^{d_X})$ by $\pi^{\ast}_n(\mu) = \L(\pi_n(X))$, where $X$ is a random variable in $L^2(\Omega, \F, \pr; \R^{d_X})$ with law $\mu$. Note that this map is well-defined, since $\L(\pi_n(X)) = \L(\pi_n(Y))$ whenever $X$, $Y \in L^2(\Omega, \F, \pr; \R^{d_X})$ have the same law and it inherits the $1$-Lipschitz continuity from $\pi_n$. Indeed, let $\mu$ and $\nu \in \P^2(\R^{d_X})$ and choose random variables $X$ and $Y$ with laws $\mu$ and $\nu$, respectively, such that $W_2^2(\mu, \nu) = \ev\lvert X - Y\rvert^2$. Then we have
\begin{equation*}
    W_2^2(\pi^{\ast}_n \mu, \pi^{\ast}_n\nu) \leq \ev \lvert \pi_n(X) - \pi_n(Y)\rvert^2 \leq \ev\lvert X - Y\rvert^2 = W_2^2(\mu, \nu).
\end{equation*}
Now, we set $b^n(t, x, \mu, g) = b(t, x, \pi^{\ast}_n \mu, g)$ and similarly define $\sigma^n$ and $\sigma_0^n$. Then it follows from Assumption \ref{ass:local_lipschitz} \ref{it:continuity_general} that
\begin{align*}
    \lvert b^n(t, x_1, \mu_1, g)& - b^n(t, x_2, \mu_2, g)\rvert \\
    &= \bigl\lvert b(t, x_1, \pi^{\ast}_n \mu_1, g) - b(t, x_2, \pi^{\ast}_n \mu_2, g)\bigr\rvert \\
    &\leq C_b \bigl(1 + M_2^2(\pi^{\ast}_n \mu_1) \lor M_2^2(\pi^{\ast}_n \mu_2)\bigr)\bigl(\lvert x_1 - x_2 \rvert + W_2(\pi^{\ast}_n \mu_1, \pi^{\ast}_n\mu_2)\bigr) \\
    &\leq C_b (1 + n) (\lvert x_1 - x_2 \rvert + W_2(\mu_1, \mu_2))
\end{align*}
with obvious modifications for $\sigma^n$ and $\sigma_0^n$. Moreover, all three coefficients satisfy the linear growth condition stated in Assumption \ref{ass:local_lipschitz} \ref{it:growth_general}, so that the McKean--Vlasov SDE 
\begin{equation*}
    \d \tilde{X}^n_t = \int_G b^n(t, \tilde{X}^n_t, \tilde{\mu}^n_t, g) \, \d \Gamma(t, g) + \sigma^n(t, \tilde{X}^n_t, \tilde{\mu}^n_t) \, \d W_t + \sigma_0^n(t, \tilde{X}^n_t, \tilde{\mu}^n_t) \, \d W^0_t,
\end{equation*}
started from $\tilde{X}^n_0 = \xi$ with $\tilde{\mu}^n_t = \L(\tilde{X}^n_t \vert \cal{G}_T)$ is in the standard Lipschitz regime. Consequently, it has a unique strong solution $\tilde{X}^n$ by \cite[Theorem A.3]{djete_mvoc_dpp_2022}.

Define $\varrho_n = \inf\{t > 0 \define M_2^2(\tilde{\mu}^n_t) \geq n\}$. Then $\varrho_n$ is a $\bb{G}$-stopping time, so that
\begin{equation*}
    \tilde{\mu}^n_{t \land \varrho_n} = \L(\tilde{X}^n_s \vert \cal{G}_T)\vert_{s = t \land \varrho_n} = \L(\tilde{X}^n_{t \land \varrho_n} \vert \cal{G}_T).
\end{equation*}
Consequently, if we set $Y^n = \tilde{X}^n_{\cdot \land \varrho_n}$ and $\nu^n = \tilde{\mu}^n_{\cdot \land \varrho_n}$, then for $t \in [0, \varrho_n]$ it holds that
\begin{equation*}
    Y^n_t = \int_{[0, t] \times G} b(s, Y^n_s, \nu^n_s, g) \, \d \Gamma(s, g) + \int_0^t \sigma(s, Y^n_s, \nu^n_s) \, \d W_s + \int_0^t \sigma_0(s, Y^n_s, \nu^n_s) \, \d W^0_s
\end{equation*}
with $\nu^n_t = \L(Y^n \vert \cal{G}_T)$. Hence, $Y^n$ is a solution to the McKean--Vlasov SDE \eqref{eq:local_lipschitz} on the interval $[0, \varrho_n]$. Since $Y^n$ is unique, we have that $Y^n = Y^m$ on $[0, \varrho_m]$ and $\varrho_m \leq \varrho_n$ whenever $n \geq m$. This allows us to define the process $X$ by $X_t = Y^n_t$ if $t \in [0, \varrho_n]$. Clearly, $X$ solves SDE \eqref{eq:local_lipschitz} up to $\varrho = \lim_{n \to \infty} \varrho_n$. However, $\varrho_n$ is the first time that $\mu_t = \L(X_t \vert \cal{G}_T)$ is equal to or greater than $n$ and a simple Gr\"onwall estimate, which relies on the linear growth condition from Assumption \ref{ass:local_lipschitz} \ref{it:growth_general} implies that $\ev \sup_{0 \leq s \leq t \land \varrho} M_2^2(\mu_s) < \infty$ for any $t \geq 0$. Now assume that $\pr(\varrho < \infty) > 0$. Then, we can find a large enough $t \geq 0$ such that $p = \pr(\varrho \leq t) > 0$. Further, we can choose $n \geq 1$ with $pn > \ev \sup_{0 \leq s \leq t \land \varrho} M_2^2(\mu_s)$. But this leads to the contradiction
\begin{equation*}
    pn > \ev \sup_{0 \leq s \leq t \land \varrho} M_2^2(\mu_s) \geq \ev[\bf{1}_{\varrho_n \leq t} M_2^2(\mu_{\varrho_n})] \geq \pr(\varrho_n \leq t) n \geq \pr(\varrho \leq t) n \geq pn.
\end{equation*}
Hence, it must hold that $\varrho = \infty$ almost surely. 

Finally, we remark that uniqueness of $X$ immediately follows from the uniqueness of the $Y^n$ on the intervals $[0, \varrho_n]$.

$\bb{F}^{W^0}$-\textit{adaptedness}: In the case that $b$ does not depend on $g$, let us replace the filtration $\bb{G}$ by $\bb{F}^{W^0}$. Note that $\bb{F}^{W^0}$ verifies the same conditions as $\bb{G}$ outlined at the beginning of the section. Consequently, we can apply the existence and uniqueness result we just established to obtain a strong solution $(X^0, \mu^0)$ to McKean--Vlasov SDE \eqref{eq:local_lipschitz} with $\mu^0_t = \L(X^0_t \vert W^0)$. Now, if we can prove that $\mu^0_t = \L(X^0_t \vert \cal{G}_T)$, then $(X^0, \mu^0)$ also solves the McKean--Vlasov SDE \eqref{eq:local_lipschitz} for the filtration $\bb{G}$, so by the pathwise uniqueness we must have $(X, \mu) = (X^0, \mu^0)$, which implies that $\mu_t = \mu^0_t = \L(X^0_t \vert W^0) = \L(X_t \vert W^0)$ as required. To show $\mu^0_t = \L(X^0_t \vert \cal{G}_T)$, we use that $X^0$ is a strong solution to McKean--Vlasov SDE \eqref{eq:local_lipschitz} so that there exists a measurable function $S \define \R^{d_X} \times C([0, T]) \times C([0, T]) \to \R^{d_X}$ with $X^0_t = S(X_0, W, W^0)$. Since $(X_0, W) \perp \cal{G}_T$ and $W^0$ is $\cal{G}_T$-measurable we have for any bounded and measurable map $\varphi \define \R^{d_X} \to \R$ that
\begin{align*}
    \ev[\varphi(X^0_t) \vert \cal{G}_T] &= \ev[\varphi(S(X_0, W, W^0) \vert \cal{G}_T] \\
    &= \int_{\R^{d_X} \times C([0, T])} \ev[\varphi(S(x, w, W^0) \vert \cal{G}_T] \, \d \L(X_0, W)(x, w) \\
    &= \int_{\R^{d_X} \times C([0, T])} \ev[\varphi(S(x, w, W^0) \vert W^0] \, \d \L(X_0, W)(x, w) \\
    &= \ev[\varphi(S(X_0, W, W^0) \vert W^0] \\
    &= \ev[\varphi(X^0_t) \vert W^0].
\end{align*}
This readily implies $\L(X^0_t \vert \cal{G}_T) = \L(X^0_t \vert W^0) = \mu^0_t$.

\textit{Stability}: Let $(X^n)_n$ be as in the statement of the proposition. Note that the family $\ev(\lvert X^n \rvert^{\ast}_T)^2$, $n \geq 1$, is uniformly integrable as the same is true for $(M_2^2(\Gamma_n))_n$ by the $L^2$-convergence of the sequence $(\Gamma_n)_n$. Now, we proceed in two steps. First, we prove that the difference $\lvert X^n_t - X_t\rvert^2$ tends to zero in expectation. Then, combining this with the easily established tightness of $(X^n)_n$ on $C([0, T])$ we conclude that $\ev(\lvert X^n - X \rvert^{\ast}_T)^2 \to 0$.

For the first step, recall that $\varrho_k = \inf\{t > 0 \define M_2^2(\tilde{\mu}^k_t) \geq k\} = \inf\{t > 0 \define M_2^2(\mu_t) \geq k\}$. Now, using Assumption \ref{ass:local_lipschitz}, elementary SDE estimates show that
\begin{align*}
    \ev \lvert &X^n_{t \land \varrho_k} - X_{t \land \varrho_k}\rvert^2 \\
    &\leq C_k \int_0^t \ev \lvert X^n_{s \land \varrho_k} - X_{s \land \varrho_k}\rvert^2 \, \d s + 2 \ev \biggl\lvert \int_{[0 , t\land \varrho_k] \times G} b(s, X_s, \mu_s, g) \, \d (\Gamma_n - \Gamma)(s, g)\biggr\rvert^2 \\
    &\ \ \ + C \epsilon_n (1 + \ev(\lvert X^n \rvert^{\ast}_T)^2) + C \ev\biggl\lvert \int_{[0, \epsilon_n] \times G} \lvert g \rvert^2 \, \d \Gamma(t, g)\biggr\rvert^2,
\end{align*}
with constants $C$, $C_k > 0$ which do not depend on $n \geq 1$. The expressions in the third line capture the difference in starting time between $X^n$ and $X$, whereas the integrals in the second line bound the error between the two solutions as it unfolds over time. Let us define the quantities $M_n = C \epsilon_n (1 + \ev(\lvert X^n \rvert^{\ast}_T)^2) + C \ev\bigl\lvert \int_{[0, \epsilon_n] \times G} \lvert g \rvert^2 \, \d \Gamma(t, g)\bigr\rvert^2$ and $D^{n, k}_t = \bigl\lvert \int_{[0 , t\land \varrho_k] \times G} b(s, X_s, \mu_s, g) \, \d (\Gamma_n - \Gamma)(s, g)\bigr\rvert^2$. By Gr\"onwall's inequality we have
\begin{equation} \label{eq:bound_on_second_moment_with_k}
    \ev \lvert X^n_{t \land \varrho_k} - X_{t \land \varrho_k}\rvert^2 \leq C(M_n + \ev D^{n, k}_t) + C_k \int_0^t \ev D^{n, k}_s \, \d s,
\end{equation}
where we enlarge $C_k$ if necessary. Note that the integral term appears because $s \mapsto \ev D^{n, k}_s$ is not necessarily nondecreasing. Clearly, $M_n$ tends to zero as $n \to \infty$. Moreover, since $\Gamma^n$ converges to $\Gamma$ on $\bb{M}_T^2(G)$ in $L^2$, it follows from Lemma \ref{lem:mgale_integral_converges} that $D^{n, k}_t \to 0$ in probability. 
Moreover, we have that
\begin{equation*}
    D^{n, k}_s \leq C \bigl(1 + (\lvert X \rvert^{\ast}_T)^2 + \ev[(\lvert X \rvert^{\ast}_T)^2 \vert \cal{G}_T]\bigr) + C \biggl\lvert \int_{[0, T] \times G} \lvert g \rvert^2 \, \d (\Gamma_n + \Gamma)(s, g)\biggr\rvert^2,
\end{equation*}
and the expression on the right-hand side is uniformly integrable in $n 
\geq 1$ by the uniform square-integrability of the family $(\Gamma_n)_n$. Thus, Vitali's convergence theorem implies that the expression on the right-hand side of Equation \eqref{eq:bound_on_second_moment_with_k} converges to zero as $n \to \infty$, whence $\lim_{n \to \infty}\ev \lvert X^n_{t \land \varrho_k} - X_{t \land \varrho_k}\rvert^2 = 0$. 

From this, we can deduce that $\lvert X^n_t - X_t \rvert^2$ tends to zero in probability. Indeed, fix $\delta$, $\epsilon > 0$ and choose $k \geq 1$ such that $\pr(\varrho_k \leq t) \leq \epsilon/2$ and then $N \geq 1$ large enough such that $\ev \lvert X^n_{t \land \varrho_k} - X_{t \land \varrho_k}\rvert^2 \leq \delta \epsilon/2$ for $n \geq N$. Then, we have that
\begin{align*}
    \pr(\lvert X^n_t - X_t \rvert^2 > \delta) &\leq \pr\bigl(\lvert X^n_{t \land \varrho_k} - X_{t \land \varrho_k}\rvert^2 > \delta\bigr) + \pr(\varrho_k \leq t) \\
    &\leq \frac{1}{\delta} \ev \lvert X^n_{t \land \varrho_k} - X_{t \land \varrho_k}\rvert^2 + \frac{\epsilon}{2} = \epsilon
\end{align*}
for $n \geq N$. Since $\delta$ and $\epsilon$ were arbitrary this shows $\lvert X^n_t - X_t \rvert^2$ converges to zero in probability. Finally, we know that the variables $(\lvert X^n \rvert^{\ast}_T)^2$, $n \geq 1$, are uniformly integrable, so by Vitali's convergence theorem we conclude that $\ev \lvert X^n_t - X_t \rvert^2 \to 0$ as $n \to \infty$.

Next, we strengthen the pointwise $L^2$-convergence to convergence in $L^2$-$\sup$. Since the diffusion coefficients are bounded and the sequence $(\Gamma_n)_n$ is uniformly square-integrable, it follows easily from Kolmogorov's tightness criterion that the family $(X^n)_n$ is tight on $C([0, T])$. Consequently, for any $\delta > 0$ we can find a modulus of continuity $\omega_{\delta} \define [0, \infty) \to [0, \infty)$, such that if we let $\cal{K}_{\delta}$ be the set of function in $C([0, T])$ with modulus $\omega_{\delta}$, it holds that $\pr(X^n \in \cal{K}_{\delta}) \geq 1 - \delta$ and $\pr(X \in \cal{K}_{\delta}) \geq 1 - \delta$. Next, fix an $\epsilon > 0$. Since the sequence $(\lvert X^n \rvert^{\ast}_T)^2$, $n \geq 1$ is uniformly integrable, we can pick $\delta > 0$ sufficiently small such that
\begin{equation} \label{eq:unif_int_est_epsilon}
    \ev\bigl[\bf{1}_{X^n \notin \cal{K}_{\delta}} (\lvert X^n \rvert^{\ast}_T)^2 + \bf{1}_{X \notin \cal{K}_{\delta}} (\lvert X \rvert^{\ast}_T)^2\bigr] \leq \frac{\epsilon}{3}
\end{equation}
for all $n$ larger than some appropriately chosen $N \geq 1$. Finally, let us choose $\epsilon_0 > 0$ with the property that $\omega_{\delta}(\epsilon_0) \leq \sqrt{\epsilon/6}$ and define $t_k = k \epsilon_0$ for $k = 1$,~\ldots, $T/\epsilon_0$ (where we assume for simplicity that $T/\epsilon_0$ is an integer). Then, increasing $N$ if necessary, we have that $\ev \lvert X^n_{t_k} - X_{t_k}\rvert^2 \leq \frac{\epsilon \epsilon_0}{6T}$ for $k = 1$,~\ldots, $T/\epsilon_0$ and all $n \geq N$ by the pointwise $L^2$-convergence. Now, on the set $\{X^n, X \in \cal{K}_{\delta}\}$, it holds that
\begin{equation*}
    (\lvert X^n - X \rvert^{\ast}_T) \leq \biggl(\sup_{1 \leq k \leq T/\epsilon_0} \lvert X^n_{t_k} - X_{t_k}\rvert + \frac{\sqrt{\epsilon}}{\sqrt{6}} \biggr)^2 \leq 2 \sup_{1 \leq k \leq T/\epsilon_0} \lvert X^n_{t_k} - X_{t_k}\rvert^2 + \frac{\epsilon}{3}.
\end{equation*}
Hence, we see that for all $n \geq N$,
\begin{align*}
    \ev(\lvert X^n - X \rvert^{\ast}_T)^2 &\leq \ev\bigl[\bf{1}_{\{X^n, X \in \cal{K}_{\delta}\}}(\lvert X^n - X \rvert^{\ast}_T)^2\bigr] + \ev\bigl[\bf{1}_{\{X^n, X \notin \cal{K}_{\delta}\}}(\lvert X^n - X \rvert^{\ast}_T)^2\bigr] \\
    &\leq 2\ev \sup_{1 \leq k \leq T/\epsilon_0} \lvert X^n_{t_k} - X_{t_k}\rvert^2 + \frac{\epsilon}{3} + \frac{\epsilon}{3} \\
    &\leq 2\sum_{k = 1}^{T/\epsilon_0} \ev\lvert X^n_{t_k} - X_{t_k}\rvert^2 + \frac{2\epsilon}{3} = \epsilon,
\end{align*}
where we applied the bound from \eqref{eq:unif_int_est_epsilon} in the second equality. This implies the desired convergence of $\ev(\lvert X^n - X \rvert^{\ast}_T)^2$ to zero.
\end{proof}

\section{Uniqueness for Linear Stochastic Fokker--Planck Equation with Random Coefficients} \label{app:spde_uniqueness}

Let $(\Omega, \F, \bb{F}, \pr)$ be a filtered probability space which supports an $\bb{F}$-Brownian motion $B = (B_t)_{t \in [0, T]}$. We consider the linear stochastic Fokker--Planck equation
\begin{equation} \label{eq:spde}
    \d \langle \nu_t , \varphi\rangle = \langle \nu_t, \L \varphi(t, \cdot) \rangle \, \d t + \langle \nu_t, \sigma_t \partial_x \varphi\rangle \, \d B_t
\end{equation}
for $\varphi \in C_b^2(\R)$, with initial condition $\nu_0 \in \P^2(\R)$. The random differential operator $\L$ acts on $\varphi \in C_b^2(\R)$ by
\begin{equation*}
    \L\varphi(t, x) =  - \lambda_t(x) \varphi(x) + \bigl(b^0_t(x) + b^1_t(x)\bigr) \partial_x \varphi(x) + a_t(x) \partial_x^2 \varphi(x)
\end{equation*}
for $\bb{F}$-progressively measurable random functions $\lambda \define [0, T] \times \Omega \times \R \to [0, \infty)$ and $b^0$, $b^1$, $a$, and $\sigma \define [0, T] \times \Omega \times \R \to \R$. We make the following assumptions.

\begin{assumption} \label{ass:spde}
Let $\lambda \define [0, T] \times \Omega \times \R \to [0, \infty)$ and $b^0$, $b^1$, $a$, and $\sigma \define [0, T] \times \Omega \times \R \to \R$ be $\bb{F}$-progressively measurable. We assume that there exist $c$, $C > 0$ and a square-integrable random variable $\zeta \geq 0$ such that
\begin{enumerate}[noitemsep, label = (\roman*)]
    \item the coefficients $b^1$, $a$, and $\sigma$ are bounded by $C$ and for all $t$, $\omega$, $x$ we have
    \begin{equation*}
        \lvert b^0_t(x)\rvert + \lvert \lambda_t(x)\rvert \leq C (\zeta + \lvert x\rvert);
    \end{equation*}
    \item for all $h \in \{\lambda, b^0, a\}$ and all $t$, $\omega$, $x$, $y$ we have
    \begin{equation*}
        \lvert h_t(x) - h_t(y) \rvert \leq C \lvert x - y\rvert;
    \end{equation*}
    \item \label{it:nondegenerate} for all $t$, $\omega$, $x$ we have $a_t(x) - \frac{\sigma_t(x)}{2} \geq c$.
\end{enumerate}
\end{assumption}

For notational convenience, we introduce the shorthand $b_t = b^0_t + b^1_t$. We shall establish a uniqueness result for SPDE \eqref{eq:spde} in the class of subprobability measure-valued solutions. A \textit{subprobability measure-valued solution} is any $\bb{F}$-progressively measurable $C([0, T]; \cal{M}_{\leq 1}^2(\R))$-valued process $\nu = (\nu_t)_{t \in [0, T]}$ with $\ev \sup_{0 \leq t \leq t} M_2^2(\nu_t) < \infty$, which satisfies SPDE \eqref{eq:spde}. Here for $p \geq 1$ and any measure $m$ in $\R$, we define $M_p(m) = (\int_{\R} \lvert x\rvert^p \, \d \mu(x))^{1/p}$.  

\begin{theorem} \label{thm:uniqueness_spde}
Let Assumption \ref{ass:spde} be satisfied. Then, for a given initial condition in $\P^2(\R)$, SPDE \eqref{eq:spde} has a unique subprobability measure-valued solution.
\end{theorem}

The idea of the proof of Theorem \ref{thm:uniqueness_spde} is to use suitable $L^2$-energy estimates for SPDE \eqref{eq:spde}. This strategy cannot be directly implemented since the initial condition $\nu_0 \in \P^2(\R)$ is not necessarily an element of $L^2(\R)$. So instead, we derive energy estimate for the cumulative distribution function $F_t(x) = \nu_t((-\infty, x]))$ of a solution $\nu = (\nu_t)_{0 \in [0, T]}$. However, while $F_t$ is a function as opposed to a measure, it is only locally square-integrable since it does not decay at infinity. Thus, we will proceed via energy estimates in weighted $L^2$-spaces, which we introduce next. 

For $\eta > \R$, we define the weight $w_{\eta}(x) = \exp(-\eta \sqrt{1 + \lvert x\rvert^2})$ and introduce the Hilbert space $L^2_{\eta}(\R)$ consisting of all measurable functions $f \define \R \to \R$ such that
\begin{equation*}
    \lVert f \rVert_{\eta} = \biggl(\int_{\R} \lvert f(x)\rvert^2 w_{\eta}(x) \, \d x\biggr)^{1/2} < \infty.
\end{equation*}
We equip $L_{\eta}^2(\R)$ with the inner product
\begin{equation*}
    (f, g) \mapsto \langle f, g\rangle_{\eta} = \int_{\R} f(x) g(x) w_{\eta}(x) \, \d x.
\end{equation*}
Next, let $g_{\epsilon} \define \R \to \R$ be the heat kernel, given by $g_{\epsilon} = (2\pi \epsilon)^{-1/2} e^{- x^2/(2\epsilon)}$, and define the operator $T_{\epsilon}$ which acts on finite signed measures $m$ on $\R$ by $T_{\epsilon} m(x) = \int_{\R} g_{\epsilon}(x - y) \, \d m(y)$ for $x \in \R$. Its action on measurable functions $f \define \R \to \R$ is analogously defined by $T_{\epsilon} f(x) = \int_{\R} g_{\epsilon}(x - y) f(y) \, \d y$ whenever the integral is defined for all $x \in \R$. Typically, $f$ will be an element of the space $L_{\eta}^2(\R)$, in which case the integral exists. Lastly, for a finite signed measure $m$ on $\R$, we denote by $\lvert m\rvert$ the variation of $m$ and by $\lVert m\rVert_{\textup{TV}}$ its total variation, i.e.\@ $\lVert m\rVert_{\textup{TV}} = \lvert m\rvert(\R)$. 

Let us now collect some estimates which will be useful for the proof of Theorem \ref{thm:uniqueness_spde}.

\begin{lemma} \label{lem:estimates}
Let Assumption \ref{ass:spde} be satisfied. Let $m$ be a finite signed measure on $\R$ with $M_1(\lvert m\rvert) < \infty$ and define $F$ by $F(x) = m((-\infty, x])$ for $x \in \R$. Set $\Delta^{h, \epsilon}_t(x) = T_{\epsilon}(h_t m)(x) - h_t(x) T_{\epsilon} m(x)$ for $h \in \{b, \sigma\}$ and $\tilde{\Delta}^{a, \epsilon}_t(x) = \partial_x T_{\epsilon}(a_t m)(x) - a_t(x) \partial_x T_{\epsilon} m(x)$. Then, for all $\eta > 0$,
\begin{enumerate}[noitemsep, label = (\roman*)]
    \item \label{it:lambda} it holds that
    \begin{align*}
        \int_{\R} T_{\epsilon}F(x) \lambda_t(x) \bigl\langle m, G_{\epsilon}(x, \cdot) - 1\bigr)\bigr\rangle w_{\eta}(x) \, \d x &= \bigl\lVert \sqrt{\lambda_t }T_{\epsilon}F\bigr\rVert_{\eta}^2, \\
        \biggl\lvert \int_{\R} T_{\epsilon} F(x) \Bigl\langle m, \bigl(\lambda_t - \lambda_t(x)\bigr)\bigl(G_{\epsilon}(x, \cdot) - 1\bigr)\Bigr\rangle w_{\eta}(x) \, \d x \biggr\rvert &\leq C_{\lambda} \epsilon \lVert m \rVert_{\textup{TV}}^2 \biggl(\int_{\R} w_{\eta}(x) \, \d x\biggr);
    \end{align*}
    \item \label{it:b} it holds that
    \begin{align*}
        \biggl\lvert \int_{\R} T_{\epsilon} F(x) T_{\epsilon}m(x) b_t(x) w_{\eta}(x) \, \d x \biggr\rvert &\leq C_b \lVert m \rVert_{\textup{TV}} \Bigl((1 + \zeta)\lVert m \rVert_{\textup{TV}} + M_1(\lvert m\rvert)\Bigr), \\
        \biggl\lvert \int_{\R} T_{\epsilon} F(x) \Delta^{b, \epsilon}_t(x) w_{\eta}(x) \, \d x \biggr\rvert &\leq C_b (2 + \epsilon) \lVert T_{\epsilon}F\lVert_{\eta} \lVert T_{\epsilon}\lvert m\rvert\lVert_{\eta};
    \end{align*}
    \item \label{it:a} it holds that
    \begin{align*}
        \int_{\R} T_{\epsilon} F(x) \partial_x T_{\epsilon}m(x) a_t(x) w_{\eta}(x) \, \d x &= - \int_{\R} T_{\epsilon} F(x) T_{\epsilon}m(x) \partial_x\bigl(a_t(x) w_{\eta}(x)\bigr) \, \d x \\
        &\ \ \ - \bigl\lVert \sqrt{a_t} T_{\epsilon}m\bigr\rVert_{\eta}^2, \\
        \biggl\lvert\int_{\R} T_{\epsilon} F(x) T_{\epsilon}m(x) \partial_x\bigl(a_t(x) w_{\eta}(x)\bigr) \, \d x \biggr\rvert &\leq C_a (1 + \eta) \lVert T_{\epsilon} F\rVert_{\eta} \lVert T_{\epsilon} \lvert m\rvert \rVert_{\eta}, \\
        \biggl\lvert \int_{\R} T_{\epsilon} F(x) \tilde{\Delta}^{a, \epsilon}_t(x) w_{\eta}(x) \, \d x\biggr\rvert &\leq C_a \epsilon \lVert T_{\epsilon} F\rVert_{\eta} \lVert T_{\epsilon} \lvert m\rvert \rVert_{\eta} ;
    \end{align*}
    \item \label{it:sigma} for all $\delta > 0$ it holds that
    \begin{align*}
        \lVert T_{\epsilon} (\sigma_t m)\rVert_{\eta}^2 &= \lVert \sigma_t T_{\epsilon} m\rVert_{\eta}^2 + 2\bigl\langle \sigma_t T_{\epsilon} m, \Delta^{\sigma, \epsilon}_t\bigr\rangle_{\eta} + \bigl\lVert \Delta^{\sigma, \epsilon}_t\bigr\rVert_{\eta}^2, \\
        \biggl\lvert 2\bigl\langle \sigma_t T_{\epsilon} m, \Delta^{\sigma, \epsilon}_t\bigr\rangle_{\eta} + \bigl\lVert \Delta^{\sigma, \epsilon}_t\bigr\rVert_{\eta}^2\biggr\rvert &\leq C_{\sigma}^2 \delta \lVert T_{\epsilon} m\rVert_{\eta}^2 + \frac{C_a \epsilon (1 + \delta)}{\delta} \lVert T_{\epsilon} F\rVert_{\eta} \lVert T_{\epsilon} \lvert m\rvert \rVert_{\eta}.
    \end{align*}
\end{enumerate}
\end{lemma}

We can now proceed to the proof of the uniqueness result.

\begin{proof}[Proof of Theorem \ref{thm:uniqueness_spde}]
Let $\nu^i = (\nu^i_t)_{t \in [0, T]}$, $i = 1$, $2$ be two subprobability measure-valued solutions of SPDE \eqref{eq:spde} with the same initial condition. Fix $\nu \in \{\nu^1, \nu^2, \nu^1 - \nu^2\}$ and set $F^{\epsilon}_t = T_{\epsilon} F_t$ and $\nu^{\epsilon}_t = T_{\epsilon} \nu_t$. We will derive an equation for $F^{\epsilon}_t$ from SPDE \eqref{eq:spde} (since SPDE \eqref{eq:spde} is linear, $\nu^1 - \nu^2$ also solves this equation). To that end, we define the function $G_{\epsilon} \define \R^2 \to \R$ by
\begin{equation*}
    G_{\epsilon}(x, y) = \int_{(-\infty, y]} g_{\epsilon}(x - z) \, \d z
\end{equation*}
for $(x, y) \in \R^2$. Then integration by parts yields
\begin{equation*}
    F^{\epsilon}_t(x) = \int_{\R} F_t(y) g_{\epsilon}(x - y) \, \d y = \nu_t(\R) - \int_{\R} G_{\epsilon}(x, y) \, \d \nu_t(y).
\end{equation*}
Since $y \mapsto G_{\epsilon}(x, y)$ is an element of $C_b^2(\R)$ for every $x \in \R$, it follows from SPDE \eqref{eq:spde} that
\begin{align*}
    \d F^{\epsilon}_t(x) &= \d \nu_t(\R) - \bigl\langle \nu_t, \L G_{\epsilon}(x, \cdot)(t, \cdot) \bigr\rangle \, \d t - \bigl\langle \nu_t, \sigma_t \partial_y G_{\epsilon}(x, \cdot)\bigr\rangle \, \d B_t \\
    &= - \Bigl\langle \nu_t, -\lambda_t \bigl(G_{\epsilon}(x, \cdot) - 1\bigr) + b_t g_{\epsilon}(x - \cdot) + a_t \partial_y g_{\epsilon}(x - \cdot) \Bigr\rangle \, \d t \\
    &\ \ \ - \bigl\langle \nu_t, \sigma_t g_{\epsilon}(x - \cdot)\bigr\rangle \, \d B_t.
\end{align*}
Here we used in the second equality that $\nu_t(\R) = \nu_0(\R) - \int_0^t \langle \nu_s, \lambda_s\rangle \, \d s$, which follows from testing SPDE \eqref{eq:spde} with the function that is constantly equal to one. Hence, by It\^o's formula, we have
\begin{align*}
    \d \lvert F^{\epsilon}_t(x)\rvert^2 &= 2 F^{\epsilon}_t(x) \Bigl(\bigl\langle \nu_t, -\lambda_t \bigl(G_{\epsilon}(x, \cdot) - 1\bigr)\bigr\rangle - T_{\epsilon}(b_t\nu_t)(x) + \partial_x T_{\epsilon}(a_t\nu_t)(x)\Bigr) \, \d t \\
    &\ \ \ - \Bigl(2 F^{\epsilon}_t(x) T_{\epsilon}(\sigma_t\nu_t)(x)\Bigr) \, \d B_t + \frac{1}{2} \lvert T_{\epsilon}(\sigma_t\nu_t)(x)\rvert^2 \, \d t.
\end{align*}
We multiply both sides of this equation by the weight $w_{\eta}(x)$, integrate over $x \in \R$, and, finally, take expectation to make the stochastic integral term disappear. This yields
\begin{align*}
    \ev \lVert F^{\epsilon}_t\rVert_{\eta}^2 &= \lVert F^{\epsilon}_0\rVert_{\eta}^2 - 2 \ev \int_0^t \biggl(\int_{\R} F^{\epsilon}_s(x) \bigl\langle \nu_s, \lambda_s\bigl(G_{\epsilon}(x, \cdot) - 1\bigr)\bigr\rangle w_{\eta}(x) \, \d x\biggr) \, \d s \\
    &\ \ \ - 2 \ev \int_0^t \biggl(\int_{\R} F^{\epsilon}_s(x) T_{\epsilon}(b_s \nu_s)(x) w_{\eta}(x) \, \d x\biggr) \, \d s \\
    &\ \ \ + 2 \ev \int_0^t \biggl(\int_{\R} F^{\epsilon}_s(x) \partial_x T_{\epsilon}(a_s \nu_s)(x) w_{\eta}(x) \, \d x\biggr)\, \d s \\
    &\ \ \ + \ev\int_0^t \lVert T_{\epsilon}(\sigma_t \nu_t)\rVert_{\eta}^2.
\end{align*}
This equality can be rewritten as
\begin{align} \label{eq:fundamental_equality}
\begin{split}
    \ev \lVert F^{\epsilon}_t\rVert_{\eta}^2 &= \lVert F^{\epsilon}_0\rVert_{\eta}^2 - \ev \int_0^t \biggl(\bigl\lVert \sqrt{\lambda_s} F^{\epsilon}_s \bigr\rVert_{\eta}^2 + \int_{\R} b_s(x) F^{\epsilon}_s(x) \nu^{\epsilon}_s(x) w_{\eta}(x) \, \d x\biggr) \, \d s \\
    &\ \ \ - \ev \int_0^t \Bigl(2 \bigl\lVert \sqrt{a_s} \nu^{\epsilon}_s \bigr\rVert_{\eta}^2 - \bigl\lVert \sigma_s \nu^{\epsilon}_s \bigr\rVert_{\eta}^2\Bigr) \d s + \ev \int_0^t \cal{E}^{\epsilon}_s \, \d s
\end{split}
\end{align}
for an error term $\cal{E}^{\epsilon}_s$. For this error, Lemma \ref{lem:estimates} provides the estimate
\begin{align} \label{eq:error_est}
    \lvert \cal{E}^{\epsilon}_s\rvert &\leq C \epsilon \lVert \nu_s\rVert_{\textup{TV}}^2 + C \delta \lVert \nu^{\epsilon}_s\rVert_{\eta}^2 + \frac{C (\delta + \epsilon(1 + \delta))}{\delta} \lVert F^{\epsilon}_s\rVert_{\eta} \lVert T_{\epsilon} \lvert \nu_s\rvert \rVert_{\eta} \notag \\
    &\leq C \epsilon \lVert \nu_s\rVert_{\textup{TV}}^2 + C \delta \lVert \nu^{\epsilon}_s\rVert_{\eta}^2 + 2C \delta \lVert T_{\epsilon} \lvert \nu_s\rvert \rVert_{\eta}^2 + \frac{3C}{\delta^2} \lVert F^{\epsilon}_s\rVert_{\eta}^2
\end{align}
for all $\delta$, $\epsilon \in (0, 1)$ and a constant $C > 0$ independent of $\epsilon$ and $\delta$. Lemma \ref{lem:estimates} \eqref{it:b} further implies that
\begin{align} \label{eq:drift_est}
    \biggl\lvert \int_{\R} b_s(x) F^{\epsilon}_s(x) \nu^{\epsilon}_s(x) w_{\eta}(x) \, \d x\biggr\rvert &\leq C \lVert \nu_s\rVert_{\textup{TV}} \Bigl((1 + \zeta)\lVert \nu_s \rVert_{\textup{TV}} + M_1(\lvert \nu_s\rvert)\Bigr) \\
    &\leq 4 C \bigl(1 + \zeta + M_1(\lvert \nu_s\rvert)\bigr)
\end{align}

Now, in the case that $\nu$ is either of the subprobability measures $\nu^1$ or $\nu^2$, we have $T_{\epsilon} \lvert \nu_s\rvert = T_{\epsilon} \nu_s = \nu^{\epsilon}_s$. Then inserting the previous two inequalities into \eqref{eq:fundamental_equality} and using that $2 a_s(x) - \sigma_s^2(x) \geq 2c$ by Assumption \ref{ass:spde} \eqref{it:nondegenerate}, we find that
\begin{align*}
    \ev \lVert F^{\epsilon}_t\rVert_{\eta}^2 + (2c - 3C\delta) \ev \int_0^t \lVert \nu^{\epsilon}_s\rVert_{\eta}^2 \, \d s &\leq \lVert F^{\epsilon}_0\rVert_{\eta}^2 + \frac{3C}{\delta^2} \ev \int_0^t \lVert F^{\epsilon}_s\rVert_{\eta}^2 \, \d s \\
    &\ \ \ + 8 C \ev \int_0^t \bigl(1 + \zeta + M_1(\lvert \nu_s\rvert)\bigr) \, \d s.
\end{align*}
Now, choosing $\delta \in (0, 1)$ sufficiently small, a simple application of Gr\"onwall's inequality implies that
\begin{equation*}
     \sup_{0 \leq t \leq T} \ev \lVert F^{\epsilon}_t\rVert_{\eta}^2 + \ev \int_0^T \lVert \nu^{\epsilon}_t\rVert_{\eta}^2 \, \d s \leq C\biggl(\lVert F^{\epsilon}_0\rVert_{\eta}^2 + \ev \int_0^T \bigl(1 + \zeta + M_1(\lvert \nu_s\rvert)\bigr) \, \d t\biggr),
\end{equation*}
where we enlarge the constant $C$ if necessary. Since the right-hand side is bounded uniformly in $\epsilon \in (0, 1)$, letting $\epsilon \to 0$ implies that $\pr$-a.s.\@ the measure $\nu_t$ admits a density in $L_{\eta}^2(\R)$ (which for simplicity we denote by the same symbol $\nu_t$). 

Next, let us consider the situation that $\nu = \nu^1 - \nu^2$. The previous observation tells us that $\nu_t$ has a density (denoted by the same letter) in $L_{\eta}^2(\R)$ and that
\begin{equation*}
    \sup_{0 \leq t \leq T}\ev\lVert F_t \rVert_{\eta}^2 + \ev\int_0^T \lVert \nu_t \rVert_{\eta}^2 \, \d t < \infty.
\end{equation*}
Now let us revisit Equation \eqref{eq:fundamental_equality}. Owing to the bound on $F$ and $\nu$, we can pass to the limit as $\epsilon \to 0$ in this expression giving
\begin{align} \label{eq:fundamental_equality_2}
\begin{split}
    \ev \lVert F_t\rVert_{\eta}^2 &= - \ev \int_0^t \biggl(\bigl\lVert \sqrt{\lambda_s} F_s \bigr\rVert_{\eta}^2 + \int_{\R} b_s(x) F_s(x) \nu_s(x) w_{\eta}(x) \, \d x\biggr) \, \d s \\
    &\ \ \ - \ev \int_0^t \Bigl(2 \bigl\lVert \sqrt{a_s} \nu_s \bigr\rVert_{\eta}^2 - \bigl\lVert \sigma_s \nu_s \bigr\rVert_{\eta}^2\Bigr) \d s + \ev \int_0^t \cal{E}_s \, \d s,
\end{split}
\end{align}
where the error $\cal{E}_s$ satisfies the estimate
\begin{equation} \label{eq:error_est_2}
    \lvert \cal{E}_s\rvert \leq 3C\delta \lVert \nu_s\rVert_{\eta}^2 + \frac{3C}{\delta^2} \lVert F_s\rVert_{\eta}^2
\end{equation}
for all $\delta \in (0, 1)$. Note that unlike in \eqref{eq:error_est}, here the constant $C$ can be chosen independently of $\eta \in (0, 1)$ since the only estimates from Lemma \ref{lem:estimates} that we use in deriving \eqref{eq:error_est_2} and whose constants depend on $\eta$ are the second line of Lemma \ref{lem:estimates} \eqref{it:lambda} and the second line of Lemma \ref{lem:estimates} \eqref{it:a}. But the right-hand side of the second line of Lemma \ref{lem:estimates} \eqref{it:lambda} vanishes as $\epsilon \to 0$ and the factor $(1 + \eta)$ appearing in the second line of Lemma \ref{lem:estimates} \eqref{it:a} can simply be bounded by $2$ if $\eta \in (0, 1)$. Hence, proceeding as below \eqref{eq:drift_est}, we can show that
\begin{equation*}
    \sup_{0 \leq t \leq T} \ev \lVert F_t\rVert_{\eta}^2 + \ev \int_0^T \lVert \nu_t\rVert_{\eta}^2 \, \d s \leq C \ev \int_0^T \bigl(1 + \zeta + M_1(\lvert \nu_s\rvert)\bigr) \, \d t
\end{equation*}
for a possibly enlarged constant $C$ which does not depend on $\eta \in (0, 1)$. We can now take $\eta \to 0$, which in view of the monotone convergence theorem implies that 
\begin{equation} \label{eq:nonweighted_est}
    \sup_{0 \leq t \leq T} \ev \lVert F_t\rVert_{L^2}^2 + \ev \int_0^T \lVert \nu_t\rVert_{L^2}^2 \, \d s < \infty.
\end{equation}

Finally, we return to \eqref{eq:fundamental_equality_2}. Using \eqref{eq:nonweighted_est}, we can pass to the limit $\eta \to 0$ in \eqref{eq:fundamental_equality_2}. This implies
\begin{align} \label{eq:fundamental_equality_3}
\begin{split}
    \ev \lVert F_t\rVert_{L^2}^2 &= - \ev \int_0^t \biggl(\bigl\lVert \sqrt{\lambda_s} F_s \bigr\rVert_{L^2}^2 + \int_{\R} b_s(x) F_s(x) \nu_s(x) \, \d x\biggr) \, \d s \\
    &\ \ \ - \ev \int_0^t \Bigl(2 \bigl\lVert \sqrt{a_s} \nu_s \bigr\rVert_{L^2}^2 - \bigl\lVert \sigma_s \nu_s \bigr\rVert_{L^2}^2\Bigr) \d s + \ev \int_0^t \cal{E}_s \, \d s.
\end{split}
\end{align}
Similarly to \eqref{eq:error_est_2}, we have the bound $\lvert \cal{E}_s\rvert \leq 3C \delta \lVert \nu_s\rvert_{L^2}^2 + \frac{3C}{\delta^2} \lVert F_s\rVert_{L^2}^2$ on the error $\cal{E}_s$ in the above. This holds for all $\delta \in (0, 1)$ and a constant $C > 0$ independent of $\delta$. Next, we estimate
\begin{align} \label{eq:drift_est_alternative}
    \biggl\lvert \int_{\R} b_s(x) F_s(x) \nu_s(x) \, \d x\biggr\rvert &\leq \biggl\lvert \int_{\R} \partial_x \lvert F_s(x)\rvert^2 b^0_s(x) \, \d x\biggr\rvert + C_b \lVert F_s\rVert_{L^2} \lVert \nu_s\rVert_{L^2} \notag \\
    &\leq C_b \lVert F_s\rVert_{L^2} \bigl(\lVert F_s\rVert_{L^2} + \lVert \nu_s\rVert_{L^2}\bigr).
\end{align}
Here we used in the second step that with $b^{0, n}_s(x) = -n \lor b^0_s(x) \land n$, we have
\begin{align*}
    \int_{\R} \partial_x \lvert F_s(x)\rvert^2 b^0_s(x) \, \d x &= \lim_{n \to \infty} \int_{\R} \partial_x \lvert F_s(x)\rvert^2 b^{0, n}_s(x) \, \d x \\
    &= \lim_{n \to \infty} \biggl(\bigl(\lvert F_s(x)\rvert^2 b^{0, n}_s(x)\bigr)\big\vert_{-\infty}^{\infty} - \int_{\R} \lvert F_s(x)\rvert^2 \partial_x  b^{0, n}_s(x) \, \d x\biggr) \\
    &= -\int_{\R} \lvert F_s(x)\rvert^2 \partial_x b^0_s(x) \, \d x,
\end{align*}
where the third equality follows from the fact that $\lim_{\lvert x\rvert \to \infty} f(x) = 0$ for $f \in H^1(\R)$. Inserting \eqref{eq:drift_est_alternative} and the estimate on the error term into \eqref{eq:fundamental_equality_3} and then arguing similarly as below \eqref{eq:drift_est}, we can conclude that
\begin{equation*}
    \sup_{0 \leq t \leq T} \ev \lVert F_t\rVert_{L^2}^2 + \ev \int_0^T \lVert \nu_t\rVert_{L^2}^2 \, \d s \leq C \lVert F_0\rVert_{L^2}^2 = 0.
\end{equation*}
This completes the proof.
\end{proof}

\subsection{Technical Results for Section \ref{sec:ps}} \label{sec:appendix}

\begin{lemma} \label{lem:metr_space_emb}
Let $(E, d_E)$ be a complete separable metric space and let $\Phi$ be an injective map from some set $F$ to $E$ such that $\Phi(F)$ is closed. Define $d_F \define E \times E \to [0, \infty)$ by $d_F(x, y) = d_E(\Phi(x), \Phi(y))$ for $(x, y) \in E \times E$. Then $(F, d_F)$ is a complete separable metric space and the map $\Phi$ is an isometry onto $(E, d_E)$.
\end{lemma}

The proof of Lemma \ref{lem:metr_space_emb} is straightforward, so we do not provide it here. Let us recall that $\bf{M}^p = \cal{M}^p_{\leq 1}(\R)$ for $p \geq 1$.

\begin{proposition}[Subprobabilities form a Polish space] \label{prop:space_of_subprobs}
The space $\bf{M}^p$ endowed with the distance function $d_p$ defined in Equation \eqref{eq:subprob_metric} forms a complete separable metric space. Moreover, a family $\cal{K} \subset \bf{M}^p$ is precompact if and only if $\{v + (1 - v(\R))\delta_0 \define v \in \cal{K}\} \subset \P^p(\R)$ is precompact.
\end{proposition}

\begin{proof}
For the first statement, owing to Lemma \ref{lem:metr_space_emb} and the definition of $d_p$ in \eqref{eq:subprob_metric}, it is enough to show that the image of $\bf{M}^p$ under the map $\Phi \define \bf{M}^p \to \P^p(\R) \times [0, 1]$, $v \mapsto (v + (1 - v(\R))\delta_0, v(\R))$ is closed in $\P^p(\R) \times [0, 1]$. Fix a sequence $(v_n)_{n \geq 1}$ in $\bf{M}^p$ such that $(m_n, r_n) = \Phi(v_n)$ converges to some $(m, r) \in \P^p(\R) \times [0, 1]$. Then by the Portmanteau theorem, it holds that
\begin{equation*}
    r = \lim_{n \to \infty} r_n \leq \limsup_{n \to \infty} m_n(\{0\}) \leq m(\{0\}).
\end{equation*}
Hence, setting $v = m - r \delta_0$, we find that $\lim_{n \to \infty} (m_n, r_n) = (m, r) = (v + r\delta_0, r) = \Phi(v) \in \Phi(\bf{M}^p)$. Consequently, $\Phi(\bf{M}^p)$ is closed.

Next, let us suppose that $\cal{K} \subset \bf{M}^p$ is precompact. Then, since $\Phi$ is continuous by Lemma \ref{lem:metr_space_emb}, it follows that $\Phi(\bf{M}^p)$ is a precompact subset of $\P^p(\R) \times [0, 1]$. But $\{v + (1 - v(\R))\delta_0 \define v \in \cal{K}\}$ is simply the image of $\Phi(\bf{M}^p)$ under the projection of $\P^p(\R) \times [0, 1]$ onto its first component and is therefore also precompact. Conversely, assume that the set $\{v + (1 - v(\R))\delta_0 \define v \in \cal{K}\}$ is precompact and fix a sequence $(v_n)_{n \geq 1}$ in $\cal{K}$. By precompactness of $\{v + (1 - v(\R))\delta_0 \define v \in \cal{K}\}$, there exists a convergent subsequence $(v_{n_k} + (1 - v_{n_k}(\R))\delta_0)_{k \geq 1}$. Since $[0, 1]$ is compact, by selecting a further subsequence if necessary, we may assume that $(\Phi(v_{n_k})_{k \geq 1}$ converges to some $m = \Phi(v) \in \Phi(\bf{M}^p)$. Then, from the definition of $\Phi$, we get
\begin{equation*}
    d_p(v_{n_k}, v) = \Bigl(W_p\bigl(v_{n_k} + (1 - v_{n_k}(\R))\delta_0, v + (1 - v(\R))\delta_0\bigr) + \lvert v_{n_k}(\R) - v(\R)\rvert\Bigr) \to 0
\end{equation*}
as $k \to \infty$. Hence, $(v_n)_{n \geq 1}$ contains a convergent subsequence. Since $(v_n)_{n \geq 1} \subset \cal{K}$ was arbitrary, it follows that $\cal{K}$ is precompact, as required. 
\end{proof}

\begin{proposition}[Kantorovich--Rubinstein duality] \label{prop:krd}
It holds for all $v_1$, $v_2 \in \bf{M}^1$ that
\begin{equation*}
    d_1(v_1, v_2) = \sup_{\lvert \varphi(0)\rvert \lor \lVert \varphi\rVert_{\textup{Lip}} \leq 1}\langle v_1 - v_2, \varphi\rangle,
\end{equation*}
where the infimum runs over all Lipschitz functions $\varphi \define \R \to \R$ and $\lVert \varphi\rVert_{\textup{Lip}}$ denotes the Lipschitz constant of $\varphi$.
\end{proposition}

\begin{proof}
We set $m_i = v_i + (1 - v_i(\R))\delta_0$, $i = 1$, $2$, and compute
\begin{align*}
    \sup_{\lvert \varphi(0)\rvert \lor \lVert \varphi\rVert_{\textup{Lip}} \leq 1}\langle v_1 - v_2, \varphi\rangle &= \sup_{\lvert \varphi(0)\rvert \lor \lVert \varphi\rVert_{\textup{Lip}} \leq 1} \Bigl(\langle m_1 - m_2, \varphi\rangle + \varphi(0)\bigl(v_2(\R) - v_1(\R)\bigr)\Bigr) \\
    &= \sup_{\lVert \varphi\rVert_{\textup{Lip}} \leq 1} \langle m_1 - m_2, \varphi\rangle + \sup_{\lvert \varphi(0)\rvert \leq 1} \varphi(0)\bigl(v_2(\R) - v_1(\R)\bigr) \\
    &= W_1(m_1, m_2) + \lvert v_1(\R) - v_2(\R)\rvert \\
    &= d_1(v_1, v_2),
\end{align*}
where we used in the second equality that the expression $\langle m_1 - m_2, \varphi\rangle$ is invariant under translations of $\varphi$ since $m_1$ and $m_2$ have unit mass and in the third equality applied the Kantorovich--Rubinstein duality for probability measures.
\end{proof}

\begin{lemma}[Weak convergence and uniform square-integrability] \label{lem:weak_conv_upgrade}
Let $(\nu^n)_n$ be a sequence of random variables with values in $D_{\bf{M}^2}[0, T]$, equipped with the $J1$-topology. Then $(\nu^n)_n$ converges weakly in $D_{\bf{M}^2}[0, T]$ if and only if $(\nu^n)_n$ converges weakly in $D_{\bf{M}^1}[0, T]$ and
\begin{equation} \label{eq:uniform_square}
    \lim_{R \to \infty} \sup_{n \geq 1} \ev \sup_{t \in [0, T]} \int_{\R} \bf{1}_{\lvert x \rvert \geq R} \lvert x \rvert^2 \, \d \nu^n_t(x) = 0.
\end{equation}
\end{lemma}

\begin{proof}
We only prove that weak convergence in $D_{\bf{M}^1}[0, T]$ together with the uniform square-integrability condition \eqref{eq:uniform_square} implies weak convergence in $D_{\bf{M}^2}[0, T]$. The converse direction is simpler. By appealing to the Skorokhod representation theorem, we may assume that the sequence $(\nu^n)_n$ converges a.s.\@ to some $D_{\bf{M}^1}[0, T]$-valued random variable $\nu$. Our first goal is to show that $\nu$ is a.s.\@ an element of $D_{\bf{M}^2}[0, T]$. The uniform-square integrability of $(\nu^n)_n$ implies that a.s.\@ $\nu_t$ lies $\bf{M}^2$ for all $t \in [0, T]$, so it remains to show that $\nu$ has c\`adl\`ag trajectories. Let us define $\varphi_k(x) = \lvert x \rvert^2 \land k$ and set $m_{k, R}(v) = \int_{\R} \bf{1}_{\lvert x \rvert \geq R} \varphi_k(x) \, \d v(x)$ for $k \geq 1$, $R \geq 1$, and $v \in \bf{M}^2$. Then we estimate
\begin{align*}
    \ev\sup_{t \in [0, T]} m_{k, R}(\nu_t) &= \ev\sup_{t \in [0, T]}\lim_{n \to \infty} m_{k, R}(\nu^n_t) \\
    &\leq \lim_{n \to \infty} \ev\sup_{t \in [0, T]} m_{k, R}(\nu^n_t) \\
    &\leq \sup_{n \geq 1} \ev \sup_{t \in [0, T]} \int_{\R} \bf{1}_{\lvert x \rvert \geq R} \lvert x \rvert^2 \, \d \nu^n_t(x).
\end{align*}
Now, we first take $k$ and then $R$ to $\infty$ to see that
\begin{equation*}
    \lim_{R \to \infty}\ev \sup_{t \in [0, T]} \int_{\R} \bf{1}_{\lvert x \rvert \geq R} \lvert x \rvert^2 \, \d \nu_t(x) = 0.
\end{equation*}
From this and the c\`adl\`ag trajectories of $\nu$ with respect to $d_1$, it easily follows that $\nu$ has c\`adl\`ag trajectories with respect to $d_2$.

Next, we prove that $(\nu^n)_n$ converges weakly to $\nu$ in $D_{\bf{M}^2}[0, T]$. Let $r \define [0, T] \to [0, T]$ be an arbitrary time change, i.e.\@ an increasing homeomorphism. We define $\mu^n_t = \nu^n_{r(t)} + (1 - \nu^n_{r(t)}(\R))\delta_0$ and, similarly, $\mu_t = \nu_t + (1 - \nu_t(\R))\delta_0$, so that by definition $d_2(\nu^n_{r(t)}, \nu_t) = W_2(\mu^n_t, \mu_t) + \bigl\lvert \nu^n_{r(t)}(\R) - \nu_t(\R) \bigr\rvert$. Our goal is to estimate $W_2(\mu^n_t, \mu_t)$ in terms of $W_1(\mu^n_t, \mu_t)$, using the uniform square-integrability from \eqref{eq:uniform_square}. We have for all couplings $\pi \in \P^2(\R^2)$ between $\mu^n_t$ and $\mu_t$ that
\begin{align*}
    W_2^2(\mu^n_t, \mu_t) &\leq \int_{\R^2} \lvert x - y\rvert^2 \, \d \pi(x, y) \\
    &\leq \int_{B_R \times B_R} \lvert x - y\rvert^2 \, \d \pi(x, y) + \int_{\R} \bf{1}_{\lvert x \rvert \geq R} \lvert x\rvert^2 \, \d (\mu^n_t + \mu_t)(x) \\
    &\leq 2R \int_{\R^2} \lvert x - y\rvert  \, \d \pi(x, y) + \int_{\R} \bf{1}_{\lvert x \rvert \geq R} \lvert x\rvert^2 \, \d \bigl(\nu^n_{r(t)} + \nu_t\bigr)(x).
\end{align*}
Minimising the right-hand side over the set of couplings $\pi \in \P^2(\R^2)$ between $\mu^n_t$ and $\mu_t$, and adding $\bigl\lvert \nu^n_{r(t)}(\R) - \nu_t(\R) \bigr\rvert^2$ on both sides implies that
\begin{align*}
    \frac{1}{2} d_2^2(\nu^n_{r(t)}, \nu_t) &\leq W_2^2(\mu^n_t, \mu_t) + \bigl\lvert \nu^n_{r(t)}(\R) - \nu_t(\R) \bigr\rvert^2 \\
    &\leq 2 R W_1(\mu^n_t, \mu_t) + \bigl\lvert \nu^n_{r(t)}(\R) - \nu_t(\R) \bigr\rvert^2 + \int_{\R} \bf{1}_{\lvert x \rvert \geq R} \lvert x\rvert^2 \, \d \bigl(\nu^n_t + \nu_t\bigr)(x) \\
    &\leq 2 R d_1(\nu^n_{r(t)}, \nu_t) + d_1^2(\nu^n_{r(t)}, \nu_t) + \int_{\R} \bf{1}_{\lvert x \rvert \geq R} \lvert x\rvert^2 \, \d \bigl(\nu^n_t + \nu_t\bigr)(x).
\end{align*}
Next, take the supremum over $t \in [0, T]$, add the term $\frac{1}{2}\lVert r - \id_{[0, T]}\rVert_{\infty}$ on both sides, minimise over all time changes $r \define [0, T] \to [0, T]$, and finally take expectations on both sides to obtain
\begin{align*}
    \frac{1}{4}\ev J_1^2(\nu^n, \nu) &\leq 2R \ev J_1(\nu^n, \nu) + \ev J_1^2(\nu^n, \nu) + \ev\sup_{t \in [0, T]}\int_{\R} \bf{1}_{\lvert x \rvert \geq R} \lvert x\rvert^2 \, \d \bigl(\nu^n_t + \nu_t\bigr)(x) \\
    &\leq 2R \ev J_1(\nu^n, \nu) + \ev J_1^2(\nu^n, \nu) + \sup_{m \geq 1}\ev\sup_{t \in [0, T]}\int_{\R} \bf{1}_{\lvert x \rvert \geq R} \lvert x\rvert^2 \, \d \bigl(\nu^m_t + \nu_t\bigr)(x), 
\end{align*}
where $J_1$ denotes the $J1$-distance on $D_{\bf{M}^p}[0, T]$. Now, we first choose $R$ large enough to make the third term on the right-hand side arbitrarily small and then take $n \to \infty$ to have the two remaining terms on the right-hand side vanish. This show that $\lim_{n \to \infty} \ev J1^2(\nu^n, \nu) = 0$, which in turn establishes the weak convergence of $\nu^n$ to $\nu$ on $D_{\bf{M}^2}[0, T]$.
\end{proof}

\begin{lemma}[Continuity of integral operator] \label{lem:mgale_integral_converges}
Let $(E, \rho)$ be a complete separable metric space and $\Phi \define [0, T] \times E \times G \to \R$ be a measurable function, such that $(x, g) \mapsto \Phi(t, x, g)$ is continuous for every $t \in [0, T]$. Suppose that $\lvert \Phi(t, x, g)\rvert \leq C_{\Phi}(1 + \rho^p(x, x_{\ast}) + \lvert g \rvert^p)$ for some $1 \leq p \leq 2$ and a fixed $x_{\ast} \in E$. Then the map
\begin{equation*}
    [0, T] \times L^p([0, T]; E) \times \bb{M}_T^2(G) \ni (t, x, \mathfrak{g}) \mapsto \int_{[0, t] \times G} \Phi(s, x_s, g) \, \d \mathfrak{g}(s, g)
\end{equation*}
is continuous.
\end{lemma}

Here $G$ is a non-empty closed subset of $\R^d$ and $L^p([0, T]; E)$ is the space of (equivalence classes of) measurable maps $x \define [0, T] \to E$ for which $\int_0^T \rho^p(x_t, x_{\ast}) \, \d t < \infty$. We equip $L^p([0, T]; E)$ with the metric
\begin{equation*}
    E \times E \ni (x, y) \mapsto \biggl(\int_0^T \rho^p(x_t, y_t) \, \d t\biggr)^{1/p}.
\end{equation*}
Note that if $(x^n)_n$ converges in $L^p([0, T]; E)$, then the family $((\rho^p(x_t, x_{\ast}))_t)_n$ is uniformly integrable. The space $\bb{M}_T^2$ consists of all square-integrable measures on $[0, T] \times G$, which have the Lebesgue measure as their first marginal, topologised by the $2$-Wasserstein distance.

\begin{proof}[Proof of Lemma \ref{lem:mgale_integral_converges}]
Fix $(t, x, \mathfrak{g}) \in [0, T] \times L^p([0, T]; E) \times \bb{M}_T^2(G)$ and assume that the sequence $(t_n, x^n, \mathfrak{g}^n)_n$ converges to $(t, x, \mathfrak{g})$. For $R > 0$ let $G_R$ denote the elements $g$ of $G$ with $\lvert g \rvert \leq R$. Then we write
\begin{align*}
    &\int_{[0, t_n] \times G} \Phi(s, x^n_s, g) \, \d \mathfrak{g}^n(s, g) - \int_{[0, t] \times G} \Phi(s, x_s, g) \, \d \mathfrak{g}(s, g) \\
    &\leq \biggl\lvert \int_{([0, t] \Delta [0, t_n]) \times G} \Phi(s, x^n_s, g) \, \d \mathfrak{g}^n(s, g)\biggr\rvert + \biggl\lvert\int_{[0, t] \times G_R^c} \Phi(s, x^n_s, g) - \Phi(s, x_s, g) \, \d \mathfrak{g}^n(s, g)\biggr\rvert \\
    &\ \ \ + \biggl\lvert\int_{[0, t] \times G_R} \Phi(s, x^n_s, g) - \Phi(s, x_s, g) \, \d \mathfrak{g}^n(s, g)\biggr\rvert + \biggl\lvert\int_{[0, t] \times G} \Phi(s, x_s, g) \, \d(\mathfrak{g}^n - \mathfrak{g})(s, g)\biggr\rvert \\
    &= I_1 + I_2 + I_3 + I_4,
\end{align*}
where $A \Delta B = A \setminus B \cup B \setminus A$ for two sets $A$ and $B$. First, let us disintegrate the measure $\mathfrak{g}^n$ as $\d \mathfrak{g}(s, g) = \d \mathfrak{g}^n_s(g) \d s$ for a family of measures $(\mathfrak{g}^n_s)_{0 \leq s\leq T}$ on $G$. Then let $\epsilon > 0$. We will prove that we can choose $R > 0$ and $N \geq 1$, such that $I_1 + I_2 + I_3 + I_4 \leq \epsilon$ for all $n \geq N$. 

By assumption $\lvert \Phi(s, x^n_s, g)\rvert \leq C_{\Phi}(1 + \rho^p(x^n_s, x_{\ast}) + \lvert g \rvert^p)$. Since $((\rho^p(x_s, x_{\ast}))_s)_n$ is uniformly integrable by the discussion below the statement of the lemma and $(\mathfrak{g}^n)_n$ is uniformly $p$-integrable as a $W_2$-convergent sequence, we can choose $N \geq 1$, such that for all $n \geq N$,
\begin{align*}
    I_1 &\leq \int_{([0, t] \Delta [0, t_n]) \times G} C_{\Phi}(1 + \rho^p(x^n_s, x_{\ast}) + \lvert g \rvert^p) \, \d \mathfrak{g}^n(s, g) \\
    &= \int_{[0, t] \Delta [0, t_n]} C_{\Phi}(1 + \rho^p(x^n_s, x_{\ast})) \, \d s + \int_{([0, t] \Delta [0, t_n]) \times G} C_{\Phi} \lvert g \rvert^p \, \d \mathfrak{g}^n(s, g)
    \leq \frac{\epsilon}{4}.
\end{align*}
For the second expression $I_2$ we note that similarly to above
\begin{equation} \label{eq:uniform_bound_phi}
    \lvert \Phi(s, x^n_s, g) - \Phi(s, x_s, g) \rvert \leq C_{\Phi} \bigl(2 + \rho^p(x^n_s, x_{\ast}) + \rho^p(x_s, x_{\ast}) + 2\lvert g \rvert^p\bigr).
\end{equation}
Consequently, we have
\begin{equation*}
    I_2 \leq \int_{[0, t]} C_{\Phi}\bigl(1 + \rho^p(x^n_s, x_{\ast}) + \rho^p(x_s, x_{\ast})\bigr) \mathfrak{g}^n_s(G_R^c) \, \d s + \int_{[0, t] \times G_R^c} \lvert g \rvert^p \, \d \mathfrak{g}^n(s, g).
\end{equation*}
The second term on the right-hand side can be made arbitrarily small, say smaller than $\frac{\epsilon}{12}$, by choosing $R > 0$ large enough since $(\mathfrak{g}^n)_n$ is uniformly $p$-integrable. Estimating the first one is slightly more involved. Let us define $y^n_t = C_{\Phi}\bigl(1 + \rho^p(x^n_s, x_{\ast}) + \rho^p(x_s, x_{\ast})\bigr)$, so that $((y^n_t)_t)_n$ forms a uniformly integrable family. Further, for $\delta > 0$ we set $A_{n, R, \delta} = \{s \in [0, t] \define \mathfrak{g}^n_s(G^c_R) > \delta\}$. Then, replacing $R$ by $R \lor 1$ if necessary,
\begin{equation} \label{eq:unif_vanish_meas}
    \sup_{n \geq 1} \text{Leb}(A_{n, R, \delta}) \leq \sup_{n \geq 1}  \frac{1}{\delta} \mathfrak{g}^n([0, t] \times G_R^c).
\end{equation}
The right-hand side vanishes as $R$ tends to infinity by the uniform $p$-integrability of $(\mathfrak{g}^n)_n$. Now we write
\begin{align*}
    \int_{[0, t]} y^n_s \mathfrak{g}^n_s(G_R^c) \, \d s &= \int_{A_{n, R, \delta}} y^n_s \mathfrak{g}^n_s(G_R^c) \, \d s + \int_{A_{n, R, \delta}^c} y^n_s \mathfrak{g}^n_s(G_R^c) \, \d s \\
    &\leq \int_{A_{n, R, \delta}} y^n_s \, \d s + \delta \int_{[0, t]} y^n_s \, \d s.
\end{align*}
Now we first pick $\delta > 0$ small enough so that the second expression in the second line becomes smaller than $\frac{\epsilon}{12}$. Then, we enlarge $R$ sufficiently so that $\int_{A_{n, R, \delta}} y^n_s \, \d s \leq \frac{\epsilon}{12}$. This is possible since the family $((y^n_t)_t)_n$ is uniformly integrable and $\sup_{n \geq 1} \text{Leb}(A_{n, R, \delta}) \to 0$ as $R \to \infty$ by \eqref{eq:unif_vanish_meas}. Putting all these estimates together yields $I_2 \leq \frac{\epsilon}{4}$.

For the integrand in $I_3$ we have
\begin{equation*}
    \biggl\lvert \int_{G_R} \Phi(s, x^n_s, g) - \Phi(s, x_s, g) \, \d \mathfrak{g}^n_s(g) \biggr\rvert \leq \sup_{g \in G_R} \lvert \Phi(s, x^n_s, g) - \Phi(s, x_s, g) \rvert.
\end{equation*}
The expression on the right-hand side converges to zero in measure as $n \to \infty$ since $\Phi$ is continuous in its last two arguments, $G_R$ is compact, and $x^n \to x$ in $L^p([0, T]; E)$. Moreover, appealing to the estimate in Equation \eqref{eq:uniform_bound_phi} shows that $\bigl(\bigl(\sup_{g \in G_R} \lvert \Phi(s, x^n_s, g) - \Phi(s, x_s, g) \rvert\bigr)_s\bigr)_n$ is uniformly integrable, so that $I_3 \to 0$ as $n \to \infty$ by Vitali's convergence theorem. In particular, enlarging $N \geq 1$ if necessary, we have $I_3 \leq \frac{\epsilon}{4}$ for all $n \geq N$.

For the last expression $I_4$, we note that $[0, T] \times G \ni (s, g) \mapsto \varphi(s, g) = \Phi(s, x_s, g)$ is measurable in the first component and continuous in the second. Thus if we define $\varphi_k(s, g) = -k \lor \varphi(s, g) \land k$ for $k \geq 1$, it follows from \cite[Corollary 2.9]{jacod_mc_conv_1981} that 
\begin{equation} \label{eq:conv_for_bounded_mc}
    \int_{[0, t] \times G} \varphi_k(s, g) \, \d (\mathfrak{g}^n - \mathfrak{g})(s, g) \to 0
\end{equation}
as $n \to \infty$. On the other hand,
\begin{equation*}
    \biggl\lvert\int_{[0, t] \times G} \varphi_k(s, g) - \varphi(s, g) \, \d (\mathfrak{g}^n - \mathfrak{g})(s, g)\biggr\rvert \leq \int_{[0, t] \times G} \bf{1}_{\lvert \varphi(s, g)\rvert \geq k} \lvert \varphi(s, g)\rvert \, \d (\mathfrak{g}^n + \mathfrak{g})(s, g).
\end{equation*}
Using once more that $(\mathfrak{g}^n)_n$ is uniformly $p$-integrable and $\lvert \varphi(s, g) \rvert \leq C_{\Phi}(1 + \rho^p(x_s, x_{\ast}) + \lvert g \rvert^p)$, we conclude that the quantity on the right-hand side above vanishes uniformly in $n \geq 1$ as $k \to \infty$. In view of Equation \eqref{eq:conv_for_bounded_mc}, increasing $N \geq 1$ if necessary and choosing $k$ large enough, we obtain $I_4 \leq \frac{\epsilon}{4}$ for $n \geq N$. Hence, it holds that $I_1 + I_2 + I_3 + I_4 \leq \epsilon$ as required.
\end{proof}

\begin{corollary}[Continuity of integral operator] \label{cor:mgale_integral_converges}
Let $(E_i, \rho_i)$, $i = 1$,\ldots, $n$, be a finite collection of complete separable metric space and $\Phi \define [0, T] \times E_1 \times \dots \times E_n \times G \to \R$ be a measurable function, such that $(x_1, \dots, x_n, g) \mapsto \Phi(t, x_1, \dots, x_n, g)$ is continuous for every $t \in [0, T]$. Suppose that
\begin{equation*}
    \lvert \Phi(t, x_1, \dots, x_n, g)\rvert \leq C_{\Phi}\bigl(1 + \rho_1^p(x_1, x_{1, \ast}) + \dots + \rho_1^p(x_n, x_{n, \ast}) + \lvert g \rvert^p\bigr)
\end{equation*}
for some $1 \leq p \leq 2$ and fixed $x_{i, \ast} \in E_i$. Then the map $[0, T] \times D_{E_1}[0, T] \times \dots \times D_{E_n}[0, T] \times \bb{M}_T^2(G) \to \R$,
\begin{equation*}
    (t, x^1, \dots, x^n, \mathfrak{g}) \mapsto \int_{[0, t] \times G} \Phi(s, x^1_s, \dots, x^n_s, g) \, \d \mathfrak{g}(s, g)
\end{equation*}
is continuous.
\end{corollary}

\begin{proof}
Simply note that the inclusion $D_{E_1}[0, T] \times \dots \times D_{E_n}[0, T] \to L^2([0, T]; E_1 \times \dots \times E_n)$ is continuous and apply Lemma \ref{lem:mgale_integral_converges}.
\end{proof}

\subsection{Technical Results for Section \ref{sec:sin_limit}}

\begin{lemma}[Continuity of subprobability mapping] \label{lem:space_for_sub}
Let $T > 0$. Then the function $\Phi \define \P^2(D[0, T] \times D_{[0, 1]}[0, T]) \to L^2([0, T]; \cal{M}^2_{\leq 1}(\R))$ defined by
\begin{equation*}
    \langle \Phi(\mu)_t, \varphi \rangle = \int_{D[0, T] \times D_{[0, 1]}[0, T]} y_t \varphi(x_t) \, \d \mu(x, y)
\end{equation*}
for $\varphi \in C_b(\R)$, $t \in [0, T]$, and $\mu \in \P^2(D[0, T] \times D_{[0, 1]}[0, T])$ is continuous.
\end{lemma}

Here $\P^2(D[0, T] \times D_{[0, 1]}[0, T])$ denotes the space of probability distributions $\mu$ on $D[0, T] \times D_{[0, 1]}[0, T]$ for which
\begin{equation*}
    \int_{D[0, T] \times D_{[0, 1]}[0, T]} (\lvert x \rvert^{\ast}_T)^2 + (\lvert y \rvert^{\ast}_T)^2 \, \d \mu(x, y) < \infty,
\end{equation*}
where $\lvert \cdot \rvert^{\ast}_T$ denotes the running supremum over $[0, T]$ of the absolute value of a path $[0, T] \to \R$. We equip $\P^2(D[0, T] \times D_{[0, 1]}[0, T])$ with the $1$-Wasserstein distance, where $D[0, T] \times D_{[0, 1]}[0, T]$ is endowed with the metric $d_{M1}$ defined in Equation (3.7) in \cite[Chapter 12]{whitt_stoch_limits_2002}, which induces the topology of convergence in $M1$ on $D[0, T] \times D_{[0, 1]}[0, T]$. 
Note that if a sequence $(\mu^n)_n$ in $\P^2(D[0, T] \times D_{[0, 1]}[0, T])$ is convergent in $\P^2(D[0, T] \times D_{[0, 1]}[0, T])$, then it is uniformly square-integrable in the sense that
\begin{equation*}
    \sup_{n \geq 1}\int_{D[0, T] \times D_{[0, 1]}[0, T]} \bf{1}_{\lvert x \rvert^{\ast}_T \geq K} (\lvert x \rvert^{\ast}_T)^2\, \d \mu^n(x, y) \to 0
\end{equation*}
as $K \to \infty$.

\begin{proof}[Proof of Lemma \ref{lem:space_for_sub}]
Assume that $(\mu^n)_n$ converges to $\mu$ in $\P^2(D[0, T] \times D_{[0, 1]}[0, T])$. Then by the Skorokhod representation theorem, we find random variables $(X^n, I^n)$, $n \geq1$, and $(X, I)$ such that $(X^n, I^n) \sim \mu^n$ as well as $(X, I) \sim \mu$ and $X^n \to X$ as well as $I^n \to I$ a.s.\@ in $M1$. Now, by the definition of convergence in $M1$, both $X^n_t$ and $I^n_t$ converge a.s.\@ to $X_t$ and $I_t$ for $t$ in the cocountable set $\bb{T}$ of a.s.\@ continuity points of $X$ and $I$. From the comment below the statement, we also know that $(\lvert X^n\rvert^{\ast}_T)_n$ is a uniformly square-integrable family. Thus, Vitali's convergence theorem implies that $\ev[\lvert X^n_t - X_t\rvert^2] \to 0$ and $\ev[(1 + \lvert X_t\rvert^2) \lvert I^n_t - I_t\rvert] \to 0$ for all $t \in \bb{T}$. But, now we estimate as in \eqref{eq:split_of_metric} that
\begin{align*}
    \int_0^T d_2^2(\Phi(\mu^n)_t, \Phi(\mu)_t) \, \d t \leq \int_0^T 3\ev[\lvert X^n_t - X_t\rvert^2] + 3\ev[(1 + \lvert X_t\rvert^2) \lvert I^n_t - I_t\rvert] \, \d t,
\end{align*}
and the right-hand side vanishes as $n \to \infty$ by our previous remarks.
\end{proof}

\begin{lemma}[Integral against nondecreasing processes] \label{lem:conv_int_non_decr}
Let $\cal{I}_0[-1, T]$ denote the space of nondecreasing c\`adl\`ag functions $\ell \define [-1, T] \to \R$ with $\ell_t = 0$ for $t \in [-1, 0)$ and equip $\cal{I}_0[-1, T]$ with the $M1$-topology. For nonnegative $\varphi \in C([0, T])$ define $I^{\varphi} \define \cal{I}_0[-1, T] \to D[0, T]$ by $I^{\varphi}_t(\ell) = \int_0^t \varphi(s) \, \d \ell$. Then $I^{\varphi}$ is continuous at any element $\ell \in \cal{I}_0[-1, T]$, which is continuous at $t = T$.
\end{lemma}

\begin{proof}
Let $(\ell^n)_n$ be a sequence in $\cal{I}_0[-1, T]$ that converges to $\ell \in \cal{I}_0[-1, T]$ which is continuous at $t = T$. By Item (iv) of \cite[Theorem 12.5.1]{whitt_stoch_limits_2002} is is enough to show that $I^{\varphi}_t(\ell^n) \to I^{\varphi}_t(\ell)$ for each $t$ in a dense subset of $[-1, T]$ including $-1$ and $T$. First note that we may assume that $\ell^n_T = \ell_T = 1$, otherwise replace $\ell^n$ by $t \mapsto (\ell^n_t + t_+)/(\ell^n_T + T)$, where $t_+$ denotes the positive part of $t$, and modify $\ell$ in an analogous manner. Now, convergence at $t = -1$ is obvious since $I^{\varphi}_0(\ell^n) = 0 = I^{\varphi}_0(\ell)$. Next, let $\bb{T}$ be the set of continuity points of $\ell$, which is dense in $[-1, T]$ and contains $T$. We claim that $I^{\varphi}_t(\ell^n) \to I^{\varphi}_t(\ell)$ for all $t \in \bb{T}$. Since $\ell^n \to \ell$ in $D[-1, T]$, Item (iv) \cite[Theorem 12.5.1]{whitt_stoch_limits_2002} tells us that $\ell^n_t \to \ell_t$ for each $t \in \bb{T}$. But we can view $\ell^n$ and $\ell$ as cumulative distribution functions of random variables $Y^n$ and $Y$ with values in $[-1, T + 1]$. Then $\ell^n_t \to \ell_t$ for all $t \in \bb{T}$ is precisely equivalent to weak convergence of $Y^n$ to $Y$ by the Portmanteau theorem. Finally, we note that for any $t \in \bb{T}$, the map $s \mapsto \varphi(s) \bf{1}_{s < t}$ is $\L(Y)$-a.s.\@ continuous, so by the continuous mapping theorem it holds that $I^{\varphi}_t(\ell^n) = \ev[\varphi(Y^n) \bf{1}_{Y^n < t}] \to \ev[\varphi(Y) \bf{1}_{Y < t}] = I^{\varphi}_t(\ell)$.
\end{proof}

\section*{Acknowledgement}
This research has been supported by the EPSRC Centre for Doctoral Training in Mathematics of Random Systems: Analysis, Modelling and Simulation (EP/S023925/1). The authors are grateful to the anonymous referees for their very careful reading of the manuscript and their insightful comments, which considerably improved the article. PJ thanks Alda\"{i}r Petronilia for discussions on this material.

\printbibliography

\end{document}